\documentclass[11pt,reqno]{amsart}
\usepackage[utf8]{inputenc}
\usepackage{amssymb, amsthm}

\usepackage{mathrsfs}
\usepackage{bm,bbm}
\usepackage{stmaryrd}
\usepackage{euscript}
\usepackage{microtype}
\usepackage{comment} 
\usepackage{xcolor}
\usepackage{enumitem}
\usepackage{cases}
\usepackage{caption}
\usepackage{svg}

\usepackage[margin=1.35in]{geometry}
\usepackage{amsfonts}
\DeclareMathAlphabet{\mathpgoth}{OT1}{pgoth}{m}{n}
\DeclareMathAlphabet{\mathesstixfrak}{U}{esstixfrak}{m}{n}
\DeclareMathAlphabet{\mathboondoxfrak}{U}{BOONDOX-frak}{m}{n}
\usepackage{amsmath,amssymb}
\usepackage{dsfont}

\usepackage{yfonts,mathtools}
\usepackage{tikz-cd}
\usepackage[all,cmtip]{xy}
\numberwithin{equation}{section}
\usepackage{soul}

\usepackage{hyperref}
\hypersetup{colorlinks}
\definecolor{darkred}{rgb}{0.5,0,0}
\definecolor{darkgreen}{rgb}{0,0.5,0}
\definecolor{darkblue}{rgb}{0,0.3,0.8}
\hypersetup{colorlinks, linkcolor=black, filecolor=darkgreen, urlcolor=darkred, citecolor=darkblue}
\makeatletter % `@' now normal "letter"
\@addtoreset{equation}{section}
\makeatother  % `@' is restored as "non-letter"

\numberwithin{equation}{section}

\newtheorem{thm}{Theorem}[section] 
\newtheorem{cor}[thm]{Corollary}
\newtheorem{thmd}[thm]{Theorem-Definition}
\newtheorem{lemd}[thm]{Lemma-Definition}

\newtheorem{conj}[thm]{Conjecture}
\newtheorem{prop}[thm]{Proposition}
 
\newtheorem{lemma}[thm]{Lemma}
\theoremstyle{definition}
\newtheorem{defn}[thm]{Definition}

\newtheorem{convention}[thm]{Convention}
\theoremstyle{remark}
\newtheorem{rem}[thm]{Remark}

\newtheorem{example}[thm]{Example}

\usepackage{pgffor}
\foreach \x in {A,...,Z}{%
\expandafter\xdef\csname c\x\endcsname{\noexpand\ensuremath{\noexpand\mathcal{\x}}}
}

\foreach \x in {a,...,z}{%
\expandafter\xdef\csname frak\x\endcsname{\noexpand\ensuremath{\noexpand\mathfrak{\x}}}
}

\foreach \x in {A,...,Z}{%
\expandafter\xdef\csname b\x\endcsname{\noexpand\ensuremath{\noexpand\mathbb{\x}}}
}

\newcommand{\sq}{/\!\!/}
\newcommand{\ob}{\operatorname{Ob}}

\newcommand{\op}{\operatorname{op}}

\newcommand{\std}{\mathrm{std}}
\newcommand{\pr}{\operatorname{pr}}

\newcommand{\flow}{\psi^s_{Z}}
\newcommand{\pX}{\partial_{\infty} X}

\newcommand{\uhclim}[1]{\underset{#1}{\operatorname{hocolim}}\hspace{1.5pt}}
\newcommand{\uhlim}[1]{\underset{#1}{\operatorname{holim}}\hspace{1.5pt}}
\newcommand{\crit}{\operatorname{Crit}}

\newcommand{\Crit}{\operatorname{Crit}}
\newcommand{\norm}[1]{\left\lVert#1\right\rVert}

\newcommand{\hor}{\operatorname{hor}}
\newcommand{\ver}{\operatorname{vert}}
\newcommand{\moduli}{\widetilde{\cM}}

\newcommand{\Gr}{\operatorname{Gr}} 
\newcommand{\diag}{\operatorname{diag}}  
\newcommand{\Tr}{\operatorname{Tr}}

\newcommand{\Z}{\mathbb{Z}}

\newcommand{\R}{\mathbb{R}}
\newcommand{\C}{\mathbb{C}}

\newcommand{\pt}{\mathrm{pt}}
\newcommand{\Cone}{\mathrm{Cone}}

\newcommand{\Aut}{\mathrm{Aut}}

\title[Equivariant partially wrapped Fukaya categories]{Equivariant Partially Wrapped Fukaya Categories on Liouville Sectors}
\author[Choa]{Dongwook Choa}
\address{Department of Mathematics\\ Chungbuk National University}
\email{dwchoa@chungbuk.ac.kr}
\author[Hu]{Jiawei Hu}
\address{Department of Mathematics\\ Boston University}
\email{jiaweihu@bu.edu}
\author[Lau]{Siu-Cheong Lau}
\address{Department of Mathematics\\ Boston University}
\email{lau@math.bu.edu}
\author[Li]{Yan-Lung Leon Li}
\address{Center of Geometry and Physics, Institute for Basic Science (IBS), Pohang 37673, Korea}
\email{ylli@ibs.re.kr}

\date{\today}

\begin{document}

\begin{abstract}
We develop an equivariant Lagrangian Floer theory for Liouville sectors that have symmetry of a Lie group $G$.   Moreover, for Liouville manifolds with $G$-symmetry, we develop a correspondence theory to relate the equivariant Lagrangian Floer cohomology upstairs and Lagrangian Floer cohomology of its quotient.   Furthermore, we study the symplectic quotient in the presence of nodal type singularities and prove that the equivariant correspondence gives an isomorphism on cohomologies which was conjectured by Lekili-Segal.
\end{abstract}
\maketitle    

{ \hypersetup{hidelinks} \tableofcontents }

\section{Introduction}
We introduce an equivariant wrapped Fukaya category  for Liouville manifolds (or sectors) equipped with a Liouville action by a compact Lie group $G$. This construction is carried out through a variant of the Borel construction in which  the homotopy quotient is modeled by approximations of a sequence of Liouville manifolds (sectors). Within this framework, we develop a Lagrangian correspondence theory which,  in the case of a free action, generates a fully faithful $A_{\infty}$-functor from a full subcategory of  the equivariant wrapped Fukaya category to the ordinary wrapped Fukaya category of the exact symplectic quotient. 

Let $X$ be a symplectic manifold whose mirror is denoted by $X^\vee$. Teleman conjectured that a Hamiltonian $G$ action on $X$ is mirror to a holomorphic fibration 
$$ F: X^\vee \to G^\vee_\mathbb{C}/\mathrm{Ad} $$
where $G^\vee_\mathbb{C}$ is the complexified Langlands dual group, and $G^\vee_\mathbb{C}/\mathrm{Ad}$ denotes the space of conjugacy classes, such that 
the mirror of a symplectic quotient $X \sq_{\xi}\, G$ at a regular moment-map level $\xi\in\mathfrak{g}^*$ is given by a fiber $X^\vee_{\phi(\xi)} = F^{-1}\{\phi(\xi)\}$ for some $\phi(\xi) \in G^\vee_\mathbb{C}/\mathrm{Ad}$. Equivariant quantum cohomology in this setup was studied in \cite{PT1, GMP1, GMP2} and equivariant open-string mirror symmetry was studied in \cite{KLZ23,lauleungli}.

In this paper, we study the equivariant partially wrapped Fukaya category of a Liouville sector with a Hamiltonian $G$-symmetry. The theory of partially wrapped Fukaya categories for Liouville sectors developed by \cite{GPS1,GPS3} enjoys nice functorial properties and provides a powerful tool of sectorial decompositions to study wrapped Fukaya categories. 
One important motivation is the Fukaya-Seidel category of a Landau-Ginzburg model $(Y,W)$, which can be understood as the partially wrapped theory for the corresponding Liouville sector obtained by taking a page of superpotential as Liouville hypersurface. Landau-Ginzburg models appear as the mirrors of Fano or general-type varieties and their Fukaya-Seidel categories reflect the derived categories of coherent sheaves of these varieties. 

For simplicity, let us consider $G = \bS^1$ (while the general theory is developed for a general compact Lie group $G$). By homological mirror symmetry \cite{Kont-HMS}, the conjecture can be reformulated as follows.
\begin{conj} \label{Teleman}
  Let $X$ be a Liouville sector equipped with a Hamiltonian $\bS^1$-action. Suppose that its symplectic quotient $\underline{X}$ at the zero moment level is smooth and is itself a Liouville sector. Then there exists a $\C[s]$-module structure on the Hochschild cohomology $HH^*(\cW(X))$ such that $HH^*(\cW(\underline{X}))$ is isomorphic to the quotient of the Hochschild cohomology of the partially wrapped Fukaya category of $\bS^1$-invariant objects in $X$ by the ideal generated by $(s-1)$.
\end{conj} 

We develop the theory of equivariant wrapped Fukaya category and correspondence to tackle this conjecture. The theory of Lagrangian correspondence was developed by \cite{MWW, Fukaya2023} and \cite{gao2017wrapped,Gaothesis} in the wrapped setting. In \cite{lauleungli}, we constructed equivariant Lagrangian correspondence for symplectic quotients in the compact setup using the machinery of \cite{Fukaya2023} where Kuranishi perturbations were essential. In the current exact setting, we can use geometric Hamiltonian perturbations in place of Kuranishi perturbation techniques. Moreover, the Lagrangian correspondence is tautologically unobstructed since it does not bound any non-constant holomorphic disk, and we can construct the cyclic element without using boundary deformation of Lagrangian submanifolds.

\subsection{Equivariant wrapped Fukaya category via Borel construction}
The first step of our construction is to define an \emph{equivariant wrapped Fukaya category} $\cW_G(X)$ associated to a Liouville manifold (or sector) $X$, which carries an action of a compact Lie group $G$ preserving the Liouville form. Throughout, we always assume that $X$ is of \emph{finite type} (i.e.\ with compact skeleton). The objects of $\cW_G(X)$ are embedded, exact, $G$-invariant Lagrangian submanifolds that are conical near infinity. The morphism complexes are defined as homotopy colimits of the corresponding equivariant Floer complexes, taken along the flow of a $G$-invariant wrapping Hamiltonian.

Classically, $G$-equivariant cohomology may be viewed as a functor $H_G^\bullet(-)$ on the category of $G$-spaces. One way to construct it is via the \emph{Borel construction}:
\begin{align*}
    H_G^\bullet(X) \coloneqq H^\bullet(X_G), \quad 
    X_G \coloneqq X \times_G EG,    
\end{align*}
where $EG$ is a contractible space on which $G$ acts freely. The associated quotient $X_G$ plays a double role:
\begin{enumerate}
    \item a locally trivial fibration $X \to X_G \to BG$;
    \item a (typically non-locally trivial) fibration 
    $
    B\operatorname{Stab}(x) \to X_G \to [x] \in X/G,
    $
    where $B\operatorname{Stab}(x)$ is the classifying space of the stabilizer of $x$.
\end{enumerate}
When the $G$-action is free, these identifications immediately yield $H_G^\bullet(X) \cong H^\bullet(X/G).$

There is a substantial body of work extending this perspective to Floer theory in various settings; see Section \ref{sec:relevant_works} for a brief summary. 
A complication in our setting is that we must work with \emph{families of Floer equations} parametrized by the infinite-dimensional Hilbert manifold $BG$.  
In infinite dimensions, many familiar tools from Morse and Floer theory break down---one reason is that  closed balls in Hilbert space are never compact.  To avoid these analytic difficulties, we adopt a \emph{finite-dimensional approximation} approach and prove that Floer theory behaves compatibly along these approximations.

A more subtle issue arises when considering families of Floer equations parametrized by Morse trajectories on $BG$.  For a general choice of model $EG \to BG$ and an arbitrary connection on this principal bundle, fiberwise exact Lagrangians need not be preserved under boundary parallel transport. 

To address this, we construct $X_G$ as the colimit of a sequence of Liouville manifolds (or sectors)
\begin{align*}
X(1) \hookrightarrow X(2) \hookrightarrow \cdots \hookrightarrow X(N) \hookrightarrow \cdots,
\end{align*}
each equipped with the structure of a Liouville fibration. 
This setup allows us to reduce the equivariant theory to Floer theory in a symplectic (Liouville) fibration, a setting that has already been extensively studied---most notably in the context of Lefschetz fibrations. Within such framework, we construct the equivariant wrapped Fukaya category $\cW_G(X)$. 

A fundamental structural result for (non-wrapped) equivariant Floer cohomology is that it satisfies an analogue of the Piunikhin-Salamon-Schwarz (PSS) isomorphism \cite{PSS96}.
\begin{thm}[Theorem \ref{thm:G-equivariant_PSS_isomorphism}]
    For an embedded $G$-invariant exact cylindrical Lagrangian,  there is an isomorphism 
    \begin{align*}
        H_{G}^{\bullet}(L;\mathbb{K}) \cong HF_{G}^{\bullet}(L,L)
    \end{align*}
    which we call the \textit{$G$-equivariant PSS isomorphism.}  
    \end{thm}
    In particular, when the $G$-action on $L$ is free, the  isomorphism $H_{G}^{\bullet}(L)\cong H^{\bullet}(L/G)$ combined with the usual PSS isomorphism implies that the equivariant Floer cohomology reduces to ordinary Floer cohomology on the quotient  $$HF_{G}^{\bullet}(L, L) \cong        HF^{\bullet}(L/G, L/G) .$$
    To extend this correspondence to pairs of cylindrical Lagrangians with wrapping, and to make the construction functorial, we will employ the technique of Lagrangian correspondences in the wrapped setting.

\subsection{Equivariant Lagrangian correspondence}
Assume that the $G$-action on the zero level set of the moment map is free, so that the symplectic quotient $X\sq_0 \,G \coloneqq \mu^{-1}(0)/G$ is smooth.\footnote[1]{Throughout the paper, we will simply write $X \sq G$ for $X \sq_{0}\, G$, since we always work at the zero moment level.} Since the homotopy quotient $X_{G}$ is modeled as the colimit of a sequence of Liouville manifolds, using finite-dimensional approximations, we study the Lagrangian correspondence
\begin{align*}
    L_{G}^{\pi} \subset X_{G}^{-}\times X\sq G.
\end{align*}
 This will produce an $A_{\infty}$-functor relating the equivariant wrapped Fukaya category $\cW_{G}(X)$ and the (ordinary) wrapped Fukaya category  $\cW(X\sq G)$.

Our approach follows the strategy of Fukaya \cite{Fukaya2023}, which constructs an $A_{\infty}$-bimodule equipped with a distinguished \emph{cyclic element}. The key ingredient is an extension to the Liouville setting of a similar construction previously developed in the compact case \cite{lauleungli}. This produces the following $A_{\infty}$-functor.

\begin{thm}[Theorem \ref{prop: GoingDown}]
    Let $X$ be a Liouville manifold equipped with a free action of a compact Lie group $G$ preserving the Liouville form. Let $\cW_{G}(X)_{0}$ denote the full subcategory of $\cW_{G}(X)$ generated by $G$-invariant Lagrangians contained in $\mu^{-1}(0)$. Then there exists a fully faithful $A_{\infty}$-functor 
\begin{align*}
    \Psi_{\downarrow}: \cW_{G}(X)_{0} \rightarrow \cW(X\sq G), 
\end{align*} 
which sends a $G$-invariant exact Lagrangian $L\subset X$ to its quotient $L/G$. Therefore, $\cW(X\sq G)$ is quasi-equivalent to the full subcategory $\cW_{G}(X)_{0}$.
\end{thm}
The same result extends to Liouville sectors, provided that the Floer data are chosen carefully so that the target space has bounded geometry. We do not address these technicalities here, and refer the reader to \cite{GPS1,GPS3} for further details.

For a path of Hamiltonian diffeomorphisms on a symplectic manifold $X$, one has the open Seidel morphism on Lagrangian Floer cohomologies defined by \cite{Hu-Lalonde, Hu_Lalonde_Leclercq, Charette_Cornea}. 
In the case of a Hamiltonian $\bS^1$-action, the closed-open image $CO(S)$ of the Seidel element \cite{Seidel97} gives an element $s\in HH^*(\cW(X))$, which is obtained by counting pseudoholomorphic polygons with an interior insertion by the Seidel element.  For invariant Lagrangians, $CO(S)$ is closely related to the equivariant disk potential \cite{KLZ23}, which takes a very simple form for a single exact Lagrangian submanifold since it does not bound any non-constant pseudoholomorphic disk. This will be useful to construct the conjectural isomorphism between the quotient of $s-1$ and $HH^*(\cW(\underline{X}))$ in Conjecture \ref{Teleman}.

\subsection{Singular symplectic quotients}
However, in the exact setup where the moment correspondence $L^\pi$ is required to be exact, the action on $L^\pi$ is not free even in simple examples.

\begin{example}
Let us consider the Hamiltonian $\bS^1$-action with weights $(1,-1)$ on $\C^2$. The action is free on all moment levels $\xi \not=0$, and it has the fixed point $0$ on the exact level $\xi =0$. The quotient $\C^2\sq_0\, \bS^1$ has a singular symplectic structure. 

On non-exact levels $\xi \not=0$, the equivariant Lagrangian correspondence $L^\pi_{\bS^1}$ is obstructed due to  contributions of holomorphic disks contained in the $x$- and $y$-axes of $\C^2$.  This obstruction leads to a discrepancy between the disk potential $W_{\C^2} = T^A (z+T^\xi w)$ for standard Lagrangian tori in $\C^2$, when restricted to the Teleman's fiber $zw^{-1}=1$, and the disk potential $W_{\C} = T^A z$ of the quotient. To eliminate this obstruction, it is necessary to introduce a bulk deformation of $\bC$ by $\mathbf{b} = \log (1+T^\xi) [\pt]$  so that both potentials become $T^A (1+T^\xi)z$. See \cite[Example 1.4]{lauleungli} for a closely related example. If we use a non-trivial spin structure for $L^\pi$, the sign is changed and we obtain the potential $T^A (1-T^\xi)z$ (with $\mathbf{b}=\log (1-T^\xi) [\pt]$). With this bulk deformation $\mathbf{b}$, the $\bS^1$-equivariant Fukaya category of $\C^2$ at the moment level $\xi$ corresponds to the Fukaya category of $(\C,\mathbf{b})$.

As $\xi \to 0$, the required bulk deformation blows up to $-\infty$. This leads to a pinching of $\C$ at the origin, where $\C^*$ has zero disk potential. In such a situation, we expect that the equivariant Lagrangian correspondence at the limiting exact level $\xi \to 0$ would give a quasi-equivalence between the $\bS^1$-equivariant wrapped category of $\C^2$ with the wrapped category of $\C^*$. In the more general case of $X$ being the total space of conic fibrations, this was conjectured by Lekili-Segal \cite{lekili2023equivariant}.
\end{example}

Under symplectic reduction of a symplectic manifold $X$ at a singular value of the moment map, the quotient $\underline{X}$ is in general a singular symplectic space. One may remove the singular locus to obtain a smooth open symplectic manifold $\underline{X}^{\circ}$. In certain cases, one can equip it with a suitable Liouville structure, and study its wrapped Fukaya category. 

In a series of examples, Lekili-Segal \cite{lekili2023equivariant} observed an intersecting phenomenon that the equivariant Lagrangian Floer cohomology of a Lagrangian $L \subset X$ is isomorphic to the wrapped Floer cohomology of the corresponding Lagrangian $\underline{L}^{\circ} \subset \underline{X}^{\circ}$. They conjectured that this correspondence holds in general. It is not known whether a Lagrangian in $\underline{X}^{\circ}$ admits a compactification to a smooth Lagrangian in $X$ in full generality. Thus, we impose a suitable local model for such compactifications. Under this assumption, and using the machinery developed in this paper, we prove the following result.

\begin{thm}[see Theorem \ref{thm: S1case} for the precise statement]
Let $X$ be a symplectic manifold with a Hamiltonian $\mathbb{S}^1$-action, and let $\mu: X \to \mathbb{R}$ denote the moment map. Consider the symplectic quotient $\underline{X} = \mu^{-1}(0)/\mathbb{S}^1$ and let $\underline{X}^\circ$ be the complement of the singular locus of $\mu^{-1}(0)$. Assume that the isotropy group of each point in $\mu^{-1}(0)$ is either trivial or the whole $\mathbb{S}^1$, and that only nodal singularities occur. Then, for invariant Lagrangians in $\mu^{-1}(0)\subset X$ which are of disk type near the singularities, the equivariant Floer cohomology is isomorphic to the wrapped Floer cohomology of the quotient in $\underline{X}^\circ$.
\end{thm}

\subsection{Relevant works in literature}
\label{sec:relevant_works}
Various equivariant extensions of Lagrangian Floer theory have been studied in the literature since the work of Khovanov and Seidel \cite{KS-quiver}. We refer to \cite[Section 1]{lauleungli} and the references therein for a summary, and to \cite{KKK} for another formulation of $\mathbb{Z}_2$-equivariant wrapped Floer theory. In this paper, we focus on the equivariant Floer theory for a Lie group $G$.

In a range of classes of geometries, substantial evidence has been found linking equivariant Floer theory to singular symplectic quotients; see, for instance \cite{lauleungli, LLLM, MSZ, AsplundLi, XL}. Lekili and Segal formulate several conjectures in this direction. The conjectures most relevant to our work are Conjecture A, D and E in \cite{lekili2023equivariant}. 

In this paper, we consider a general Liouville manifold $X$ with a Liouville $\bS^1$-action whose quotient $X \sq\, \bS^1$ has at worst nodal type singularities. We establish an isomorphism between the equivariant Floer cohomology of $X$ and wrapped Floer cohomology of the quotient $X \sq\, \bS^1$. 

For an additive hypertoric variety $X$,  Lee, Li, Liu and Mak \cite{LLLM} used sectorial descent to compute the partially wrapped Fukaya category of its hyperplane arrangement in terms of its equivariant convolution algebra, which is (conjecturally) related to its equivariant partially wrapped Fukaya category;
for a multiplicative hypertoric variety $X$, using microlocal sheaf theory, McBreen, Shende and Zhou \cite{MSZ} proved a ``Fukaya category commutes with reduction" theorem in terms of module action on categories.

\subsection{Acknowledgment}
The work of the first and the fourth authors was supported by the Institute for Basic Science (IBS-R003-D1). The second author would like to thank Denis Auroux and Daniel Pomerleano for helpful discussions. The third and fourth authors express their gratitude to Naichung Conan Leung for great discussions in The Chinese University of Hong Kong and sharing his insights on 3D mirror symmetry and Teleman's conjecture.

\subsection*{Notations and  conventions} 
For the convenience of the reader, we summarize some notations and conventions used throughout this paper. They are ordered according to their first appearance in the main text.
\begin{itemize}
    \item{\textbf{Lie groups and actions.}}  All the Lie groups are assumed to be compact and  connected. When a Lie  group $G$ acts on a manifold $X$, we denote the action by an element $g$ on a point $x\in X$ by $\tau_{g}(x)$, or $g\cdot x$ for short. The vector field on  $X$ induced by a Lie algebra element $ \xi \in \frakg$ is denoted by $X_{\xi}$. 
    \item{\textbf{Hamiltonian vector fields.}} The sign convention for a Hamiltonian vector field is $$\omega(X_{H}, \cdot) = -dH.$$ 
    \item{\textbf{Morse and Floer theory conventions.}} We adopt the ``back-propagation'' convention. For a Riemann disk, the negative puncture corresponds to the output of the $\mu^{k}$ operations, while the positive punctures correspond to the inputs.  
    The similar convention applies to Floer strips, as well as to Morse trajectories and Morse trees.     
    \item{\textbf{Wrapped Floer complexes.}} In contrast to \cite{GPS1}, where the wrapped Floer complex is obtained as the colimit
    \begin{align*}
        CF^{\bullet}(L^0, K) \rightarrow  CF^{\bullet}(L^1, K)\rightarrow \cdots,
    \end{align*}
    with respect to the inverse system $\cdots \rightarrow L^1\rightarrow L^0$, our wrapped Floer complex  is defined as the colimit 
    \begin{align*}
        CF^{\bullet}(L, K^0)\rightarrow CF^{\bullet}(L, K^1)\rightarrow \cdots  
    \end{align*}
    with respect to the direct system $K^{0} \rightarrow K^{1}\rightarrow \cdots$. The reason for this choice is that, in our situation, it is often more convenient to consider pseudo-holomorphic strips $u$ bounded by $L_{0}$ and $L_{1}$ with a Hamiltonian perturbation term. Such a strip corresponds bijectively to an honest holomorphic strip by setting: $\tilde{u}\coloneqq (\phi^{t}_{H})^{-1}(u)$, with boundary conditions $L_{0}$ and $(\phi^{1}_{H})^{-1}(L_{1})$. 
    \item{\textbf{Categories and $A_{\infty}$-categories.} } Unless stated otherwise, all categories are assumed to be small. Our $A_{\infty}$-composition follows the same ``forward-composition'' convention as in \cite{GPS1,GPS3}.  
  The $A_{\infty}$-operations take the form
    \begin{align*}
        \mu^d: \cC(X_{0},X_{1})\otimes \cdots \otimes \cC(X_{d-1}, X_{d}) \rightarrow \cC(X_{0},X_{d})[2-d]. 
    \end{align*}
    \item{\textbf{$A_{\infty}$-modules.}}    The conventions for $A_{\infty}$ left modules, right modules, and bimodules are defined accordingly with respect to the $A_{\infty}$-operations specified above; see \cite{GPS1} for details. For an $A_{\infty}$-category $\cC$, we denote by $\mathrm{Mod}^{\cC}$ the dg-category of \emph{left} $\cC$-modules, and by $\mathrm{Mod}_{\cC}$ the dg-category of \emph{right} $\cC$-modules. 
\end{itemize}
\section{Group actions on Liouville manifolds or sectors}

\subsection{Liouville manifolds}
References for Liouville manifolds and the relations with contact manifolds include \cite{Yakovgromov91,Sei08,GPS1}. To fix notation and convention, we briefly recall several standard definitions.

A vector filed $Z$ on a symplectic manifold $(X,\omega)$ of dimension $2n$ is called \textit{Liouville} if $\cL_{Z}\omega = \omega$. By Cartan's formula, this is equivalent to  $\omega = d\lambda$ where $\lambda\coloneqq \iota_{Z}\omega$ is the associated \textit{Liouville form}. The flow $\psi_{Z}^s$ of $Z$ exponentially expands both $\lambda$ and $\omega$, i.e. 
\begin{align*}
    (\psi_{Z}^{s})^* \lambda = e^s \lambda, \quad  (\psi_{Z}^{s})^* \omega = e^s \omega.
\end{align*} 

A (co-oriented) \textit{contact structure}  on a manifold $Y$ of dimension $2n - 1$ is a distribution $\cD$ on $Y$ given globally as the kernel of a $1$-form $\alpha$ such that $\alpha\wedge (d\alpha)^{n-1} \neq 0 $. The form $\alpha$ is called a \textit{contact form}. Whenever convenient we will refer to a pair $(Y, \alpha)$ or to a pair $(Y, \cD)$ as a contact manifold. The \textit{Reeb vector field}  is defined by the equations $\alpha(R_{\alpha})=1$ and $d\alpha(R_{\alpha},\cdot)= 0$.

A \textit{Liouville domain} $X_{0}$ is an exact symplectic manifold with boundary whose Liouville vector field is transverse to the boundary and outward pointing (often referred to as \emph{convexity} of the boundary); the restriction of $\lambda$ to the boundary of $X_{0}$ is a contact form $\alpha = \lambda|_{\partial X_{0}}$.

A \textit{Liouville manifold} is an exact symplectic manifold which is cylindrical at infinity:  there exists a Liouville domain $X_0 \subset X$ such that $X$ can be identified with $X_{0}$ attaching to an infinite cone: 
$$X=  X_{0} \cup_{\partial X_{0 }} \left(\partial X_{0}\times \mathbb{R}_{\geq 0}\right) $$
where $\left(\partial X_{0}\times \mathbb{R}_{\geq 0}\right)$ is equipped with the symplectic form $d(e^s\alpha )$;  the symplectomorphism is induced by the flow of $Z$ for time $s$. The infinite cone $\left(\partial X_{0}\times \mathbb{R}_{\geq 0}, d(e^{s}\alpha)\right)$ is called the \textit{positive half symplectization} of the contact manifold $(\partial X_{0}, \alpha)$, which we will denote by $S_{\geq 0}\partial X_{0}$. In general, for a contact manifold $(Y,\alpha)$ where $\alpha>0$, one may construct a symplectic manifold $(SY,d(e^s \alpha))$ in a similar manner by replacing $\bR_{\geq 0}$ with $\bR$, which we refer to as the \textit{symplectization of $(\alpha,Y)$}. The Liouville vector field in above local coordinates is $Z= \frac{\partial }{\partial s}$. We often use coordinate transformation $r=  e^s$.

The image of $\partial X_{0}$ under the Liouville flow is a family of contact manifolds parametrized by $s$, which are contactomorphic to each other. Thus there is a well-defined notion of ``contact-type boundary at infinity'', which we denote it by $\partial_{\infty} X$. There is a natural embedding of the symplectization of $\partial_{\infty} X$ into $X$. We  assume $\partial_{\infty} X$ is compact in general. However, the issues emanating from the non-compactness of $\partial_{\infty} X$ on defining wrapped Fukaya category can be neutralized by taking a colimit of the compact cases over a homology hypercover (see \cite[Section 1.3]{GPS1}). 

An object invariant under the Liouville flow near infinity is called \textit{cylindrical}; the infinite cone defined above is also called a \textit{cylindrical end}. 

Between Liouville manifolds,  a \textit{Liouville embedding} is an embedding $(X,\lambda ) \xhookrightarrow{\phi} (X',\lambda')$ such that 
\begin{equation}
    \label{eq:liouville map}
    \begin{aligned}
      &\phi^*\lambda'=\lambda+df \quad &&\text{for some compactly supported $f$,}\\
      &\phi_* Z_{X}= Z_{X'}  &&\text{outside a compact set.}
    \end{aligned}
    \end{equation} Here, the second condition is  redundant  if the codimension of $\phi$ is zero.

\begin{example}
The Euclidean space $\bC^n$ with its standard symplectic structure $\omega_{\std}=\sum_{i} dx_i\wedge dy_{j}$ is a Liouville manifold. The Liouville vector field in local coordinates is 
\begin{align*}
    Z= \frac{1}{2}\sum_{i}\left(x_{i}\frac{\partial }{\partial x_{i}}+y_{i}\frac{\partial }{\partial y_{i}}\right). 
\end{align*}  
\end{example}

\begin{example}
The cotangent bundle $(T^*Q,\omega= - d\lambda_{\mathrm{can}})$ of a compact manifold $Q$ is a Liouville manifold, whose corresponding Liouville vector field $Z$ is given by radial dilation in each fiber. Any tubular neighborhoods of the zero section with boundary transverse to $Z$ can be chosen to be $X_{0}$. A canonical choice is to equip $Q$ with a Riemannian metric, then  choose $X_{0}$ as the unit codisk bundle.
\end{example}

\subsection{Liouville actions}
We now briefly recall the notions of symplectic and contact reductions. 
\begin{defn}[Symplectic reduction]
    Let $(X,\omega)$ be a  symplectic manifold with a proper Hamiltonian action of a Lie group $G$, with the corresponding equivariant moment map $\mu: X\rightarrow \frakg^*$ such that $$d\langle \mu, \xi\rangle = -\iota_{{X_{\xi}}}\omega $$ for $\xi \in \frakg$. The symplectic reduction is
\begin{align*}
    X\sq G \coloneqq  \mu^{-1}(0)/G.
\end{align*}
\end{defn}
If  $G$ acts freely on $\mu^{-1}(0)$ (and hence $0$ is a regular value of $\mu$), then $X\sq G$ is a smooth symplectic manifold.

Recall that a \textit{contact action} is a group action preserving the contact structure $\cD$. Regardless of whether the Lie group is compact, one can always find a contact form $\alpha$ so that $\ker \alpha = \mathcal{D}$. This assertion follows by averaging the 1-form over some open covers induced by the contact action. For more details, see, for example, \cite{Le}.

\begin{defn}[Contact reduction]
    Let $(Y,\alpha )$ be a contact manifold with a proper contact action of a Lie group  $G$. The \textit{contact moment map} is defined by 
    \begin{equation}
        \label{equ:contact_moment_map}
        \mu:Y\rightarrow \mathfrak{g}^*:\langle \mu, \xi \rangle = \alpha(X_{\xi})
    \end{equation}
    for $\xi \in \frakg$. The contact reduction is defined by 
    \begin{align*}
        Y \sq G \coloneqq  \mu^{-1}(0)/G.  
    \end{align*}
\end{defn}

Similarly to the symplectic case, if $G$ acts freely on $\mu^{-1}(0)$ and $0$ is a regular value of the corresponding moment map $\mu$, the contact reduction is a smooth contact manifold with a contact form descended from  $\alpha|_{\mu^{-1}(0)}$. 

See \cite{Al,Gh} for more details (cf. \cite{GS} for a similar result but from a different perspective).

Let $(Y,\alpha)$ be a contact manifold with a contact action of a Lie group $G$, with moment map $\mu_{Y}$. There is a natural way to extend it to  a Hamiltonian action on its symplectization $SY$, by defining the moment map $\mu_{SY}: SY\rightarrow \frakg^*$ according to the formula: 
\begin{equation}
    \label{equ:extension}
    \mu_{SY}(y,s) = e^s \mu_{Y}(y)  
\end{equation}
 for all $(y,s)\in Y\times \bR$. The corresponding action is given by 
 \begin{align*}
    g\cdot (y,s) = (g\cdot y, s).
 \end{align*}
 In particular, we have $\cL_{Z}X_{\xi} = 0$ for all $\xi \in \frakg^*$. 
 %Since the Liouville vector field is transverse to $\partial X_{0}$, there is a \textit{collar neighborhood} of $\partial X_{0}$ of the form
%\begin{align*}
%\partial X_{0}\times (-\epsilon,\epsilon) \rightarrow X.
%\end{align*} 

In order for the reduction of a Liouville manifold to remain a Liouville manifold, we introduce the following class of group actions.
\begin{defn}
\label{defn:Liouville_action}
Let $(X, \lambda)$ be a Liouville manifold with a proper Hamiltonian group action by $G$ with moment map $\mu_{X}$. We say this action is a \textit{pre-Liouville action} if it  satisfies the following two equivalent conditions:
\begin{enumerate}
    \item Near infinity, $\cL_{X_{\xi}}\lambda = 0$  and \begin{align}
        \label{equ: equal_sign}
        \langle \mu_{X},\xi \rangle = \lambda (X_{\xi}).  
    \end{align}
    \item  There exists $\partial_{\infty} X \subset X$ such that the restriction of $G$-action to $\partial_{\infty} X$ is a contact action. The moment map $\mu_{X}$ satisfies 
    \begin{align}
        \label{equ:equ_identity}
        \mu_{X} =  e^s \mu_{\partial_{\infty} X}    
    \end{align}
    on $S_{\geq 0}\partial_{\infty} X$, where $\mu_{\partial_{\infty}X}$ is defined by Equation \eqref{equ:contact_moment_map}.
\end{enumerate} 
\end{defn}
Here we view $\partial_{\infty} X$ as a contact hypersurface near infinity. We refer to $\mu_{X}|_{\pX}$ as the restriction of $\mu_{X}$ to $\pX$, and refer to $\mu_{\pX}$ as the contact moment map defined by the induced contact action, which are actually equal by  Equation \eqref{equ: equal_sign}.

\begin{proof}[Proof of equivalence in  Definition \ref{defn:Liouville_action}]
To show (1) implies (2), we observe that since $(\psi^{s}_{Z})^* \lambda  = e^s \lambda$, the leaves of $G$-orbits lie in the level sets of $\flow$.  So WLOG, we may define $\pX = \psi^{-1}(0)$. Then the restriction of $G$-action to $\pX$ is a contact action. Taking derivative of the following identity with respect to $Z$, 
\begin{align*}
    \langle \mu_{X},\xi \rangle = \lambda (X_{\xi}), 
\end{align*} 
we have $Z \mu_{X} = \mu_{X}.$ Combing the fact $\mu_{X}|_{\pX} = \mu_{\pX}$, we obtain the desired result. 

Conversely, let $\alpha = \lambda|_{\pX}$, then $\lambda  = e^s \alpha$. Take derivative with respect to $Z$, we have $\cL_{X_{\xi}}\lambda = 0$. Equation \eqref{equ: equal_sign} follows from combining  $\lambda  = e^s \alpha$ and  Equation \eqref{equ:equ_identity}.
\end{proof}

\begin{defn}
\label{defn:real_Liouville_action}
We say a proper group action $G$ on $(X,\lambda)$ is \textit{a Liouville action} if it preserves the Liouville form $\lambda$. 
\end{defn}
The following proposition is well-established. For the sake of clarity, we present it here for future reference. 
\begin{prop}
    \label{prop:Group_action_1_form}
    A Liouville action is a pre-Liouville action. 
\end{prop}
\begin{proof}
  We define the moment map $\mu: X \rightarrow \frakg^*
    $ by: 
   \begin{align}
   \label{equ:moment_map}
       \langle \mu(x),\xi  \rangle = \lambda(X_{\xi}) \text{ for all $x\in X$ and $\xi \in \frakg$.}
   \end{align}   
   One can check $\mu$ is an equivariant map. We  compute that $$d\langle \mu, \xi\rangle = \mathcal{L}_{X_{\xi}}\lambda - \iota_{X_{\xi}}d\lambda = - \iota_{X_{\xi}}\omega $$
    which shows it is a Hamiltonian action. It then follows from (1) of Definition \ref{defn:Liouville_action}.   
\end{proof}

\begin{thm}
\label{thm:perturbing_1_form}
Let $(X,\lambda)$ be a Liouville manifold with a pre-Liouville action by $G$. One can deform $\lambda$ in a compact region so that the $G$-action on $(X,\overline{\lambda})$ is a Liouville action. Moreover, we have $\overline{\lambda}=\lambda +df$ where $f$ is a compactly supported function. 
\end{thm}
\begin{proof}

Since  $G$ is compact, one can find a bi-invariant volume form  $dg$  such that $\int_{G}dg =1$. Let $X_{0}$ be a Liouville domain large enough such that the $G$-action preserves $\lambda|_{\partial X_{0}}$. We can average the primitive $1$-form in $X_{0}$ by the following transformation:
      $$\lambda\rightarrow \overline{\lambda} = \int_{G}\tau^*_{g}\lambda dg.$$
Since outside $\partial_{\infty} X$, the $1$-form is invariant under $G$-action, we have $\lambda = \overline{\lambda}$ outside a compact subset. This implies that $(X_{0},\overline{\lambda})$ is still a Liouville domain.

Since $\lambda$ and $\overline{\lambda}$ are in the same cohomology class, by our construction, we get $$\overline{\lambda}-\lambda = df,$$where $f$ is a compactly supported function. 
\end{proof}
Therefore, when discussing the wrapped Fukaya category of a Liouville manifold under a pre-Liouville action, the following homotopy argument shows that we can generally assume the group action preserves the Liouville form. 

\begin{cor}
    In the same setting as Theorem \ref{thm:perturbing_1_form}, there is a quasi-equivalence between the wrapped Fukaya categories $\cW(X,\lambda)\cong \cW(X,\overline{\lambda})$. 
\end{cor}
\begin{proof}
    By a Moser-type argument \cite[Proposition 11.8]{CE12}, there exists an isotopy $\phi_{t}:X\rightarrow X$ of Liouville isomorphisms that intertwines $\lambda$ and $\overline{\lambda}$. The wrapped Fukaya category is invariant under such deformations. 
\end{proof}

\begin{thm}
    \label{thm:Liouville_reduction}
        If  a group $G$ acts on a Liouville manifold $X$ in a (pre)-Liouville fashion, and suppose $G$ acts freely on $\mu^{-1}(0)$, then $X\sq G$ is a smooth Liouville manifold.
    \end{thm}
    \begin{proof}
    The proof is similar to the symplectic case. Since $\lambda$ vanishes in the direction of $G$-action by \eqref{equ:moment_map}, it descend to a Liouville form on $X\sq G$. To show $X\sq G$ is a Lioville manifold, we observe that for a contact manifold $(Y,\alpha)$, the symplectization commute with reduction, i.e. $(SY) \sq G \cong S(Y\sq G)$.  If the action is only pre-Liouville, one can apply Theorem \ref{thm:perturbing_1_form}.
    \end{proof}
We will refer to $X\sq G$ as the \textit{Liouville reduction of $X$}.   
\begin{rem}
The above theorem also shows that if we decompose $X$ into $X_{0} \cup_{\partial X_{0 }} \left(\partial X_{0}\times \mathbb{R}_{\geq 0}\right)$, the symplectic reduction will preserve the decomposition. In the case of $G=\bS^1$, this  coincides with the notion of \textit{contact cut} defined by \cite{Le}.
\end{rem}

\begin{example}
    \label{ex:closed_manifold}
    ($\dim Q>\dim G$) Let $G$ acts on a closed manifold $Q$, 
    \begin{align*}
        G\rightarrow \mathrm{Diff}(X), \quad g\rightarrow \tau_{g}.     
    \end{align*}
    It lifts to a Hamiltonian action on $T^*Q$ by $$g\cdot(q,p) = (\tau_{g} q, (\tau_{g^{-1}})^* p).$$ The fact it is a Liouville action follows from the observation $\lambda_{\mathrm{can}}$ is invariant under all automorphisms of $T^*Q$ arising from the automorphisms of the base, then apply Proposition \ref{prop:Group_action_1_form}. 
    
    The moment map of this action is 
    \begin{align}
    \label{equ:moment_map_for_cotangentbundle}
        \mu(q,p)(\xi) = p(X_{\xi}(q)).
    \end{align}
    In the free case, the cotangent lifted action on $T^*X$ is also free, hence $T^*X /\!\!/G$ is a Liouville manifold. It turns out it is canonically symplectomorphic to $T^*(X/G)$. An explicit symplectomorphism can be constructed as follows: From  Equation \eqref{equ:moment_map_for_cotangentbundle}, one has $$p\in \mu^{-1}(0)\cap T_{q}^*Q\Leftrightarrow p(X_{\xi}(q)) =0.$$ In other words, the zero level set of $\mu$ is the annihilator of the subbundle generated by the orbit of $G$. Now consider the map \begin{align*}\phi: TX/G & \rightarrow T(X/G) \\
    \phi([v_{x}])& \rightarrow D_{x}\pi (v_{x}) \end{align*} where $D_{x}\pi$ is the differential of the quotient map at $x$. It is easy to check that its dual is the desired symplectomorphism. 
\end{example}

\begin{thm}
Let $X$ and $Y$ be Liouville manifolds with Liouville $G$ actions. Then the diagonal action of $G$ on $X\times Y$ is also Liouville.    
\end{thm}
\begin{proof}
    The product $X\times Y$ is a Liouville manifold with Liouville form $\lambda_{X}+\lambda_{Y}$, which is invariant under the diagonal action. 
\end{proof}

\subsection{Stopped Liouville manifolds and Liouville sectors}
\label{sec:Group_action_on_Liouville_sectors}
A \textit{stopped Liouville manifold} $(X,\frakf)$ is Liouville manifold together with a closed subset $\frakf\subset \partial_{\infty} X$ called the \textit{stop}. 

This concept is closely related to the notion of a Liouville pair $(X,F)$ where $F$ is a $(2n-2)$-dimensional Liouville domain embedded in $\partial_{\infty} X$, called the \textit{Liouville hypersurface}, as discussed in \cite{Sylvan15,GPS1}. In this case, the stop is given as a closed neighborhood of $F$. We will often abuse the notation and identify the Liouville domain $F$ with its completion. 

It follows from Theorem \ref{thm:Liouville_reduction}, if the stop is $G$-invariant, the symplectic reduction of a stopped a Liouville manifold is also a stopped Liouville manifold. 
\begin{lemma}
Given a stopped Liouville manifold $(X,\frakf)$ with a Liouville action by $G$, where $\frakf$ is $G$-invariant, the symplectic reduction of $(X,\frakf)$ is a stopped Liouville manifold $(X\sq G, \tilde{\frakf} )$, where $\tilde{\frakf}$ is the reduction of $\frakf$.    
\end{lemma}
In the following section, we will refine the above statement in the case of a Liouville pair.

The result of removing an open neighborhood of 
$F$ is a Liouville manifold with boundary, called a \textit{Liouville sector}.  Near infinity, the Liouville sector can be identified with the symplectization of a contact manifold with a convex boundary. 

Equivalently, a Liouville sector can be defined as a Liouville manifold with boundary, characterized by the existence of a \textit{linear function} (i.e. $ZI=\alpha I$ for some $\alpha >0$) along the boundary, as shown in \cite{GPS1}.  

\begin{defn}
    A Liouville sector is Liouville manifold with boundary satisfying the following  condition: 
\begin{itemize}
  \item For some (any) $\alpha>0$, there exists a function $I:\partial X \rightarrow \mathbb{R}$ with $ZI=\alpha I$  near infinity and $X_{I}$ is strictly outward pointing. 
\end{itemize}
Without otherwise specified, we will always assume $\alpha = 1$. 
\end{defn}

The existence of such a function carries rich geometric information around the boundary, leading to many important properties of the Liouville sector. For our purpose, we only state some of them: 
\begin{enumerate}
    \item One can deform the Liouville $1$-form by a compactly supported function so that the Liouville vector field $Z$ is tangent to $\partial X$ everywhere.  
    \item  Every Liouville sector arises from a Liouville pair $(X,F)$. 
    \item The defining function $I$ gives rise to a decomposition $\partial X = F\times \bR$.
\end{enumerate}
Therefore, we will implicitly identify a Liouville sector with its underlying Liouville pair and we shall always assume $Z$ is tangent to its boundary.

In the spirit of Definition \ref{defn:Liouville_action}, we define the following. 
\begin{defn}
Let $X$ be a Liouville sector. Let $(G,\mu)$ be a group action together with  a smooth equivariant function $\mu:X\rightarrow \frakg^*.$ We say $(G,\mu)$ is a pre-Liouville action if 
\begin{enumerate}
    \item The restriction of $G$ action to $X^{\circ}$ is a Hamiltonian action with moment map $\mu|_{X^{\circ}}$; 
    \item Near infinity, it preserves $\lambda$ and satisfies \begin{align}
        \label{equ: equal_sign2}
        \langle \mu_{X},\xi \rangle = \lambda (X_{\xi}).  
    \end{align}
\end{enumerate}
\end{defn}
Note that Equation \eqref{equ: equal_sign2} hold for both $G|_{\partial X}$ and $G|_{X^{\circ}}$.

\begin{defn}
\label{equ:real_Liouville_action}
We say a proper group action $G$ on a Liouville sector $(X,\lambda)$ is  \textit{a Liouville action} if it preserves the Liouville form $\lambda$. 
\end{defn}

\begin{thm}
    Let $(X,\partial{X},\lambda)$ be a Liouville sector with a pre-Liouville action by $G$. One can deform $\lambda$ in a compact region so that the $G$-action on $(X,\overline{\lambda})$ is a Liouville action. Moreover, we have $\overline{\lambda}=\lambda +df$ where $f$ is a compactly supported function. 
\end{thm}

Let $\iota: \partial X\rightarrow X$ be the inclusion of the boundary into a symplectic manifold. Following conventions in \cite{Oh2021}, we will refer to the restriction form $$ \omega_{\partial X} \coloneqq   \iota^* \omega$$ as the \textit{pre-symplectic form}. Its kernel defines a distribution whose integral curves give rise to an one-dimensional foliation $C$, called the \textit{characteristic foliation}. For convenience, we also call the kernel characteristic foliation. The foliation is endowed with an orientation so that $\omega(N,C)>0$ for any inward pointing vector $N$. 
 We first recall a Lemma proved by \cite{Oh2021}. 

\begin{lemma}[Lemma 7.3 of \cite{Oh2021}]
    \label{lem:presymplectic-bdy}
        Fix a diffeomorphism
        $\phi: (X, \partial X)\to (X, \partial X)$ and suppose $\phi^*\lambda = \lambda  +df$
        for a function $f: X \to \bR$, not necessarily  compactly supported.
        Then the restriction $\phi|_{\partial X} = \phi_\partial : \partial X \to \partial X$
        is a presymplectic diffeomorphism, i.e., satisfies $\phi_{\partial}^*\omega_{\partial X} = \omega_{\partial X}$.
         In particular, it preserves the characteristic foliation of $ \partial X$.
\end{lemma}
In particular, when $\phi$ is given by a Liouville action, the above lemma tells us the characteristic foliation is preserved by the group action.   

\begin{thm}
    \label{thm:existence_of_linear_function}
    There exists a defining function $ZI = I$ such that $I$ is a $G$-invariant function. 
\end{thm}
\begin{proof}
    Let $\hat{I}$ be a prefixed $1$-defining function. Recall by our assumption, $\partial X$ is completely contained in $S\partial_{\infty}X$, thus we can parametrize $\partial X$ by the coordinate restricted from $S\partial_{\infty}X$. Since $G$ is compact, we define 
    \begin{align*}
        I= \int_{G} \hat{I}(g\cdot y, s)dg.  
    \end{align*}
   Hence, we have 
    \begin{align*}
        ZI & = \int_{G} \frac{\partial}{\partial s}\hat{I}(g\cdot y, s)dg \\ 
        & = \int_{G} \hat{I}(g\cdot y, s)dg \\ 
        & = I. 
    \end{align*} 
By the Lemma \ref{lem:presymplectic-bdy}, the characteristic foliation is invariant under $G$, so for a characteristic foliation $C$:
\begin{align*}
    dI|_{C} & = \int_{G} d\hat{I}(g\cdot y, s)|_{C}dg \\
    & = \int_{G} d\hat{I}(g\cdot y, s)|_{g\cdot C}dg \\
    & = \int_{G} \tau_{g}^* d\hat{I}(y, s)|_{C}dg > 0  
\end{align*}
where the third equality follows from the fact symplectomorphism is orientation preserving. 

\end{proof}

We now turn to a local description of  Liouville action on a Liouville sector. The following theorem is an adaptation of the product decomposition near the boundary of a Liouville sector specified to the linear function $I$ defined in Theorem \ref{thm:existence_of_linear_function}.

\begin{cor}
\label{cor:symplectic_reduction_of_Liouville_pair}
Let $X$ be a Liouville sector with a $G$-invariant defining function as given by Theorem \ref{thm:existence_of_linear_function}. Let $(X,F)$ be the underlying Liouville pair. The symplectic reduction of $X$ is again a Liouville sector, with its Liouville pair being $(X\sq G, F\sq G)$.  
\end{cor}

\begin{proof} 
    By Theorem \ref{thm:Liouville_reduction}, $X\sq G$ is a Liouville manifold with boundary. Since $\omega$ and $\lambda$ are $G$-invariant, the Liouville vector field $Z$ is also $G$-invariant, which descends to $X\sq G$.  Applying Theorem \ref{thm:existence_of_linear_function}, $I$ descends to a $1$-defining function for $X\sq G$, which shows $X\sq G$ is a Liouville sector. 

    By \cite[Proposition 2.25]{GPS1}, there is a decomposition near the boundary of $X$, 
    \begin{equation}
        \label{equ:decomposition}
        (X,\lambda_{X})= (F\times \mathbb{C}_{\mathrm{Re}\geq 0},\lambda_{F}+\lambda_{\mathbb{C}}+df),
        \end{equation}
        in which $I=y$ is the imaginary part of the $\mathbb{C}_{\mathrm{Re}\geq 0}$-coordinate, $\lambda_{\bC}=-ydx$, $(F,\lambda_{F})$ is a Liouville manifold and $f:F\times \mathbb{C}_{\mathrm{Re}\geq 0}\rightarrow \mathbb{R}$ satisfies the following properties:
\begin{enumerate}
\item $f$ is supported inside $F_{0}\times \mathbb{C}$ for some Liouville domain $F_{0}\subseteq F$.
\item $f$ coincides with some $f_{\pm \infty}: F\rightarrow \mathbb{R}$ for $|I|$ sufficiently large.
\end{enumerate}
We remark that $f$ is not a compactly supported function in general. By our choice of $I$, the leaves of $G|_{\partial X}$-orbits (i.e. restricted to $F\times \sqrt{-1}\bR$) lie in $F\times\{y\}$ for $y\in \sqrt{-1}\bR$. Since $G$ preserves the Liouville form and $f$ is supported inside some Liouville domain, it follows that $G$ preserves $\lambda_{F}$ away from a compact set. In other words, this is a pre-Liouville action and $F\sq G$ is a Liouville manifold.  
\end{proof}

With this understood, we study the lifted action on cotangent bundle of a manifold with boundary.
\begin{example}
($\dim Q>\dim G \text{ and } \dim Q \geq 2 $) Let $Q$ be a connected $n$-manifold with boundary, and $G$ a Lie group acting on $Q$. For simplicity, we assume $\partial Q$ is connected. As shown in Example \ref{ex:closed_manifold}, the cotangent lifted action in a Liouville action.   
  It follows that the symplectic reduction of the cotangent bundle is $T^*(Q/G)$, which is a Liouville sector with Liouville hypersurface being $T^*(\partial Q/G)$.  
\end{example}

\begin{example}[Open book decomposition]
Let $(M,\alpha)$ be a contact manifold. An open book decomposition is a pair $(B,\pi)$ where 
\begin{enumerate}
    \item $B$ is a codimension two submanifold called the \textit{binding} of the open book and 
    \item $\pi:M\setminus B\rightarrow \mathbb{S}^1$ is a fibration of the complement of $B$ such that each fiber $\pi^{-1}(\theta)$, $\theta\in \mathbb{S}^1$, corresponds to the interior of a compact hypersurface $\Sigma_{\theta}$, called the \emph{pages} of the open book, with $\partial \Sigma_{\theta}$ equal to $B$. 
\end{enumerate}
A page of the open book decomposition gives rise to a Liouville hypersurface as shown in \cite[Lemma 2.18]{GPS1}. Therefore, suppose a Liouville action preserves the pages, we can apply Corollary \ref{cor:symplectic_reduction_of_Liouville_pair}. The resulting symplectic reduction is a Liouville pair with its Liouville hypersurface being the reduction of the page. 
\end{example}

\begin{example}[Conic fibration]
    Consider the conic fibration $\bC^2\setminus\{\epsilon\} \rightarrow \bC\setminus\{\epsilon\}$ given by $(x,y)\rightarrow xy$. Write polar coordinates  $(x,y)=(r_{1}e^{i\theta_{1}},r_{2}e^{i\theta_{2}})$ and consider  the $\mathbb{S}^1$-action of weight $(1,-1)$. Then the Liouville form  $r_{1}^2d\theta_{1}+ r_{2}^2d\theta_{2}$ is invariant under $\mathbb{S}^1$-action. We choose $F_{0}$ as a page of the standard open book decomposition given by $\frac{xy}{|xy|}=-1$, in other words, they are the Liouville hypersurface specified by $\theta_{1}+\theta_{2}=\pi$.

    The map $f:\mathbb{C}^2\setminus\{xy=\epsilon\}\rightarrow \mathbb{C}\setminus\{\epsilon\}$ can also be understood as the quotient map, so the stopped Liouville manifold downstairs is $(\mathbb{C}\setminus\{\epsilon\},-\infty)$.  
\end{example}

\section{Homotopy quotient of Liouville manifolds} \label{sec:liouvilleborel}
\subsection{Borel construction revisited}
 The classifying space $BG$ of a compact Lie group  $G$, together with its universal bundle $EG\rightarrow BG$, can be constructed via a sequence of finite-dimensional  approximations. In particular, one may choose a collection of  $N$-connected closed manifold equipped with free $G$-actions such that $EG = \bigcup_{N} EG(N)$ (explicit constructions of these approximations  will be introduced in the Section \ref{sec:approximation}).
\begin{equation} 
    \label{diag:approximation1}
\begin{tikzcd} 
  G = EG(0) \arrow[r, hook] \arrow[d]
    & EG(1) \arrow[r, hook] \arrow[d]
    &\cdots
    \arrow[r, hook]
    &EG(N) \arrow[r, hook] \arrow[d]
    &\cdots  \arrow[r, hook] 
    &EG \arrow[d]\\
pt = BG(0) \arrow[r, hook] 
    & BG(1) \arrow[r, hook] 
    &\cdots
    \arrow[r, hook]
    &BG(N)\arrow[r, hook] &\cdots \arrow[r, hook] & BG 
    \end{tikzcd}
\end{equation}

For each $N \in \mathbb{Z}_{>0}$,  the space $EG(N)$ is an $N$-connected principal $G$-bundle over $BG(N).$ The $G$-action induces a Liouville action on their cotangent bundles, and we have the canonical identification 
\begin{align*}
    T^*EG(N) \sq G = T^*(EG(N)/G) = T^*BG(N).
\end{align*}
We choose a family of $G$-invariant Riemannian metrics $\{g_{N}\}$ on $EG(N)$ such that the restriction of each metric agrees with the one on the preceding stratum in the approximation tower. 

In general, for an embedded submanifold $i: X \hookrightarrow Y$, there is a natural exact sequence of vector bundles:
\begin{align*}
    0 \longrightarrow N^*X \longrightarrow i^*T^*Y \longrightarrow T^*X \longrightarrow 0.
\end{align*}
A choice of Riemannian metric determines a splitting of this sequence, and hence an  embedding of cotangent bundles as in   
\eqref{diag:approximation2}. In the same manner, the cotangent bundle $T^*EG$ is approximated by the sequence $T^*EG(N)$.

For each cotangent bundle $T^*EG(N)$, the  $G$-bundle $EG(N)$ embeds as the zero section $0_{EG(N)}\hookrightarrow T^*EG(N)$, which is a $G$-invariant exact  Lagrangian contained in  $\mu^{-1}(0)$. Its reduction is $0_{BG(N)}\subset T^*BG(N)$.  We note that $0_{EG(N)}$ and $0_{BG(N)}$ are compatible with the embeddings in \eqref{diag:approximation2}. 
\begin{equation} \label{diag:approximation2}
    \begin{tikzcd}
  T^*G = T^*EG(0) \arrow[r, hook] \arrow[d]
    & T^*EG(1) \arrow[r, hook] \arrow[d]
    &\cdots
    \arrow[r, hook]
    &T^*EG(N) \arrow[r, hook] \arrow[d]
    &\cdots \arrow[r, hook] 
    &T^*EG \arrow[d] \\
pt = T^*BG(0) \arrow[r, hook] 
    & T^*BG(1) \arrow[r, hook] 
    &\cdots
    \arrow[r, hook]
    &T^*BG(N) \arrow[r, hook] &\cdots \arrow[r, hook] & T^*BG  
    \end{tikzcd}
\end{equation}

Let $(X,\lambda)$ be a Liouville manifold (or sector) and let $G\rightarrow \mathrm{Aut}(X,\lambda)$ be a Liouville action.  We denote by $\mu_{X}$ the moment map for the action of $G$ on $X$, and  by $\mu_{T^*EG(N)}$ the moment map for the Liouville action of $G$ on $T^*EG(N)$, obtained by lifting the $G$-action on $EG(N)$.

The diagonal action of $G$ on $(X\times T^*EG(N), \lambda\oplus \lambda_{T^*EG(N)})$ is a free Liouville action. Here $\lambda_{T^*EG(N)}$ denotes the canonical one form of $T^*EG(N)$.

To define the model of the Borel construction that  we will use, we  give two equivalent descriptions. The first is  the associated bundle of the universal $G$-bundle, that is the classical Borel construction. The second one is obtained by the Liouville reduction of the diagonal action above. These two constructions agree by an  argument analogous to the classical \textit{Sternberg--Weinstein universal phase space construction}  \cite{weinsteinym}. 
\begin{defn}
    \label{defn:Borel_Liouville_space}
    Let $X$ be a Liouville manifold (or sector) equipped with a Liouville action of  $G$. For each $N$, the \textit{Borel--Liouville space} $X(N)$ is defined by either of the following equivalent constructions (up to choices of Riemannian metrics) : 
    \begin{itemize}
        \item{(\textbf{Construction \MakeUppercase{\romannumeral 1}})}  The associated bundle of the $G$-principal bundle $  \mu_{T^*EG(N)}^{-1}(0) \rightarrow T^*BG(N)$, i.e. $X(N)\coloneqq X\times_{G} \mu_{T^*EG(N)}^{-1}(0)$. 
        \item{(\textbf{Construction \MakeUppercase{\romannumeral 2}})} The Liouville reduction $X(N)\coloneqq \left(X\times T^*EG(N)\right)\sq G = \left(\mu_{X}+\mu_{T^*EG(N)}\right)^{-1}(0)/G$. 
    \end{itemize} 
\end{defn}
We define  the \textit{Borel--Liouville construction} (or simply Borel construction) to be 
\begin{align*}
    X_{G}\coloneqq  \varinjlim_{N}X(N). 
\end{align*} 
By Construction \MakeUppercase{\romannumeral 1}, $\mu_{T^*EG}^{-1}(0)$  deformation retracts onto its base  $EG$ in a $G$-equivariant way, and hence has the same homotopy type as $EG$. it follows that  $X_{G}$ is homotopy equivalent to the classical Borel construction. 

By Construction \MakeUppercase{\romannumeral 1}, there is a fibration 
\begin{equation} 
    \label{diag:X_fibration}
\begin{tikzcd} 
 X \arrow[r, hook] 
    & X(N) \arrow[r, "\pi_{N}"] 
     & T^*BG(N)  
    \end{tikzcd}.
\end{equation}
To prove the equivalence of two constructions, we begin by analyzing how the above  fibration structure is realized using Construction  \MakeUppercase{\romannumeral 2}.  A symplectic analogue appears in  \cite{Cazassus_2024}. 

Fix  $q\in BG(N)$, let $O_{q} \coloneqq T_{q}(G\cdot q)$ denote the tangent space to the $G$-orbit passing through $q$.  As explained in Example \ref{ex:closed_manifold}, the zero level set of the moment map $\mu_{T^*EG(N)}$ is the annihilator of the subbundle generated by the  $G$-orbits.
We choose a $G$-invariant metric on $EG(N)$, which determines an orthogonal splitting 
\begin{align*}
    T_{q}EG(N) = O_{q}\oplus O_{q}^{\perp} \cong \mathfrak{g}\oplus O_{q}^{\perp}.
\end{align*}
Passing to duals, we obtain a corresponding spiltting  $T_{q}^* EG(N) =  \mathfrak{g}^*\oplus (O_{q}^{\perp})^*$. 
We set $(\cO^{\perp})^* = \bigcup_{q}(O_{q}^{\perp})^*$, so that 
 $$\mu_{T^*EG(N)}^{-1}(0) = (\cO^{\perp})^*.$$ 
 
 Thus, the orthogonal projection $T^*EG(N) \rightarrow (\cO^{\perp})^*$,
composed with the quotient projection induces a projection $T^*EG(N) \rightarrow T^*BG(N)$. Restricting to the zero level set of the diagonal moment map gives a $G$-invariant map 
\begin{align}
    \mu_{\mathrm{diag}}^{-1}(0)\hookrightarrow X\times T^*EG(N) \rightarrow T^*BG(N)
\end{align}
which descends to a projection  $X(N) \xrightarrow[]{\pi_{N}} T^*BG(N)$. We now show that the fiber of $\pi_{N}$ is diffeomorphic to $X$. For any $b=(q_{b},p_{b})\in T^*BG(N)$, the inclusion $i_{b}$ of the fiber $\pi^{-1}_{N}(b)$ is given explicitly by 
\begin{equation}
    \label{equ:inclusion_of_the_fiber}
    i_{b}: X \rightarrow X(N): x\mapsto \left[x,(q_{b},p_{b}-\mu_{X}(x))\right]
\end{equation}
where we identify $\mu_{X}(x) \in   \mathfrak{g}^* \cong O_{q}^*$.

\begin{proof}[Proof of equivalence in  Definition \ref{defn:Borel_Liouville_space}]
    \label{rem:equivalence_of_two_constructions}        
     Let   
     \begin{align*}
        X(N)\coloneqq \left(X\times T^*EG(N)\right)\sq G, \quad X(N)'\coloneqq X \times_G \mu^{-1}_{T^*EG(N)}(0).
     \end{align*}
We construct an isomorphism of $G$-bundles between these two models. Consider the diagram 
            \begin{equation} 
            \begin{tikzcd} 
              X \arrow[r, hook] \arrow[equal]{d}
                & X(N) \arrow[r, "\pi_{N}"] \arrow[d, "\Phi"]
                & T^*BG(N)  \arrow[equal]{d}\\
             X \arrow[r, hook] 
                & X(N)' \arrow[r, "\pi_{N}'"] 
                 & T^*BG(N)  
                \end{tikzcd},
            \end{equation}
           where the vertical map   $\Phi$ is defined on representatives  by $$\Phi([x, (q, p)]) = [x, (q,\pr_{(\cO^{\perp})^*}(p))].$$ Here, $\pr_{(\cO^{\perp})^*}$ denotes the orthogonal projection onto the annihilator of the orbit directions. 
            
           It is straightforward to check $\Phi$ is a $G$-invariant isomorphism. 
    \end{proof}
    From Construction  \MakeUppercase{\romannumeral 2}, the Liouville reduction endows $X(N)$ with a Liouville $1$-form $\lambda_{X(N)}$. We now show that this Liouville structure is compatible  with the fibration.  
    \begin{rem}
        The two-form $(\Phi^{-1})^* d\lambda_{X(N)}$ on $X(N)'$ defines a weak coupling form  and hence determines a Hamiltonian connection. In particular, this form coincides with the symplectic structure  on the symplectic Borel space used in \cite{KLZ23}.
    \end{rem}
\begin{lemma}
    \label{lem:Liouville_fibration}
    The fibration \eqref{diag:X_fibration} is compatible with the Liouville form $\lambda_{X(N)}$ and the symplectic form $\omega_{X(N)}$ obtained in Theorem \ref{thm:Liouville_reduction}, in the sense that 
    \begin{align}
        i_{b}^*\lambda_{X(N)}= \lambda_{X},\quad    i_{b}^*\omega_{X(N)}= \omega_{X}
    \end{align}
    where  $i_{b}$ denotes the inclusion of the fiber at a point $b\in T^*BG(N)$. 
    \end{lemma}
    \begin{proof}
         Let $x\in X$ and $v\in T_{x}X$. From the explicit formula for $i_{b}$ in \eqref{equ:inclusion_of_the_fiber}, we have 
        \begin{align*}
            \left((i_{b})^*\lambda_{X(N)}\right)_{x}(v) = \lambda_{X}(v)_{x} + \lambda_{\mathrm{can}}(-(d \mu_{X})_{x}(v)).
        \end{align*}
        By the property of the canonical $1$-form, $\lambda_{\mathrm{can}}(-(d \mu_{X})_{x}(v))$ vanishes since $-(d \mu_{X})_{x}(v)$ lies in the vertical direction. Taking  derivatives establishes the compatibility of symplectic forms as well.
        \end{proof}

Definition \ref{defn:Borel_Liouville_space} shows that whenever a Liouville manifold carries a Liouville $G$-action, one may construct a model of  Borel construction  obtained by finite-dimensional approximations of Liouville manifolds. \footnote[1]{Throughout this paper we do not address whether the direct limit $X_{G}$ carries a Liouville structure. Our use of the Borel construction relies only on the finite-stage Liouville manifold. } 

 Recall that a connection on a fiber bundle $\pi: E\rightarrow B$ is a splitting of the tangent bundle  
\begin{equation}
    TE \cong TE^{\hor}\oplus    TE^{\ver}, 
\end{equation} 
where $TE^{\ver} = \ker d\pi$. As usual, we refer to $TE^{\ver}$ (resp. $TE^{\hor}$) as the \textit{vertical (resp. horizontal) bundle}.

In particular, the symplectic form $\omega_{X(N)}$ determines a canonical connection on the fibration 
\begin{align*}
    \pi_{N}:X(N)\rightarrow T^*BG(N),
\end{align*}
given by the  $\omega_{X(N)}$-orthogonal complement of $TX(N)^{\ver} $.  We refer to this as the  \textit{connection associated with $\omega_{X(N)}$}.  Let us conclude this section by recording some standard properties of this  connection.
\begin{rem}[Parallel transport]
    \label{rem:parallel_transport}
    The connection determines a well-defined notion of parallel transport, which allows us to construct local trivializations via families of parallel transports. 
    
    The differential of the projection $$d\pi_{N}: T_{x}X(N)\rightarrow T_{\pi_{N}(x)}T^*BG(N)$$ is surjective everywhere and restricts to an isomorphism $$d\pi_{N}|_{\mathrm{hor}}: T_{x}X(N)^{\hor}\rightarrow T_{\pi_{N}(x)}T^*BG(N).$$ Given a vector $ Y\in T_{\pi_{N}(x)}T^*BG(N)$, we denote  by $Y^{\#}$ its \textit{horizontal lift}, namely the preimage of  $Y$ under  $d\pi_{N}|_{\mathrm{hor}}$. 

    For any path $\gamma: [0,1]\rightarrow T^*BG(N)$, the associated parallel transport 
    \begin{align*}
        \Phi_{[\gamma]}: F_{\gamma(0)} \rightarrow F_{\gamma(1)}
    \end{align*}
    is a symplectomorphism between the fibers $(F_{\gamma(0)},\omega_{\gamma(0)})$ and  $(F_{\gamma(1)},\omega_{\gamma(1)})$. 
\end{rem}
\subsection{Borel construction of Lagrangians}
\label{rem:Borel_Lagrangians}
Let $(X,\omega, \mu)$ be a symplectic manifold with a Hamiltonian  $G$-action. If the $G$-action is free, then every $G$-invariant Lagrangian $L\subset \mu^{-1}(0)$ %\footnote{Recall that for any $G$-invariant Lagrangian $L$, there exists a unique central element $c\in\mathfrak{g}^*$ such that $L\subseteq \mu^{-1}(c)$. See e.g. \cite[Lemma 4.1]{CGS}. Therefore, after replacing $\mu$ by $\mu-c$, we may assume that $L\subset \mu^{-1}(0)$.} 
corresponds uniquely to a Lagrangian 
\begin{align*}
 \underline{L}\coloneqq L/G.
\end{align*}
Thus the map $     L\rightarrow \underline{L}$ defines a one-to-one correspondence between $G$-invariant Lagrangians $L$ in the zero level set and  Lagrangians in $X\sq G$. 

The moment map $\mu$ is constant on any connected $G$-invariant Lagrangian $L$ (see, for example, \cite[Lemma 4.1]{CGS}); we denote this constant by $\mu(L)$. Hence  in Construction \MakeUppercase{\romannumeral 2}, the reduction of the Lagrangian $$
L\times (b,-\mu(L))\subset \mu_{\mathrm{diag}}^{-1}(0)$$ defines a  Lagrangian submanifold of $X(N)$, lying over the zero section of $T^*BG(N)$. We denote these Lagrangians by $L(N)$. These Lagrangians correspond bijectively to  Lagrangians $L\times_{G}0_{EG(N)}$ in Construction \MakeUppercase{\romannumeral 1} by the identification established earlier. This fits into the diagram 
\begin{equation} 
    \label{diag:fibration}
\begin{tikzcd} 
  L \arrow[r, hook] \arrow[d]
    & L(N) \arrow[r] \arrow[d]
    & 0_{BG(N)}  \arrow[d]\\
 X \arrow[r, hook] 
    & X(N) \arrow[r, "\pi_{N}"] 
     & T^*BG(N)  
    \end{tikzcd}.
\end{equation}
We  set 
\begin{align*}
    L_{G}\coloneqq \varinjlim_{N}L(N),
\end{align*}
which recovers the usual Borel construction of $L$. Since each $L(N)$ is Lagrangian, parallel transport along the paths contained in $0_{BG(N)}$, computed using connection associated with $\omega_{X(N)}$,  preserves the Lagrangian fiber $L$. 

The preceding discussion applies to general symplectic manifolds and Lagrangians. In the exact setting, we additionally require the primitives of exact Lagrangians be chosen $G$-invariant. This can always be achieved. Indeed, let $f_{L}: L\rightarrow \bR$ be a  primitive of $L$, we may replace $f_{L}$ by its averaged function over $G$ which still satisfies $df_{L} = \lambda|_{L}$ as $\lambda$ is $G$-invariant.  This function also serves as a primitive for the  reduced Lagrangian  $\underline{L}$.

Using the associated-bundle description of the Borel--Liouville construction, we have the diagram 
\begin{equation} 
    \label{diag:associate_fibration}
\begin{tikzcd} 
  L  \arrow[d]
    & \arrow[l]  L \times  0_{EG}  \arrow[r] \arrow[d]
    & 0_{EG}  \arrow[d]\\
\underline{L}
    & \arrow[l] L_{G} \arrow[r, "\pi_{G}"] 
     & 0_{BG}  
    \end{tikzcd}.
\end{equation}
Via this diagram, the $G$-invariant primitive $f_{L}$ induces a global function on the Borel construction $L_{G}$ (by lifting to $L\times 0_{EG}$ and descending to $L_{G}$). By construction, the restriction of this function to every fiber of $\pi_{G}$ agrees with $f_{L}$ itself. Clearly, the same argument applies to each finite-dimensional approximation $L(N)$, and these constructions are compatible with respect to the approximation maps.

Combining this observation with Lemma~\ref{lem:Liouville_fibration}, we conclude that each fiber $L \subset L(N)$ is an exact Lagrangian with primitive given by  $f_{L}$ with respect to the fiberwise restriction of Liouville form $\lambda_{X(N)}$.
\subsection{Classifying spaces as Hilbert manifolds}
\label{sec:approximation}
In the previous subsection, we used a finite-dimensional approximation $EG(N)$ to $EG$, where each $EG(N)$ is a closed manifold. The purpose of this subsection is to provide an explicit and  ``workable'' model for these approximations and to develop a corresponding Morse framework that meets the requirements for  applications in Floer theory. We follow a similar strategy as in  \cite{KLZ23}. 

Let $G$ be a compact connected Lie group, $G$ can be embedded into $ U(k)$ for some sufficiently large $k\in \bN$. Since $U(k)$ acts freely on the contractible Stiefel variety $V(k,\infty)$, the restriction of this action to $G$ remains free. Thus $V(k,\infty)$ serves as a model for the  universal $G$-bundle $EG$, which we approximate by  
\begin{equation}
    EG(N)\coloneqq  V(k,N+k). 
\end{equation}
The corresponding classifying space and its finite-dimensional approximations are obtained by taking the quotient by $G$:
\begin{align*}
    BG(N)\coloneqq  EG(N)/G.  
\end{align*}
The infinite Stiefel variety $V(k,\infty)$ is an infinite dimensional Hilbert manifold and so is $BG$. 

\begin{defn}
\label{defn:pseudo_gradient_vector_field}
    A vector field $\cV$ is said to be  \textit{pseudo-gradient} adapted to a  Morse function $f:X\rightarrow \bR$ if it satisfies the following  conditions:
\begin{enumerate}
    \item in a Morse chart around a critical point, $\cV$ coincides with the standard  negative gradient with respect to  the Euclidean metric; and 
    \item $df(\cV)\leq 0$ where equality holds if and only if  at  a critical point of $f$.
\end{enumerate}
\end{defn}
The key ingredient here is that $BG(N)$  admits a rich supply of Morse functions together with adapted pseudo-gradient vector fields that respect  the system of finite-dimensional approximations. To carry this out, we begin by analyzing the fibration structure of $BG$.

Consider the natural projection from the Stiefel variety to the complex Grassmannian, 
\begin{align}
    \label{equ:projection_from_EG_to_Gr}
    EG = V(k,\infty) \rightarrow \Gr(k,\infty)
\end{align}
which sends a $k$-frame to  the subspace spanned by that frame. This projection descends to a  fibration 
\begin{align}
    \label{equ:projection_from_BG_to_Gr}
 U(k)/G \hookrightarrow BG\xrightarrow{   \pi_{BG}} \Gr(k,\infty). 
\end{align}
 These projections are compatible with the embeddings of $EG(N)$ and $BG(N)$ respectively. The corresponding projection from $BG(N)$ to $\Gr(k,N+k)$ is denoted by $\pi_{BG(N)}$. 

Let $V\subset  \bC^{\infty}$ be a closed subspace, and denote by $P_{V}$ the orthogonal projection onto $V$. Let $\cL(E,F)$ denote the space of continuous linear maps from a Banach space $E$ to $F$, and write $\norm{T}$ for the operator norm of $T\in \cL(E,F)$. The assignment $V\rightarrow P_{V}$  induces an embedding $ \Gr(k,\infty) \rightarrow \cL(\bC^{\infty},\bC^{\infty})$, and hence we may define a metric on $\Gr(k,\infty)$ by 
\begin{align}
    \mathrm{dist}(V_{1},V_{2})\coloneqq  \norm{P_{V_{1}}-P_{V_{1}}}. 
\end{align}
This makes $\Gr(k,\infty)$ a complete metric space. Moreover, one can show that the unit ball is contractible.  The infinite complex Grassmannian $\Gr(k,\infty)$ is also a Hilbert manifold. 

To construct Morse functions  on each finite dimensional approximation  $BG(N)$, we begin by constructing Morse functions on $\mathrm{Gr}(k,N+k)$ that satisfy the desired properties and then lift their gradient vector fields to   pseudo-gradient vector fields on $BG(N)$ via the fibration \eqref{equ:projection_from_BG_to_Gr} and the following lemma. 

\begin{lemma}
\label{lem:pseudo_gradient_Lemma}
Let $ F\hookrightarrow E\xrightarrow{\pi} B$ be a locally trivial fibration, where $F$ and $B$ are connected compact finite-dimensional manifolds. Fix an auxiliary connection with a Riemannian metric  of the  form $\pi^*g_{B}\oplus g_{F}$ where $g_{B}$ is a Riemannian metric on $B$,  and $g_{F}$ is a fiberwise metric. We further assume that these metrics are of Euclidean type in local coordinates near critical points of the Morse functions introduced below\footnote[1]{Technically, this assumption is not required for the current lemma if one drops the locally Euclidean assumption from Definition \ref{defn:pseudo_gradient_vector_field}; it will be used later when discussing the existence of a CW complex structure, see Remark \ref{rem:CW_complex}.}. 

Let $f_{B}:B\rightarrow \bR$ be a Morse function on the base.  Choose a function $f_{E}:E\rightarrow \bR$ satisfying the condition: in a local trivialization $U_{i}\times F \cong \pi^{-1}(U_{i})$, one has 
    \begin{align}
        \label{equ:pullback_morse}
        f_E = \pr_F^* f_{F}.         
    \end{align}
where $f_{F}: F\rightarrow \bR$ is a Morse function and  $\pr_F: U_{i} \times F \to F$ denotes  the projection to the  fiber.

Then there exists a function $F_{E}: E\rightarrow \bR$ such that the vector field 
\begin{align}
    \label{equ:vertical_plus_horizontal}
    \cV\coloneqq (-\nabla f_{B})^{\#}-\nabla^{v}f_{E}
\end{align}
is a pseudo-gradient vector filed for $F_{E}$. Here $(-\nabla f_{B})^{\#}$ is the horizontal lift the negative gradient vector field of $f_{B}$ with respect to the connection determined by the metric defined above, and   $\nabla^{v}f_{E}$ denotes the fiberwise gradient vector field of $f_{E}$.  
\end{lemma}
\begin{proof}
Choose a local trivialization $U_{i}\times F$ around each critical point $c_{i} \in \crit(f_{B})$ such that $f_{E}$ if of the form \eqref{equ:pullback_morse}.  Let $\phi_{i}:U_{i} \rightarrow \bR$  be a bump function supported in $U_{i}$ identically equal to  $1$ near $c_{i}$.  Define a function $F_{E}$  by 
\begin{align*}
    F_{E} \coloneqq  \pi^*f_{B} +\sum_{c_{i}\in \crit(f_{B})}\epsilon_{i} \phi_{i} f_{F},
\end{align*} 
where the constants $\epsilon_{i}>0$ are taken sufficiently small,  and $f_{F}$ is viewed as a function defined on $U_{i}\times F$. 

 Without loss of generality, we may assume that $\mathrm{supp}(d\phi_{i})$ lies in an  annulus  centered at $c_{i}$. Because $F$ is compact, for sufficiently small $\epsilon_{i}$, no new critical points are created. Moreover, in a neighborhood of each $c_i$ on which $\phi_i \equiv 1$, both $\nabla f_B$ and $\nabla^{v} f_F$ are gradient vector fields, and the connection induced by the metric agrees with the trivial connection $\ker(\pr_F)_*$ on $U_i \times F$ near the critical points. Thus, the vector field $\cV$ coincides in this region with the genuine gradient vector field of a quadratic form with respect to the Euclidean metric, provided that $f_B$ and $f_F$ themselves enjoy this property. 

To verify the Condition (2) of Definition \ref{defn:pseudo_gradient_vector_field}, we compute  
\begin{align}
    \label{equ:mixted_term}
    dF_{E}(\cV)= - \norm{\nabla f_{B}}^2 - \sum_{c_{i}\in \crit(f_{B})}\epsilon_{i} d\phi_{i}(\nabla f_{B})f_{F} - \sum_{c_{i}\in \crit(f_{B})} \epsilon_{i}\phi_{i}\norm{\nabla^{v}f_{F}}^2. 
\end{align} 
 Since $d\phi_{i}, \nabla f_{B}$ and  $f_{F}$ are bounded on the support of $d\phi_{i}$, for sufficiently small $\epsilon_{i}$, it follows that $dF_{E}(\cV)\leq 0$, with  equality holding only at the critical points of $F_{E}$.  
\end{proof}

Unfortunately, Lemma \ref{lem:pseudo_gradient_Lemma} does not extend to infinite-dimensional  Hilbert manifolds. The  obstruction is that a closed ball in a Hilbert space is compact if and only if the Hilbert space itself is finite-dimensional.  Thus,  the second term in \eqref{equ:mixted_term} can no longer be uniformly bounded; on the other hand, it may happen that $F_{E}$ acquires additional critical points not originating from those of $f_{F}$ and $f_{B}$. 

Nevertheless, for the purposes of the Borel construction, the following theorem is sufficient. The key point is that all the finite-dimensional approximation are compact, and all Morse-theoretic constructions take place at these finite levels before passage to the colimit. \footnote[1]{An alternative strategy would be to verify directly that vector fields of the form \eqref{equ:vertical_plus_horizontal} on a Hilbert manifold satisfy the analytic properties required in Morse theory, which is beyond the scope of the present paper.}

\begin{thmd}
\label{thm:family_of_Morse_functions}
Let $ BG $ and its finite-dimensional approximations  $BG(N)$ be as described above. A family of Morse-Smale pairs $\{(f_{N},\cV_{N})\}$ consisting of Morse functions $f_{N}$ and pseudo-gradient vector fields $\cV_{N}$ is said to be  admissible if following conditions hold: 
\begin{enumerate}
    \item For each $N\in \bN$, there is an inclusion of critical points sets: $$\Crit(f_{N})\subset \Crit (f_{N+1}).$$
    \item For any $c\in \Crit(f_{N})$, the stable manifold satisfy $W^{s}(f_{N+1},\cV_{N+1}; c)\times_{BG(N+1)}BG(N) = W^{s}(f_{N},\cV_{N};c)$. 
    \item For any $c\in \Crit(f_{N})$, the image of the unstable manifold $W^u(f_{N},\cV_{N}; c)$ in $BG(N)$ coincides with $W^u(f_{N+1},\cV_{N+1};c)$ in $BG(N+1)$. 
    \item For any $c\in BG(N+1)\setminus BG(N)$,  the negative gradient flow line $\gamma_{t}(c)$ satisfies $\gamma_{t}(c) \not\in BG(N)$ for all $t\geq 0$. 
    \item  For each $l\geq 0$, there exists an integer $N(l)>0$ such that $$
        |c|>l \quad \text{for all } N\geq N(l) \text{ and } c\in \Crit(f_{N})\setminus \Crit(f_{N-1}).$$
\end{enumerate}
Moreover, for any  sequence of strictly increasing negative real numbers $(a_{1},a_{2},\ldots)$, one can construct such a family of Morse functions.  
\end{thmd}

\begin{proof}
We begin by recalling the  construction of a Morse function on the complex Grassmannian; see \cite{Nic94,DV96, Nic07} for details.

Let  $ -a_{1}>-a_{2}>\cdots >-a_{n}>\cdots >0$ be a strictly increasing sequence of negative real numbers, and write 
\begin{align*}
    A\coloneqq  \mathrm{diag}(a_1, \ldots, a_n,\cdots), \quad  A_n \coloneqq  \mathrm{diag}(a_1, \ldots, a_n). 
\end{align*}  
The function \begin{align}
\label{equ:Morse_function_on_complex_grassmannian}
f_{A_n}(V) \coloneqq  \mathrm{Re} \left( \Tr(A_n P_V) \right)    
\end{align}
is a Morse function on $ \Gr(k, n)$  (We adopt the cohomological convention, so $f_{A_{n}}$ differs  from that  in \cite{Nic94} by a negative sign). The critical points of $ f_{A_{n }}$ are $A_{n}$-invariant $k$-planes, which are nondegenerate iff  $A_{n}$ is nondegenerate. Moreover, the negative gradient flow of $f_{A_{n}}$ though a $k$-plane $V$ in
the natural metric is given by 
\begin{align}
    \label{equ:Morse_flow}
    \gamma_{t}(V) = e^{-A_{n}t} V = \diag (e^{-a_{1}t}, \cdots, e^{-a_{n}t})V.
\end{align}
Thus, we obtain a Morse function $f_{A}$ on $ \Gr(k,\infty)$ whose restriction to each stratum is $f_{A_{k+N}}$. There is an natural inclusion of critical point sets obtained by identifying $A_{k+N}$-invariant $k$-planes with $A_{k+N+1}$-invariant $k$-planes. It is straightforward to verify that $f_{A}$ satisfies the conditions (2) (3) and~(4) from the explicit description of the gradient flow \eqref{equ:Morse_flow}. For condition~(5), note that for any critical point $p$ contained in  $\Gr(k,\infty)\setminus \Gr(k,k+N)$, we have $|p|>2N$, see~ \cite{Nic94} for an explicit computation of the Morse index.

Next, let $f_{U(k)/G}$ be a Morse function on the fiber $U(k)/G$. Choose an open cover $\{U_{\alpha}\}$ of $BG$ together with a  partition of unity subordinate to it $\{\rho_{\alpha}\}$ (the existence of such a partition for Hilbert manifolds is proved in \cite[Section 3]{Lan95}). We then define 
\begin{align*}
    F_{U(k)/G} = \sum_{\alpha}\rho_{\alpha}f_{U(k)/G}, 
\end{align*}
where both $\rho_{\alpha}$ and $f_{U(k)/G}$ are understood to be defined on local trivializations $U(k)/G \times U_{\alpha}$.

Fix a connection on the fibration \eqref{equ:projection_from_BG_to_Gr} and choose a corresponding Riemannian metric of the form $\pi_{BG}^* g_{\Gr(k,\infty)}\oplus g_{U(k)/G}$ (cf. Remark \ref{rem:CW_complex}). We then consider the vector field 
\begin{align*}
    \cV_{\infty}\coloneqq  (-\nabla f_{A})^{\#} + (-\nabla^v F_{U(k)/G}) 
\end{align*}
where $(-\nabla f_{A})^{\#}$ denotes the horizontal lift of the negative gradient vector field of $f_{A}$, and $-\nabla^v F_{U(k)/G}$ is the fiberwise negative gradient. 

We define $\cV_{N}\coloneqq  \cV_{\infty}|_{BG(N)}$, then it follows from the construction that 
\begin{align*}
    \cV_{N} = (-\nabla f_{A_{k+N}})^{\#} + (-\nabla^v F_{U(k)/G}).
\end{align*}
Thus we can  construct a Morse function $f_{N}$ as in Lemma \ref{lem:pseudo_gradient_Lemma}.  The family $\{(f_{N},\cV_{N})\}$  clearly satisfies the desired properties, as the  projections of $\cV_{N}$ to $\Gr(k,N+k)$ do. Up to a generic perturbation, they also satisfy the Morse-Smale condition. We will perform the perturbations inductively across the approximations, ensuring at each stage that all of the above properties remain intact. 
\end{proof}

\begin{rem}
\label{rem:PS_conditions}
 It is well known that Morse theory on infinite-dimensional Hilbert manifolds is well-defined under suitable analytical assumptions (see, for example, \cite{AM05,AM06,Qin10}). A standard way  to address the difficulties arising from infinite dimensionality  is to require that the Morse function $f$ satisfies the \emph{Palais-Smale Condition (C)} \cite{PS64}, which means every sequence $\{x_{n}\}\subset X$ for which $f(x_{n})$ is bounded and $\norm{\nabla f(x_{n})}\rightarrow 0$ possesses a convergent subsequence. If, in addition, each critical point $c\in \crit(f)$ has finite Morse index, then the  unstable manifolds are finite-dimensional and admit  compactifications in the usual way. We refer  the reader  to \cite{Qin10}, where such Morse functions are said to satisfy  the \textit{CF-conditions}. 

In our setting, we have constructed a Morse function on $ \Gr(k,\infty) $ whose critical points all have finite index. Moreover, the negative pseudo-gradient flow  of $f_{A}$ is explicitly defined for all $t\in \bR$, and is either unbounded or convergent to a limit point---behavior that is  stronger than required by the Palais-Smale condition.   
\end{rem}

\begin{example}
\label{ex:Schubert_cell}
Let $G= \mathbb{S}^1$, so that  $k=1$ and  $\Gr(1,N+1) = \mathbb{CP}^{N}$. For a point $V= [z] \in  \mathbb{CP}^{N}$, the corresponding projection operator $P_{V}$ is represented by the normalized outer product 
\begin{align*}
    P_{V} = \frac{1}{|z|^2}zz^*. 
\end{align*}
The fibers of the projection \eqref{equ:projection_from_BG_to_Gr} are $0$-dimensional. Hence, the Morse function takes the form $$f_{N}([z]) = \frac{1}{|z|^2} \sum_{i=1}^{N+1} a_{i}|z_{i}|^2 $$ which is the standard Morse function on $\mathbb{CP}^{N}$. Consider a point of the form $$[z] = [z_{1}:\cdots:z_{i}:1:0:\cdots :0]^T.$$ Then, we see that $[z]$ lies in the unstable manifold of the critical point of Morse index $2i$. Indeed, we have 
$$\lim_{t\rightarrow -\infty}\gamma_{t} ( [z]) = \diag(e^{(-a_{1}+a_{i+1})t},\cdots, e^{(-a_{i}+a_{i+1})t}, 1, \cdots )[z] = [0:\cdots :1 :0:\cdots 0].$$ 
All the Morse indices are even, and the unstable manifold coincide with the Schubert cells.

The situation for a more general unitary group $U(k)$ is similar. Let $I\coloneqq \{i_{1},\cdots, i_{k}\}\subset \{1,2,\cdots \}$. The unstable manifold corresponding to the $I$-th critical point consists of matrices of the form
\begin{align*}
    \begin{pmatrix}
        * & \cdots & * \\
        \vdots & & \vdots \\
        *  &  & \vdots \\
        1 &  & \vdots \\
          & \ddots  & \vdots \\
         & &  1
     \end{pmatrix}.
\end{align*} 
One sees that the dimension of the unstable manifold (and hence the Morse index) is given by $(i_{1}-1)+\cdots (i_{k}-k)$. This  provides an alternative verification of Condition (5) in Theorem-Definition~\ref{thm:family_of_Morse_functions}. 
\end{example}
\section{Equivariant Floer cohomology}
\label{sec:equivariant_Floer_cohomology}
\subsection{Almost complex structure}
\label{sec:almost_complex_structure}
Since the ambient space is non-compact, it is essential to choose almost complex structures that prevent pseudo-holomorphic curves from escaping to infinity. A well-established approach is to use contact type almost complex structures. We now recall this construction and adapt it to the Borel construction. 
\begin{defn}
Let $(\mathbb{R}\times V,\omega)$ be the symplectization of a contact manifold $V$. An $\omega$-compatible almost complex structure $J$ is said to be \textit{contact type} if it satisfies the following conditions: 
\begin{enumerate}
    \item $J$ is invariant under translations in $\mathbb{R}$. 
    \item $J(\partial_{r}) = R$ where $r$ is the coordinate of $\mathbb{R}$, and $R$ is the Reeb vector field on $V$.
    \item $J(\cD)\subset \cD$ where $\cD$ is the contact distribution. 
\end{enumerate} 
Theses conditions are equivalent to $dr= \lambda\circ J$ where $r= e^s$.  
\end{defn}

Fix a compatible almost complex structure $J_{T^*BG(N)}$ on $(T^*BG(N), \omega_{T^*BG(N)})$, the pullback $\pi^*_{N}J_{T^*BG(N)}$ is a natural almost complex structure on the horizontal subbundle $TX(N)^{\hor}$. We will denote this almost complex structure by $J_{T^*BG(N)}$ for the sake of simplicity. 

\begin{defn}
    \label{defn:admissible_almost_complex_structures}
Let $J_{T^*BG(N)}$ be an almost complex structure on $T^*BG(N)$ which is of contact type near infinity.    We consider almost complex structures on $X(N)$  of the form 
    \begin{align*}
        J_{X(N)} = \begin{pmatrix}
            J_{T^*BG(N)}   & 0 \\
           0 &  J_{F}  
        \end{pmatrix},
    \end{align*}
    viewed as a block matrix with respect to the splitting $$    TX(N)= TX(N)^{\hor}\oplus    TX(N)^{\ver}.$$     Moreover, we choose a subset $X^{\mathrm{in}}(N)\subset X(N)$ which is proper over $T^*BG(N)$ such that on the complement  $X(N) \setminus X^{\mathrm{in}}(N)$, the restriction of $J_{X(N)}$ to each fiber is of contact type.  
    
    The almost complex structure satisfying the above conditions is called \textit{admissible}.    
\end{defn}

It follows from the construction that  $\pi_{N}$ is $(J_{X(N)}, J_{T^*BG(N)})$-holomorphic and the space of admissible almost complex structures is contractible. Therefore, for the Borel spaces $\{X(N)\}$, we may choose a family of admissible almost complex structures $\{J_{X(N)}\}$ such that 
\begin{align*}
    J_{X(N)}|_{X(N-1)} = J_{X(N-1)}, \quad  D\pi_{N}\circ J_{X(N)} = J_{T^*BG(N)}\circ D\pi_{N}.
\end{align*}

\subsection{Family of Hamiltonians}
Recall that we use the coordinate transformation $r=  e^s$ on the cylindrical end. Let $(Y, \alpha)$ be a contact manifold, and let $R_{\alpha}$ be the corresponding Reeb vector field. We consider its positive half symplectization  $(Y\times (1,\infty), d(r \alpha))$. Let $h:Y\rightarrow \bR$ be a smooth function that is invariant under the Reeb flow.  Then the Hamiltonian function  $H\coloneqq  h(y)r$ defines the Hamiltonian vector field 
\begin{align*}
    X_{H}(y,r) = h(y) R_{\alpha}(y,r).
\end{align*}
 We refer to such functions as \textit{linear Hamiltonians} and the function $h(y)$ is called the \textit{slope  functions}. It is straightforward to verify that $ X_H $ satisfies the following conditions:

\begin{enumerate}
    \item $ZH = H$; 
    \item $dr(X_{H}) = 0$;
    \item $H = \lambda(X_{H})$
\end{enumerate}
where $Z$ and $\lambda$ are  the Liouville vector fields and Liouville form respectively.

\begin{defn}
\label{defn:family_Hamiltonians}
Fix a  family of Morse-Smale pairs $\{(f_{N},\cV_{N})\}$ as in Theorem \ref{thm:family_of_Morse_functions}.   Let $\{H(N)\}$ be a family of Hamiltonian functions, with each $H(N) \in C^{\infty}(X(N))$. 
We say that this family is \textit{admissible} if it satisfies the following conditions: 
\begin{enumerate}
    \item The restriction of $H(N+1)$ to $X(N)$ agrees with $H(N)$, i.e. $H(N+1)|_{X(N)} = H(N)$. 
    \item For each $b \in T^*BG(N)$, the restriction $H_{b}\coloneqq  H(N)|_{\pi_{N}^{-1}(b)}$ is linear near infinity. (We omit $N$ in the notation, as this condition is independent of $N$ by (1).)
     \item Near each $c \in \crit(f_N)$, the Hamiltonian $H_b$ takes the form 
    $$
        H_b = \pr_X^* H_c,
    $$
    where $\pr_X: U \times X \to X$ is the projection map in a local trivialization, and $H_c$ is a linear Hamiltonian defined on $X$.
    \item    For any two critical points $c_{0}, c_{1} \in \crit(f_{N})$, the parallel transport along the flow line of $\mathcal{V}_{N}$ connecting $c_{0}$ to $c_{1}$ does \emph{not} send a Hamiltonian chord of $H_{c_{0}}$ between $L_{0}$ and $L_{1}$ to a Hamiltonian chord of $H_{c_{1}}$ between $L_{0}$ and $L_{1}$.  
\end{enumerate}
 \end{defn}
 For any $b \in T^*BG(N)$, we denote by $F_b$ the fiber over $b$, which is Liouville isomorphic to $X$ by Lemma~\ref{lem:Liouville_fibration}. 
 
 The restriction $i_b^*\omega_{X(N)}$ of the symplectic form to the fiber $F_b$ agrees with the symplectic form $\omega_X$ on $X$. Hence, the fiberwise Hamiltonian vector field of $H(N)$ over a critical point $c \in \crit(f_N)$ coincides with the usual Hamiltonian vector field on $X$ with respect to $\omega_X$. Moreover, since the parallel transport preserves the Lagrangian submanifolds, conditions (4) is well-defined. 

 \begin{rem}
    In the Definition \ref{defn:admissible_almost_complex_structures} of the family of almost complex structures, we require that the almost complex structure is of contact type outside a region $X^{\mathrm{in}}(N) \subset X(N)$ that is proper over the base. By enlarging $X^{\mathrm{in}}(N)$ if necessary, we may assume that the Hamiltonian $H(N)$ is linear outside the same region.
 \end{rem}
\begin{defn} 
\label{defn:compatible_family}
Let $L_{0},L_{1}\subset X$ be exact cylindrical $G$-invariant Lagrangians. For an admissible family $\{ f_{N},\cV_{N}, H(N)\}$ associated to their Borel constructions, we say that  it is \textit{compatible with $L_{0}$ and $L_{1}$} if, for each level $N\in \bN$ and each critical point $c\in \crit(f_{N})$, the fiberwise time-$1$ Hamiltonian flow $\varphi_{H_{c}}^1(L_{0})$ intersect transversely with $L_{1}$ in the fiber $F_{c}$ and  their intersection set is compact. In this case, the \textit{Borel Floer complex} is defined by 
\begin{equation}
    CF(L_{0},L_{1};X(N))\coloneqq  \bigoplus_{\substack{\mathbf{x}\coloneqq  (c, x)} } \bK\cdot \mathbf{x},
\end{equation} 
where each generator $\mathbf{x}\coloneqq  (c,x)$ consists of a critical point   $c\in \crit(f_{N})$ and a  Hamiltonian $1$-chord  $x$ of $H_{c}$ with endpoints  on $L_{0}$ and $L_{1}$. 
\end{defn}
As usual, the \emph{action functional} at  a generator  $\mathbf x = (c, x)$  is defined by 
\footnote{A pair $(c, x)$ itself is not a critical point of $\mathcal A(c,x)\coloneqq \mathcal A_c(x) + f_N(c)$, as the Hamiltonian term $H_c$ depends on $c\in BG(N)$:}
\begin{equation*}
\label{eq: action}
    \mathcal A_c(x)\coloneqq f_{L_1}(x(1))-f_{L_0}(x(0)) + \int_{0}^{1} -x^*\lambda_X+ H_{c}(x(t)) dt. 
\end{equation*}
The definition make sense because $x$ is confined to a fiber $F_c\simeq X$. 
\begin{lemma}
\label{lem:abundant_choices}
The choice of compatible admissible families in Definition \ref{defn:compatible_family} forms a Baire set. Moreover one can choose the linear Hamiltonians to have the constant slopes $a\in \bR_{\geq 0}$ (with $a = 0$ understood as the case of a compactly supported Hamiltonian).
\end{lemma}
\begin{proof}
For the Morse-Smale pair $\{f_{N},\cV_{N}\}$ part, note that diagonal matrices with distinct eigenvalues form a dense subset in the spectrum of diagonal matrices (cf. Theorem \ref{thm:family_of_Morse_functions}), and that increasing sequences of real numbers constitute a convex set.

To construct the family of Hamiltonians,  take a local trivialization around each critical point $c_{i}$ and define $H_{b}$ near $c_{i}$ as in (3) of Definition \ref{defn:family_Hamiltonians}, extending it to all of  $X(N)$. For a generic choice of $H$, the image $\phi^{1}_{H_{c}}(L_{0})$ intersects $L_{1}$ transversely and is  contained in a compact set. Since for each $N\in \bN$, the critical set $\crit(f_{N})$ is finite, we can inductively perturb Hamiltonians in  compact subsets so that (4) of Definition \ref{defn:family_Hamiltonians} is satisfied. 

The intersection of all such choices over $N$ is  a Baire set by the Baire category theorem. 
\end{proof}
\subsection{Gradings and signs}
Under the assumption that the bicanonical bundle $\kappa^2$ of $X$ admits a trivialization (for example, when $2c_{1}(X) = 0$), we work with cohomologically unital, $\bZ$-graded $A_{\infty}$-categories over an algebraically closed field $\mathbb{K}$ of characteristic $2$.  See \cite{Sei08} for a detailed construction.

In general, for the family version of Floer theory, one only obtains  a $\bZ/2\bZ$-grading, and there is no canonical refinement to a $\bZ$-grading. This is because we make no prior assumptions on the automorphism group of $X$. However, in the context of Borel construction, there is a canonical way to obtain such a refinement. 

To  define the $\bZ$-grading, we fix a fiber of the projection $X(N)\rightarrow T^*BG(N)$, and choose  a trivialization of the bicanonical bundle on the fiber $X$. Since the structure group $G$ is connected, it acts trivially on the homotopy class of such a trivialization. Thus, the grading is well-defined on each fiber.

For each generator $\mathbf{x} \coloneqq  (c, x) \in CF(L_0, L_1; X(N))$, by the above discussion, there are well-define $\bZ$-gradings of $c$ and $x$. The grading of $|x|$ is defined as 
\begin{align*}
    |\mathbf{x}|\coloneqq  |c|+|x|. 
\end{align*}

\subsection{Floer theory in Borel construction}
\label{sec:borel_construction_of_floer_cohomology}
Let $L_{0},L_{1}\subset X$ be exact cylindrical Lagrangians, and fix a compatible admissible family  $\{f_{N},\cV_{N}, H(N)\}$ as in Definition \ref{defn:compatible_family}. Let 
\begin{align*}
    \mathbf{x}_{i} = (c_{i}, x_{i})\in CF(L_{0},L_{1};X(N)), \quad  i =0,1 
\end{align*}
be two generators of the Borel Floer complex. We abbreviate $H_{(\tau)} \coloneqq  H(N)_{\gamma(\tau)}$, where $\gamma$ denotes a trajectory of $\cV_{N}$; other auxiliary data (such as almost complex structures) are abbreviated in the same way.  We impose the boundary condition $\partial_{\tau}H_{(\tau)}\leq 0$ outside $X^{\mathrm{in}}(N)$, where the derivative is taken with respect to parallel transport along $\gamma$. Let  $\pi_{N}: X(N) \rightarrow T^*BG(N)$ denote the projection map as before.

\begin{defn}
    \label{defn:moduli_space}
    Let $\moduli(\mathbf{x}_{0},\mathbf{x}_{1};X(N))$ denote the space of pairs $$ \{(\gamma, u)\mid  \gamma: \bR\rightarrow 0_{BG(N)}, u: \bR\times [0,1] \rightarrow X(N)\}$$ satisfying the following conditions: 
    \begin{numcases}{}
            \partial_{\tau} \gamma - \cV_{N} = 0, \quad \pi_{N}\circ u = \gamma,  & \label{equ:horizontal_Floer_equation} \\ 
            \left(du^{\ver} - X_{H_{(\tau)}}\otimes dt\right)^{0,1}_{J_{(\tau,t)}^{\ver}} = 0, &  \label{equ:vertical_Floer_equation} \\
            u(\tau,0)\in L_{0}(N),\quad u(\tau,1) \in L_{1}(N), & \label{equ:boundary_condition} \\  
            \lim_{\tau \rightarrow -\infty} \gamma (\tau) = c_{0}, \quad  \lim_{\tau \rightarrow +\infty} \gamma (\tau) = c_{1}, & \\ 
            \lim_{\tau \rightarrow -\infty} u(\tau ,\cdot) = x_{0}, \quad  \lim_{\tau \rightarrow +\infty} u(\tau,\cdot) = x_{1}.
    \end{numcases}
    Here $du^{\ver}$ denotes vertical component of the differential $du$ with respect to the connection determined by $\omega_{X(N)}$, and $ X_{H_{(\tau)}}$ is the fiberwise Hamiltonian vector field. We define 
    \begin{equation}
        \cM(\mathbf{x}_{0},\mathbf{x}_{1};X(N))\coloneqq  \moduli(\mathbf{x}_{0},\mathbf{x}_{1};X(N))/ \bR. 
    \end{equation}
\end{defn} 
 From the horizontal equation \eqref{equ:horizontal_Floer_equation}, the projection $\pi_{N}\circ u$ is a Morse trajectory on the zero section of $T^*BG(N)$. Therefore, each element of $\cM(\mathbf{x}_{0},\mathbf{x}_{1};X(N))$ is of one of the following two types:
\begin{itemize}
    \item[(1)] A Hamiltonian-perturbed pseudo-holomorphic curve bounded by $L_{0}$ and $L_{1}$,  entirely contained in the fiber $F_{c_{i}}$ over the critical points $c_{0}, c_{1} \in \crit(f_{N})$; 
    \item[(2)] A continuation-type map associated to a family of Hamiltonians  $H_{(\tau)}$, whose projection $\pi_{N}\circ u$ is a non-constant Morse trajectory.
\end{itemize}
The first case corresponds to the situation where $\pi_{N}\circ u$ is constant, while the second arises when $\pi_{N}\circ u$ is non-constant.    Along a given trajectory $\gamma$,  we may trivialize  the bundle $X(N)$ and view the vertical component of $u$ as a map into a fixed fiber $X$. Under this identification, the boundary condition \eqref{equ:boundary_condition} simply becomes $$u(\tau, i)\in L_{i}$$ for $i\in \{0,1\}$, since the parallel transport preserves the Lagrangian. 

%The moduli space $\cM(\mathbf{x}_{0},\mathbf{x}_{1};X(N))$ is the common zero locus of the Fredholm sections defined by the Floer equations \eqref{equ:horizontal_Floer_equation} and \eqref{equ:vertical_Floer_equation}. Hence, its expected (virtual) dimension is given by 
%\begin{equation}
%    \label{equ:virtual_dimension}
%    \mathrm{vdim} \hspace{2pt} \cM(\mathbf{x}_{0},\mathbf{x}_{1};X(N)) =   |\mathbf{x}_{0}|-|\mathbf{x}_{1}|-1. 
%\end{equation}
%Let $Z\coloneqq  \bR\times [0,1]$. 
% Since the vertical and horizontal subbundles are  $\omega_{X(N)}$-orthogonal, the \textit{geometric energy} of the curve $u$ satisfies {\color{red}(correction needed )}
%\begin{align*}
%    E_{\mathrm{geom}}(u) & = \frac{1}{2}\int_{Z} |\partial_{\tau} u|^2 d\tau dt \\
%         & = \frac{1}{2} \int_{Z} \omega_{X(N)}\left(\partial_{\tau}u^{\ver}+ \partial_{\tau}u^{\hor}, J_{N}(\partial_{\tau}u^{\ver}+\partial_{\tau}u^{\hor})\right) d\tau dt\\
%         & =  \frac{1}{2}\int_{Z}\norm{du^{\ver} - X_{H_{(\tau)}}\otimes \beta}^2d\tau \wedge  dt + \frac{1}{2}\int_{\bR}|\partial_{\tau} \gamma|^2 d\tau \\
%         &\leq  E_{\mathrm{geom}}(u)^{\mathrm{vert}} + C. 
%\end{align*} 
%where $\partial_{\tau}u^{\ver}$ and $\partial_{\tau} u^{\hor} $ denote the  projections of $\partial_{\tau}u $ onto the vertical and horizontal subbundles, respectively, and $C$ is a uniform energy bound for Morse trajectories. Therefore, to obtain uniform energy bounds, it suffices to control the vertical geometric energy, which will be discussed later. 

The moduli space $\cM(\mathbf{x}_{0},\mathbf{x}_{1};X(N))$ is the common zero locus of the Fredholm sections defined by the Floer equations \eqref{equ:horizontal_Floer_equation}
and \eqref{equ:vertical_Floer_equation}. Hence, its expected (virtual) dimension is given by 
\begin{equation}
    \label{equ:virtual_dimension}
    \mathrm{vdim} \hspace{2pt} \cM(\mathbf{x}_{0},\mathbf{x}_{1};X(N)) =   |\mathbf{x}_{0}|-|\mathbf{x}_{1}|-1. 
\end{equation}
\begin{thm}
\label{thm:regular_moduli}
For a generic choice of admissible almost complex structure,  the moduli space   $\cM(\mathbf{x}_{0},\mathbf{x}_{1};X(N))$ is a smooth manifold of the expected dimension given in \eqref{equ:virtual_dimension}. 
\end{thm}
\begin{proof}
    It is standard that the transversality for moduli space of Floer strips can be achieved by choosing generic time-dependent almost complex structures; see, for example, \cite{FHS95} and \cite[(9k)]{Sei08}. We briefly recall the  argument and explain how it applies in our setting. 
    
    The key observation is that for a generic choice of time-dependent almost complex, the associated linearized Cauchy-Riemman operator is surjective, except possibly at constant components, which must be treated separately. 

    In the present case,  the zero set defining the moduli space is the joint  zero locus of the two Floer (Morse) equations  \eqref{equ:horizontal_Floer_equation} and \eqref{equ:vertical_Floer_equation}. The corresponding linearized operator decomposes as the  direct sum of the linearizations of horizontal and vertical  components. 
    
    Surjectivity of this operator follows from the surjectivity of each component: the Floer part can be made transverse by a generic perturbation of the almost complex structure, while the Morse part is transverse by the Morse-Smale condition, which holds by construction. 

    To complete the proof, we must rule out the constant solutions.  This is ensured by  Condition (4) in Definition \refeq{defn:family_Hamiltonians}.  
    \end{proof}

\subsubsection{Compactness}
We now show that the usual maximum principle is satisfied, ensuring the  Gromov-Floer compactification applies. 
\begin{thm}
    \label{thm:Gromov_compactification}
        If $\dim \cM(\mathbf{x}_{0},\mathbf{x}_{1};X(N)) \leq 1 $, $\cM(\mathbf{x}_{0},\mathbf{x}_{1};X(N))$ admits a Gromov compactification $\overline{\cM}(\mathbf{x}_{0},\mathbf{x}_{1};X(N))$.  
\end{thm}
The proof is almost identical to the classical case. For completeness and reader's convenience, we include the proof here. A more global formulation of this argument will be presented later in Definition \ref{lemd:fiberwise_wrapped_data}.

Since each fiber is non-compact, we must impose additional conditions on the Lagrangian boundary (equivalently, on the Hamiltonian functions) to prevent the strip form escaping to infinity. Recall that we require $\partial_{\tau} H_{(\tau)}\leq 0$ outside $X^{\mathrm{in}}(N)$.

Our approach follows the standard argument using \emph{non-negative moving boundary conditions}. We adapt the following lemma from \cite[Lemma 7.2]{abouzaid2010open} (see also \cite[Lemma 2.46]{GPS1}) to our setting.

To define the energy, for later use, we give  a more general statement. Let $\Sigma$ be a compact Riemman surface with boundary, and $\alpha$ a  $1$-form satisfying $d\alpha\leq 0$ (see \eqref{equ:sub_closed_1_form} for detailed definitions).  We write $\lambda^{\ver}$ the vertical component of the Liouville form $\lambda_{X(N)}$ (for simplicity, we drop the subscripts as it is independent of the stratum). We define the  \textit{geometric energy} by 
\begin{align*}
   E(u)^{\ver} \coloneqq  \frac{1}{2}\int_{\Sigma} \Vert du^{\ver} -X_{H_{(\tau)}}\otimes\alpha\Vert^2
    = \int_{\Sigma} u^* d\lambda^{\ver} - u^*(dH_{(\tau)})\wedge \alpha 
\end{align*}
where $dH_{(\tau)}$ is taken fiberwise. Integrating by parts, we have 
\begin{align}
\label{equ:topological_energy}
    E(u)^{\ver} & = \int _{\Sigma} u^* d\lambda^{\ver} - d (u^*H_{(\tau)}\wedge \alpha) + \partial_{\tau} (u^* H_{(\tau)})d\tau\wedge \alpha  + u^* H_{(\tau)} d\alpha \\
    & \leq \int _{\Sigma} u^* d\lambda^{\ver} - d (u^*H_{(\tau )} \alpha)
\end{align}   
For the inequality, we use the condition $\partial_\tau H_{(\tau)}\leq 0$ and the subclosedness of $\alpha$. When $\Sigma$ is a strip as in Definition \ref{defn:moduli_space}, applying Stokes' theorem we obtain 
\begin{align}
     E(u)^{\ver} \leq \mathcal A_{c_0}(x_0) - \mathcal A_{c_1}(x_1)
\end{align}
which is the standard inequality bounding geometric energy by topological energy.

\begin{rem}
 When we trivialize the bundle $X(N)$ over a given trajectory,  the vertical geometric energy  coincides with the standard energy of a Floer trajectory with (possibly) moving Lagrangian boundary conditions. The same statement also applies to the action functional.
\end{rem}

For a Liouville manifold, the following lemma is sufficient to ensure the compactness of the moduli spaces. In the case of a Liouville sector, the wrapping Hamiltonian we construct is chosen so that $\partial_{\tau} H_{(\tau)}= 0$ outside $X^{\mathrm{in}}(N)$; see Lemma-Definition \ref{lemd:fiberwise_wrapped_data}. In other words, we do not consider moving boundary condition which are not compactly supported on a Liouville sector.
\begin{lemma}
\label{lem:no_escape_lemma}
    Let $u$ be a pseudo-holomorphic curve as in Definition \ref{defn:moduli_space}. Suppose that $u(\partial \Sigma) \subset X(N) \setminus X^{\mathrm{in}}(N),$ then either $u$ is a locally constant map, or its image lies entirely in $\partial X^{\mathrm{in}}(N)$.
\end{lemma}
\begin{proof}
It follows from Equation \eqref{equ:topological_energy} that 
{\allowdisplaybreaks
\begin{align}
    E(u)^{\ver} 
    & \leq \int _{\Sigma} u^* d\lambda^{\ver} - d (u^*H_{(\tau )} \alpha)\\
    &= \int_{\partial \Sigma} u^*\lambda^{\ver}  - u^*H_{(\tau)} \alpha \\
    & = \int_{\partial \Sigma} u^* \lambda^{\ver} - u^* \lambda^{\ver}(X_{H_{(\tau)}})\alpha \\ \label{equ:linear_Hamiltonian_substitution}
    & = \int_{\partial \Sigma} \lambda^{\ver}\left(du^{\ver}- X_{H_{(\tau )}}\otimes \alpha \right) \\
    & =   \int_{\partial \Sigma} -\lambda^{\ver}J\left(du^{\ver}- X_{H_{(\tau )}}\otimes \alpha \right)j \\
    & = \int_{\partial \Sigma} -dr \left(du^{\ver}- X_{H_{(\tau )}}\otimes \alpha \right)j \\ \label{equ:coordinate_r}
    & = \int_{\partial \Sigma}  -dr du^{\ver} j \\ 
    & \leq 0. 
\end{align}
}

The identity in Equation~\eqref{equ:linear_Hamiltonian_substitution} follows directly from the assumption that $H_{(\tau)}$ is linear outside $X^{\mathrm{in}}(N)$.
We note that the radius function $r$ appearing in Equation~\eqref{equ:coordinate_r} is globally defined outside of $X^{\mathrm{in}}(N)$ since it is $G$-invariant. 

This suffices to conclude the argument.
\end{proof}  

\subsubsection{Differential and Equivariant Floer cohomology}
\label{sec:differential}
We define the Floer differential on $ CF^{\bullet}(L_{0},L_{1};X(N))$ by 
\begin{align*}
    \partial_{N} \mathbf{x}_{1}\coloneqq  \sum_{|\mathbf{x}_{0}| = |\mathbf{x}_{1}|+1}\#\cM (\mathbf{x}_{0},\mathbf{x}_{1};X(N)) \mathbf{x}_{0}. 
\end{align*}
By the compactness of the moduli space, this differential is well-defined and satisfies $\partial_{N}^2= 0$. We denote the corresponding Floer cohomology by 
\begin{align*}
    HF^{\bullet}(L_{0},L_{1}; X(N))\coloneqq  H^{\bullet}(CF^{\bullet}(L_{0},L_{1};X(N)), \partial_{N}).
\end{align*}
 Since $\crit{f_{N}}\subset \crit{f_{N+1}}$ and $H(N+1)|_{X(N)} = H(N)$, there is a natural inclusion of chain complexes
\begin{align*}
    CF^{\bullet}(L_{0},L_{1};X(N)) \hookrightarrow CF^{\bullet}(L_{0},L_{1};X(N+1)).
\end{align*}
 We define $P_{N+1}$ to be the projection map  
 \begin{align*}
    P_{N+1}:CF^{\bullet}(L_{0},L_{1};X(N+1))\rightarrow CF^{\bullet}(L_{0},L_{1};X(N)).
 \end{align*}

\begin{lemma}
\label{lem:approximation_commutative}
    The following diagram commutes:
\begin{equation}
    \label{equ:approximation_commutative}
     \begin{tikzcd}
         CF^{\bullet}(L_{0}, L_{1}; X(N+1)) \arrow[r, "P_{N+1}"] \arrow[d, "\partial_{N+1}"'] &  CF^{\bullet}(L_{0}, L_{1}; X(N)) \arrow[d, "\partial_{N}"] \\
         CF^{\bullet+1}(L_{0}, L_{1}; X(N+1)) \arrow[r, "P_{N+1}"] & CF^{\bullet+1}(L_{0}, L_{1}; X(N))
     \end{tikzcd}.
 \end{equation}  
 In other words, the projection map $P_{N}$ defines a chain map.  
\end{lemma}
\begin{proof}
Let $\mathbf{x}_{i} = (c_{i},x_{i}) \in CF^{\bullet}(L_{0}, L_{1}; X(N))$ for $i=0,1$. Each element $ (\gamma, u) \in \cM (\mathbf{x}_{0},\mathbf{x}_{1};X(N))$ projects to a Morse trajectory $\gamma$ on $BG(N)$. The moduli space admits a fibration 
\begin{align}
    \label{equ:fibration_of_moduli_space}
    \cM_{\gamma}(\mathbf{x}_{0},\mathbf{x}_{1};X(N)) \rightarrow  \cM (\mathbf{x}_{0},\mathbf{x}_{1};X(N)) \rightarrow \cM_{\mathrm{Morse}}(c_{0},c_{1};\cV_{N}),
\end{align}
where $\cM_{\mathrm{Morse}}(c_{0},c_{1};\cV_{N})$ denotes the moduli space of Morse trajectories connecting $c_{0}$ and $c_{1}$. 

By the properties of the pseudo-gradient vector field $\cV_{N}$ established in Theorem \ref{thm:family_of_Morse_functions}, we have 
\begin{align*}
    W^u(f_{N+1},\cV_{N+1}; c_{0})\cap  W^s(f_{N+1},\cV_{N+1}; c_{1}) = W^u(f_{N},\cV_{N}; c_{0})\cap  W^s(f_{N},\cV_{N}; c_{1}),
\end{align*}
which implies that  the count of trajectories is preserved under the inclusion $X(N)\subset X(N+1)$. In particular, for $|x_{0}|=|x_{1}|+1$,
\begin{align*}
    \#\cM (\mathbf{x}_{0},\mathbf{x}_{1};X(N+1)) = \#\cM (\mathbf{x}_{0},\mathbf{x}_{1};X(N)).
\end{align*}
Hence, the Floer differential commutes with the projection $P_{N}$: 
\begin{align*}
    P_{N+1} \circ \partial_{N+1} \mathbf{x}_{1} & = \sum_{\substack{|\mathbf{x}_{0}| = |\mathbf{x}_{1}|+1\\c_{0}\in \crit(f_{N+1})}}\#\cM (\mathbf{x}_{0},\mathbf{x}_{1};X(N+1)) P_{N+1}(\mathbf{x}_{0})\\
    & =  \sum_{\substack{|\mathbf{x}_{0}| = |\mathbf{x}_{1}|+1\\c_{0}\in \crit(f_{N})}}\#\cM (\mathbf{x}_{0},P_{N+1}(\mathbf{x}_{1});X(N)) \mathbf{x}_{0} \\
    & =  \partial_{N}\circ P_{N+1}\mathbf{x}_{1}.
\end{align*}
Thus, the diagram \eqref{equ:approximation_commutative} commutes and $P_{N}$ defines a chain map. 
\end{proof}

\begin{defn}
The \textit{equivariant Lagrangian Floer cohomology} is defined as the inverse limit 
\begin{align*}
    HF^{\bullet}_{G}(L_{0}, L_{1})\coloneqq  \varprojlim_{N} HF^{\bullet}(L_{0},L_{1}; X(N)). 
\end{align*}
\end{defn}
Recall that an inverse system $\{A_{i}, f_{ij}\}_{i\in I}$ satisfies the \textit{Mittag-Leffler condition} if for every index $i$, there exists $j\geq i$ such that for all $k\geq j$, 
\begin{align*}
    \mathrm{Im}(f_{ik}) = \mathrm{Im}(f_{ij}).
\end{align*}
 If this condition holds, then the  the first derived functor $\varprojlim^1$ vanishes, and  the inverse limit respects the cohomology. 

In our case, the inverse system is formed by projections maps which are always surjective. Hence the system satisfies Mittag-Leffler condition. Accordingly, we define the \textit{equivariant Floer cochain complex} by 
\begin{align*}
    CF_{G}^{\bullet}(L_{0},L_{1})\coloneqq \varprojlim_{N}{CF^{\bullet}(L_{0}, L_{1}; X(N))}.
\end{align*}
This complex depends on several choices of auxiliary data. However, up to quasi-isomorphism, it is well-defined and its cohomology agrees with the equivariant Floer cohomology $HF^{\bullet}_{G}(L_{0}, L_{1})$. 

\subsection{Continuation maps}
\label{sec:continuation_maps}
We now discuss the the natural maps between Floer complexes arising from varying the different choices of auxiliary data such as Hamiltonian functions and almost complex structures. This involves two aspects: the invariance of the Floer cohomology and the wrapped Floer cohomology. 

Let $L_{0}, L_{1}\subset X$ be exact cylindrical Lagrangians. We first fix an admissible family of Morse--Smale paris $\{f_{N},\cV_{N}\}$. Let $\{H^{-}(N)\}$ and  $\{H^{+}(N)\}$ be admissible families of compatible Hamiltonian functions satisfying $$\partial_{\tau}H^{\pm}_{(\tau)}\leq 0.$$ 

For a one-parameter family of admissible cylindrical Hamiltonians $\{H_{\tau}(N)\}$, we say it is \textit{monotone homotopy} if  satisfy the following conditions: 
\begin{enumerate}
    \label{equ:non_negative_homotopy}
    \item $H_{\tau} = H_{-}$ for $\tau \ll 0$, 
    \item $H_{\tau} = H_{+}$ for $\tau  \gg 0 $, and 
    \item $\partial_{\tau} ((H_{\tau})_{(\tau)})\leq 0$ outside $X^{\mathrm{in}}(N)$, where $ (H_{\tau})_{(\tau)} \coloneqq  H_{\tau}|_{\gamma(\tau)}$ and the derivative is taken with respect to parallel transport along a trajectory $\gamma$ of $\cV_{N}$,
\end{enumerate}
and we define a \textit{continuation map}:
\begin{align*}
    \kappa^{N}_{H^{-},H^{+}}: CF^{\bullet}(L_{0}, L_{1}; X(N), H^{+}(N)) \rightarrow CF^{\bullet}(L_{0}, L_{1}; X(N), H^{-}(N)). 
\end{align*}

\begin{rem}
    Let $\psi \in C^{\infty}(\bR, [0,1])$ be a monotone increasing function such that $\psi(\tau) = 0$ for $\tau \ll 0$ and $\psi(\tau ) =1$ for $\tau \gg 0$. This gives rise to a homotopy of Hamiltonians
\begin{align}
    H_{\tau } = (1-\psi(\tau)) H^{-}+ \psi(\tau) H^{+},\quad  \tau \in \bR
\end{align} 
 which induces an exact Lagrangian isotopy. Moreover, we observe that $ H_{\tau }$ satisfies $\partial_{\tau} ((H_{\tau})_{(\tau)})\leq 0$ away from $X^{\mathrm{in}}(N)$ whenever $H^{-} \geq  H^{+}$ near infinity. In particular, this includes the case where $L_{0}$ and $L_{1}$ are disjoint near infinity and $H^{\pm}$ are compactly supported. 
\end{rem}
For  generators $\mathbf{x}_{-} = (c_{-},x_{-})$ and $\mathbf{x}_{+} = (c_{+},x_{+})$  of $CF^{\bullet}(L_{0}, L_{1}; H^{-}(N), X(N))$ and $ CF^{\bullet}(L_{0}, L_{1}; H^{+}(N), X(N))$ respectively, choose a family of admissible almost complex structures, we define the moduli space as follows. 

\begin{defn}
    \label{defn:moduli_space_continuation_map}
    Let $\cM(\mathbf{x}_{-},\mathbf{x}_{+};X(N), (H_{\tau})_{\tau \in \bR})$ denote the space of pairs $$ \{(\gamma, u)\mid  \gamma: \bR\rightarrow 0_{BG(N)}, u: \bR\times [0,1] \rightarrow X(N)\}$$ satisfying the following conditions: 
    \begin{numcases}{}
        \partial_{\tau} \gamma - \cV_{N} = 0, \quad \pi_{N}\circ u = \gamma,    & \\ \label{equ:cont_horizontal_Floer_equation}
            \left(du^{\ver} - X_{(H_{\tau})_{(\tau)}}\otimes dt\right)^{0,1}_{J_{(\tau,t)}^{\ver}} = 0, & \\ \label{equ:cont_vertical_Floer_equation} 
            u(\tau,0)\in L_{0}(N),\quad u(\tau,1) \in L_{1}(N),  & \\ 
            \lim_{ \tau \rightarrow -\infty} (\gamma (\tau ), u(\tau,\cdot)) = \mathbf{x}_{-}, \quad  \lim_{ \tau \rightarrow +\infty} (\gamma (\tau ), u(\tau,\cdot)) = \mathbf{x}_{+}.
    \end{numcases}
    Here $du^{\ver}$ denotes vertical component of the differential $du$. 
\end{defn}
The virtual dimension of $\cM(\mathbf{x}_{-},\mathbf{x}_{+};X(N), (H_{\tau})_{\tau \in \bR})$ is given by $|\mathbf{x}_{-}|-|\mathbf{x}_{+}|$. By arguments analogous to Theorem \ref{thm:regular_moduli} and Theorem \ref{thm:Gromov_compactification}, we obtain the following conclusions: 
\begin{enumerate}
    \item For a generic choice of admissible almost complex structure,  the moduli space   $\cM(\mathbf{x}_{-},\mathbf{x}_{+};X(N), (H_{\tau})_{\tau \in \bR})$ is a smooth manifold of the expected dimension. 
    \item  If $\dim \cM(\mathbf{x}_{-},\mathbf{x}_{+};X(N), (H_{\tau})_{\tau \in \bR}) \leq 1 $, $\cM(\mathbf{x}_{0},\mathbf{x}_{1};X(N))$ admits a Gromov-Floer compactification. 
\end{enumerate}
We therefore define \textit{continuation map} 
\begin{align}
    \kappa_{H^{-},H^{+}}^{N} (\mathbf{x}_{+})\coloneqq  \sum_{|\mathbf{x}_{-}| =|\mathbf{x}_{+}|} \#\cM(\mathbf{x}_{-},\mathbf{x}_{+};X(N), (H_{\tau})_{\tau \in \bR}) \mathbf{x}_{-}. 
\end{align}
A standard homotopy argument leads to the following argument. 
\begin{prop}
\label{prop:continuation_map_properties}
\begin{enumerate}
    \item  The linear map $ \kappa_{H^{-},H^{+}}^{N}$ is a chain map. 
    \item Up to homotopy, $ \kappa_{H^{-},H^{+}}^{N}$ depends only on the endpoints $H^{-}$ and $H^{+}$. 
    \item The induced maps on Floer cohomology satisfy the composition law $F_{H^{-},H'}^{N}\circ F_{H',H^{+}}^{N} = F_{H^{-},H^{+}}^{N}$, provided the corresponding homotopies are monotone. 
\end{enumerate}
\end{prop}

\begin{proof}
    The first statement follow from  analyzing the boundary strata of the moduli space  $\cM(\mathbf{x}_{-},\mathbf{x}_{+};X(N), (H_{\tau})_{\tau \in \bR})$. 
    
    For (2), note that the space of admissible auxiliary data is convex. The standard construction of chain homotopies in Floer theory therefore applies verbatim; see \cite{Flo88,MS04} for detailed proofs.
    
    Finally, the statement (3) follows from the usual gluing argument.
    \end{proof}

We now turn our attention to the Borel construction. An argument entirely analogous to that of Lemma \ref{lem:approximation_commutative} shows that the continuation maps are compatible with the projection maps in the inverse system, in the sense that the following diagram commutes: 
\begin{equation}
    \label{equ:continuation_commutative}
     \begin{tikzcd}
        CF^{\bullet}(L_{0}, L_{1}; X(N+1), H^{+}(N+1)) \arrow[r, "P_{N+1}"] \arrow[d, "\kappa_{N+1}"'] &  CF^{\bullet}(L_{0}, L_{1}; X(N), H^{+}(N)) \arrow[d, "\kappa_{N}"] \\
        CF^{\bullet}(L_{0}, L_{1}; X(N+1), H^{-}(N+1)) \arrow[r, "P_{N+1}"] & CF^{\bullet}(L_{0}, L_{1}; X(N),H^{-}(N))
     \end{tikzcd}.
 \end{equation}  
 This commutativity follows from the fact that the moduli space of continuation trajectories projects to the moduli space of Morse trajectories, and the rest of the argument proceeds  exactly as in Lemma \ref{lem:approximation_commutative}. 
 
 As a consequence, we obtain a continuation map on the level of equivariant  Floer cohomology for different choices of Hamiltonian data 
 \begin{align*}
    F_{H^{-},H^{+}}^{G}: HF_{G}^{\bullet}(L_{0},L_{1}; \{H^{+}(N)\}) \rightarrow HF_{G}^{\bullet}(L_{0},L_{1}; \{H^{-}(N)\}), 
 \end{align*}
 which we refer to as the \textit{G-equivariant continuation map}. 

\subsubsection{Invariance of equivariant Floer cohomology}
Suppose that the admissible families of Hamiltonians $\left\{H^{-}(N)\right\}$ and  $\left\{H^{+}(N)\right\}$ agree outside the region $X^{\mathrm{in}}(N)$ for each $N$.  In this situation, the monotone homotopy may be defined in either direction, and Proposition \ref{prop:continuation_map_properties} immediately implies that the resulting equivariant Floer cohomologies coincide. 

More generally, if the Hamiltonian isotopy is sufficiently small, for instance when it is supported inside a Weinstein neighborhood of $L$, one  can also construct an inverse homotopy of such  an isotopy; see for example \cite[Lemma 3.21]{GPS1} and \cite[Proposition 5.6]{AA24}.

We formalize this as follows. 
\begin{cor}
    Let $\left\{H^{-}(N)\right\}$ and  $\left\{H^{+}(N)\right\}$ be admissible families of Hamiltonians compatible with $L_{0}$ and $L_{1}$, and suppose that $H_{-}(N)$ agrees with $H_{+}(N)$ outside a family of closed subspace $X^{\mathrm{in}}(N)$. Then the $G$-equivariant continuation map
    \begin{align*}
        F_{H^{-},H^{+}}^{G}: HF_{G}^{\bullet}(L_{0},L_{1}; \{H^{+}(N)\}) \rightarrow HF_{G}^{\bullet}(L_{0},L_{1}; \{H^{-}(N)\}), 
     \end{align*}
     is an isomorphism. 
\end{cor}
\begin{defn}
    For a single exact cylindrical Lagrangian $L\subset X$, we  define $HF_{G}^{\bullet}(L, L)$ to be the equivariant Floer cohomology computed using any  admissible family of Hamiltonians with sufficiently $C^2$-small slope. 
\end{defn}

 We now verify that equivariant Floer cohomology is independent of all other admissible auxiliary choices. By the preceding argument, it suffices to fix the slopes of Hamiltonians and vary the remaining data within admissible class.

Let $\{(f^{\pm}_{N},\cV^{\pm}_{N})\}$ be admissible families of Morse-Smale pairs, where each $f_{N}^{\pm}$ and  $\cV^{\pm}_{N}$ are constructed based on strictly increasing sequence of negative real numbers 
$$A^{\pm}\coloneqq  \diag\{a^{\pm}_{1},a_{2}^{\pm},\cdots \},$$
as shown in Theorem \ref{thm:family_of_Morse_functions}. 

Since the space of such sequences is convex, we may choose  a one-parameter family  $A_{\tau}$ of strictly increasing negative sequences such that $A_{\tau} = A^{-}$ for $\tau \ll 0$ and $A_{\tau }= A^{+}$ for $\tau \gg 0$. As in Theorem \ref{thm:family_of_Morse_functions}, this determines a corresponding family of Morse-Smale pairs $\{f_{\tau, N}, \cV_{\tau, N}\}$ on each $BG(N)$, interpolating  between $\{(f^{-}_{N},\cV^{-}_{N})\}$  and $\{(f^{+}_{N},\cV^{+}_{N})\}$ at the ends. 

The chain map between the associated Morse complexes is  defined by counting solutions $\gamma: \bR\rightarrow BG(N)$  to the gradient equation 
\begin{align*}
    \partial_{\tau}\gamma  = \cV_{\tau, N}.
\end{align*}
Each trajectory satisfies the same compatibility condition as in Theorem \ref{thm:family_of_Morse_functions}, since it is obtained by lifting of gradient vector fields of functions   \begin{align}
    f_{A_{\tau, N+k}}(V) \coloneqq  \mathrm{Re} \left( \Tr(A_{\tau, N+k} P_V) \right)    
    \end{align} defined on the Grassmannian  $\mathrm{Gr}(k, N+k)$. 

Therefore, we can define the moduli space  $\cM(\mathbf{x}_{-},\mathbf{x}_{+};X(N), ( \cV_{\tau, N})_{\tau \in \bR})$ similarly as in  Definition \ref{defn:moduli_space_continuation_map}. A similar argument to that in Lemma \ref{lem:approximation_commutative} shows that the corresponding chain map commutes with the projection map $P_{N}$, and hence induces a well-defined continuation map on equivariant Floer cohomology: 
\begin{align*}
    HF_{G}(L_{0},L_{1}; \{\cV^{+}_{N}\}) \rightarrow         HF_{G}(L_{0},L_{1}; \{\cV^{-}_{N}\}).
\end{align*}
The inverse continuation map is obtained by reversing the family $(\cV_{\tau, N})_{\tau \in \bR}$, implying the above map is an isomorphism. 

The invariance with respect to other auxiliary data, such as choice of almost complex structures and metrics can be proved in the same way. 

\subsubsection{Wrapped Floer cohomology as direct limit}
With these preparations in place, we are now ready to define the equivariant version of wrapped Floer cohomology. Let $L_{0}$ and $L_{1}$ be exact cylindrical Lagrangian submanifolds. We consider admissible families of liner Hamiltonians whose slopes are positive constants.

Fix an increasing sequence  of positive real numbers $\{s_{n}\}_{n\in \bN}\in \bR_{>0}$ satisfying  $$\lim_{n\rightarrow \infty} s_{n}= \infty,$$ and such that, for each $n$, $s_{n}$ is not equal to the length of any Reeb chord from $\partial_{\infty}L_{0}$ and $\partial_{\infty}L_{1}$ over any critical point. For a residual set of almost complex structure, we may therefore define 
\begin{align*}
    HF_{G}(L_{0},L_{1}; H_{s_{n}})
\end{align*}
as in Section \ref{sec:borel_construction_of_floer_cohomology}.   
Whenever $s_{n}\geq s_{m}$, the continuation map constructed in the previous section induces a morphism 
\begin{align*}
    HF_{G}(L_{0},L_{1}; H_{s_{m}}) \rightarrow HF_{G}(L_{0},L_{1}; H_{s_{n}}). 
\end{align*}
which makes the collection $\{HF_{G}(L_{0},L_{1}; H_{s_{n}})\}$ a direct system. 
\begin{defn}
The \textit{$G$-equivariant wrapped Floer cohomology} is defined as  
\begin{align*}
    HW_{G}(L_{0},L_{1})\coloneqq \varinjlim_{s_{n}} HF_{G}(L_{0},L_{1}; H_{s_{n}}).
\end{align*}
\end{defn}

\subsection{\texorpdfstring{$G$}{LG}-equivariant PSS isomorphism}
\label{sec:G_equivariant_PSS_isomorphism}
There is an isomorphism 
    \begin{align*}
        \mathrm{PSS}: H^{\bullet}(L)\rightarrow HF^{\bullet}(L,L). 
    \end{align*} 
    which is called Piunikhin-Salamon-Schwarz (PSS) map. 
    The modules $HF^{\bullet}(L,K)$ and $HF^{\bullet}(K,L)$ over it are unital. The aim of this subsection is to enhance the PSS isomorphism to the  following identification.  
\begin{thm}[$G$-equivariant PSS isomorphism]
\label{thm:G-equivariant_PSS_isomorphism}
For an embedded $G$-invariant exact cylindrical Lagrangian,  there is an isomorphism 
\begin{align*}
    H_{G}^{\bullet}(L;\mathbb{K}) \cong HF_{G}^{\bullet}(L,L)
\end{align*}
which we call the \textit{$G$-equivariant PSS isomorphism.}  
\end{thm}
\begin{proof}
    \hyperlink{proof:G-equivariant_PSS_isomorphism}{The proof is given at the end of this section. }
\end{proof}
In the case where the $G$-action is free, the equivariant and ordinary cohomologies coincide via the standard isomorphism $H_G^{\bullet}(L) \cong H^{\bullet}(L/G)$.  
Applying this identification to the equivariant Floer cohomology yields the following corollary, which is a special case of a more general statement in the next section. 
\begin{cor}
	\label{cor:floer_equivariant_isomorphism}
	Let $ L \subset \mu^{-1}_{X}(0)$ be an embedded, $G$-invariant, exact, cylindrical Lagrangian submanifold.  
	If the $G$-action on \(L\) is free, then there is an isomorphism
	\begin{align*}
		HF_{G}^{\bullet}(L, L) \;\cong\; HF^{\bullet}(L / G,\, L / G).
	\end{align*}
	\end{cor}

Recall that a \textit{local system} on a topological space $X$ is a functor from the fundamental groupoid $\pi_{1}(X)$ to the category of abelian groups (in our case, one may always assume the target is the category $\bK$-vector spaces). More explicitly, it assigns to each point of $X$ an abelian group $A_{x}$, and to every homotopy class $[\gamma]$ of paths a group homomorphism $\Phi_{[\gamma]}$ between the corresponding abelian groups, in a way compatible with the concatenation. We will refer $\Phi_{[\gamma]}$ as the \textit{parallel transport} along $\gamma$. 

If all the  $A_{x}$ is a fixed abelian group $A$, and $\Phi_{[\gamma]}$ are all identity maps, we say this is a \textit{constant local system with value $A$}. 

\begin{comment}
\begin{rem}
    In the present setting, we will consider cohomology with local coefficients in Morse or Floer cohomology groups. Here, the parallel transport is defined only along certain paths in the base. Following \cite{Oan08}, this is referred to as a \textit{local subsystem}. For our purposes, we will choose an extension to a local system over  the base that behaves coherently with respect to approximations as $N$ varies.
\end{rem}    
\end{comment}

Just as in the usual construction of the PSS isomorphism, we will employ a Morse model as an intermediate step to build the correspondence between Floer cohomology and singular cohomology. The exposition here follows closely \cite{Hutchings08, Oan08,SS10}. 

\subsubsection{The $E_{2}$-page of the Leray--Serre spectral sequence}
Following \cite{MaC01,Oan08}, we briefly recall the construction of the second page of Leray--Serre spectral sequence. Let $F\hookrightarrow E\xrightarrow{\pi} B$ be a locally trivial fibration such that $B$ is a CW-complex and both $B$ and $F$ are path-connected. The second page $LS_{2}^{p,q}$ of the Leray--Serre spectral sequence admits an explicit description as $$H^{p}(B;\cH^{q}(F)),$$ where $\cH^{q}(F)$ denotes the local system on $B$ whose fiber at a point $b\in B$ is $H^{q}(F)$. 

For a path $\gamma$ contained in a contractible open space, we define the morphism as $\Phi_{[\gamma]}\coloneqq  i_{\gamma(1)}^*(i_{\gamma(0)}^*)^{-1}$ where \begin{align*}
    F_{\gamma(0)}\xhookrightarrow{i(0)} \pi^{-1}(U)\xhookleftarrow{i(1)} F_{\gamma(1)} 
\end{align*}
are the inclusions. For a general path $\gamma$, choose a subdivision $0= t_{0}<\cdots < t_{n+1} =1$ such that each $\gamma|_{[t_{k},t_{k+1}]}$ is contained in a contractible open set. For $x,y \in [0,1]$, let $\Phi_{x,y}\coloneqq  \Phi_{\gamma|_{[x,y]}}$.  Define $\Phi_{[\gamma ]}\coloneqq  \Phi_{t_{n},t_{n+1}}\circ \cdots \Phi_{t_{0},t_{1}}$. For $|x-z|$ sufficiently small, we have $\Phi_{x,y}\circ \Phi_{y,z} = \Phi_{x,z}$. It follows that $\Phi_{[\gamma]}$ is independent of the choice of sufficiently fine partitions. The proof of the composition is similar.

\subsubsection{A Morse model}
In the proof of Theorem \ref{thm:family_of_Morse_functions}, we have constructed a  vector field $\cV_{\infty}$ on $BG$ by lifting the gradient on the base  and adding a vertical gradient component, and we realize this vector field as the pseudo-gradient stratum-wise. 

Here we present an alternative viewpoint, following \cite{Hutchings08}, by interpreting the construction in terms of a \emph{family of Morse functions} (not to be confused with the family of Morse-Smale pairs used earlier).

Our goal is to build a  Morse-theoretic model that realizes the  Leray--Serre spectral sequence associated to the fibration $$L\hookrightarrow L(N) \xrightarrow{\pi_{N}} BG(N).$$
 Choose a Morse-Smale pair $(f_{N},\cV_{N})$ on $0_{BG(N)}\subset T^*BG(N)$ as in Theorem \ref{thm:family_of_Morse_functions}. Next, we equip it with a smooth family of functions on the fibers of $f_{N}$, that is, a family of functions \begin{align}
    \label{equ:parametrized_Morse_functions}
    f_{b}: L\rightarrow \bR, \quad b\in 0_{BG(N)},
\end{align} parametrized by points $b\in 0_{BG(N)}$. Moreover, we assume that $f_{b}$ is locally constant as a family in a neighborhood of each $c \in \Crit(f_{N})$. 

In addition, we choose a corresponding family of Riemannian metrics so that, for each critical point $c\in \crit(f_{N})$,  the  corresponding fiberwise function $f_{c}$  is Morse. Recall that we have already chosen the metric on $BG(N)$ generically so that each pair $(f_{N},\cV_{N})$ is Morse-Smale, as in Theorem \ref{thm:family_of_Morse_functions}.  

Since each fiber is noncompact, we further require that $f_{c}$ is nongegative, linear near infinity, and has no critical point on the cylindrical end. Morse functions on $L$ satisfying these conditions automatically fulfill the  Palais-Smale Condition (C) (See Remark \ref{rem:PS_conditions}), ensuring that Morse theory is well-defined in the noncompact setting; because each $f_{c}$ is bounded below, its negative gradient flow is defined for all positive time and converges to a critical point; see \cite{MW10}. 

We denote the family of Morse functions by $\widehat{f}_{N}$ and the family of Riemannian metrics by $\widehat{g}_{N}$, which are required to be compatible with the system of approximations in an obvious way. 

As before, we shall choose $(\widehat{f}_{N}, \widehat{g}_{N})$ generically so that the parallel transport does not map a critical point to another critical point, ensuring the generic transversality conditions. For brevity, we will always omit the Riemannian metric from the notation. 

The fiberwise Morse cohomology is $\bZ$-graded over an algebraically closed field $\bK$ of characteristic $2$.

Now, define a bigraded cochain complex $CM^{\bullet,\bullet}(L(N); \widehat{f}_{N})$ as follows. The generators are given by 
\begin{equation}
    CM^{p,q}\coloneqq \{(c,\alpha)\mid c\in\crit^{p}(f_{N}), \alpha \in \crit^{q}(f_{c})\}.
\end{equation}

For generators $\mathbf{a}\coloneqq  (c_{i}, \alpha)$  and  $\mathbf{b}\coloneqq (c_{j}, \beta) $, we define the  unparametrized moduli space $  \cM_{\mathrm{Morse}}(\mathbf{a}, \mathbf{b}; \widehat{f}_{N})$  to consist of pairs $(\gamma, \varphi)$ of the form: 
 \begin{equation}
    \label{equ:Family_Morse}
     \cM_{\mathrm{Morse}}(\mathbf{a}, \mathbf{b}; \widehat{f}_{N}) \coloneqq   \left\{ \begin{aligned} & \gamma: \bR\rightarrow 0_{BG(N)}\\ & \varphi: \bR\rightarrow L(N) \end{aligned}  \hspace{2pt}\middle|\hspace{2pt} \begin{aligned}  & \partial_{\tau}\gamma - \cV_{N}=0  \\  & \frac{d\varphi^{\mathrm{vert}}}{d \tau}  + \nabla^{v} f_{{\gamma(\tau)}}(\varphi(\tau)) = 0 \\ & \pi_{N}\circ \varphi(\tau) = \gamma(\tau) \\  & \lim_{\tau\rightarrow -\infty} (\gamma, \varphi)(\tau) = (c_{i}, \alpha) \\
    & \lim_{\tau \rightarrow +\infty} (\gamma, \varphi)(\tau) = (c_{j}, \beta) \end{aligned} 
      \right\}  /\bR. 
\end{equation}
Here $d\varphi^{\mathrm{vert}}$ denotes the vertical component of $d\varphi$. The gradient vector field $\nabla^{v} f_{\gamma(\tau)}(\varphi(\tau))$ is taken fiberwise with respect to the chosen fiberwise Riemannian metric.

By trivializing the fibration along $\gamma$, we see that this definition recovers precisely the continuation maps associated with the family of functions $f_{\gamma(\tau)}$, where $f_{\gamma(\tau)} = f_{c_{i}}$ for $\tau\ll 0$ and $f_{\gamma(\tau)} = f_{c_{j}}$ for $\tau\gg 0$. The expected dimension of the moduli space is 
\begin{align*}
    \dim \cM_{\mathrm{Morse}}(\mathbf{a}, \mathbf{b}; \widehat{f}_{N}) &= |\mathbf{a}|-|\mathbf{b}|-1 
\end{align*}
For each $k \geq 0$, we define the component of the differential $\delta_{k}: CM^{p,q}\rightarrow CM^{p+k, q-k+1}$ by setting 
\begin{align*}
   \delta_{k}(\mathbf{b})= \sum_{\substack{|c_{i}|= p+k, \\|\alpha| = q-k+1 }}\# \cM_{\mathrm{Morse}}(\mathbf{a}, \mathbf{b}; \widehat{f}_{N})\mathbf{a}
\end{align*}
The total differential is $\delta = \sum_{k\geq 0}\delta_{k}$. A standard compactness and gluing argument shows that $\delta^{2} = 0$. We write the (family) Morse cohomology of the total complex $CM^{\bullet}$ by 
\begin{align*}
    HM^{\bullet}(L(N);\widehat{f}_{N})\coloneqq  H^{\bullet}\left(CM^{\bullet}(L(N); \widehat{f}_{N}),\delta\right). 
\end{align*}

\begin{rem}
Recall that a \emph{morphism of spectral sequences} $K$ is a collection of maps 
$K_{r}:E_{r}^{\bullet,\bullet}\rightarrow {}'E_{r}^{\bullet,\bullet}$ for each $r \geq 0$, such that $K_{r}$ commutes with the differentials. 
If $K_{r}$ is an isomorphism on the $E_{r}$-page for some $r$, then it induces isomorphisms on all subsequent pages $E_{s}$ for $s \geq r$ by five lemma. In particular, if the spectral sequences converge, then $K$ induces an isomorphism on the cohomology. 
\end{rem}

We have the filtration 
\begin{align}
\label{equ:filtration}
    F^p CM^{\bullet}\coloneqq  \left\{(c,\alpha)\mid |c|\geq p\right\}
\end{align}
 to be the subcomplex generated by pairs whose component $c$ has grading at least  $p$. 

 In the closed case, Hutchings \cite{Hutchings08} proved  that the corresponding spectral sequence agrees with the  Leray--Serre spectral sequence from the  $E_{2}$-page. In what follows, we adapt this argument to the case when the fiber is a  cylindrical Lagrangian. 

By our construction, the  Morse--Smale pair $(f_{N},\cV_{N})$ induces  a CW-complex decomposition of $BG(N)$ (cf. Remark \ref{rem:CW_complex} below). The  unstable (descending) manifold $ W^{u}(c) $ of  a critical point $c\in \crit(f_{N})$ admits a compactification with faces, whose $k$-stratum is given by 
\begin{align*}
    \overline{W^{u}(c)_{k}} = \coprod_{I} \cM_{I}\times W^{u}(r_{k}) , 
\end{align*}
where $$\cM_{I}\coloneqq \coprod_{c = r_{0}, r_{1},\cdots, r_{k}\text{ distinct} }\cM(r_{0},r_{1})\times \cdots \cM(r_{k-1}, r_{k})$$ for each sequence of critical points $I=\{r_{0},r_{1},\cdots, r_{k}\}$. The evaluation map $e: \overline{W^{u}(c)} \rightarrow BG(N)$ is smooth, and its restriction to $ \cM_{I}\times W^{u}(r_{k})$ coincides with  the projection onto $W^{u}(r_{k})$, while $e|_{W^{u}(c)}$ is the identity. Each unstable manifold $W^{u}(c)$ (resp. closure $\overline{W^{u}(c)}$) is homeomorphic to an open (resp. closed) disk of dimension $|c|$, with $e$ serving as the characteristic map. As a result, every unstable manifold closure is contained in finitely many cells, satisfying the axioms of a CW-complex.

\begin{rem}
\label{rem:CW_complex} As shown in \cite{Lau92,Qin10}, if a pseudo-gradient vector field agrees  with the genuine gradient vector field of a quadratic form with respect to the Euclidean metric in a neighborhood of each critical point, then the unstable manifolds define a CW decomposition of the underlying manifold as illustrated above. Such assumptions are referred to as \emph{special Morse functions} in \cite{Lau92} and as \emph{locally trivial metrics} in \cite{Qin10}.

In our case, this assumption is automatic and is incorporated into Definition~\ref{defn:pseudo_gradient_vector_field}. For the Morse functions $ \mathrm{Re} \left( \Tr(A_n P_V) \right) $ on the complex Grassmannian $\mathrm{Gr}(k,N)$, this property can be verified directly from the explicit expression of the gradient flow in \eqref{equ:Morse_flow}. In fact, the resulting CW decomposition coincides with the classical Schubert cell decomposition, as shown in Example \ref{ex:Schubert_cell}. When lifting a gradient vector field on the base to a pseudo-gradient vector field on the total space of a fibration, we always have sufficient flexibility in the choice of Morse functions and Riemannian metrics in the fiber direction. 
\end{rem}

 The existence of CW-complex structure relies on the compactness of each $BG(N)$. When working with a general manifold $M$, the situation changes substantially.  For a Morse-Smale pair $(f,\cV)$ on $M$, let  $M^a \coloneqq f^{-1}\left((-\infty,a]  \right)$ and $    K^{a}\coloneqq \bigsqcup_{f(p)\leq a} W^{u}(p)$.  Assuming that $f$ is bounded below and satisfies Palais--Smale condition (C), one only obtains a CW decomposition of the sets $K^{a}$, because the flow is only forward complete. 

 However, in our case, where $M=L$ or its Borel space $L(N)$ is a cylindrical Lagrangian,  the Morse function is linear near infinity, and therefore proper and bounded below (also known as exhaustive). There  is a homotopy equivalence between $K^{a}$ and $M^{a}$ (see, for example, \cite[Theorem 3.8]{Qin10}). Moreover, the vector fields that appear in the family of Morse theory can also be  realized as  pseudo-gradient vector fields for suitable Morse functions. By \cite[Theorem 2.8]{AM06}, this implies a further homotopy equivalence $M^a \cong M$,
whenever $a$ is taken larger than the maximal critical value of $f$. 

Therefore, even though $L$ is noncompact, it suffices to compare the Morse cohomology of $\widehat{f}_{N}$ with the  cellular cohomology of the underlying finite CW-complex, which is a standard identification. This reduction relies crucially on the fact that we only work with Liouville manifolds of \emph{finite type} (in fact, only those obtained by completions of Liouville domains). The result may fail depending on one's definition of Liouville manifolds. 

\begin{prop}[Morse-theoretic Leray--Serre spectral sequence]
\label{prop:comparison_between_Morse_and_LS}
Let $(L(N),\widehat{f}_{N})$ be defined as above, and let $\pi_{N}: L(N)\rightarrow BG(N)$ denote the projection in the Borel construction. Let $\cE_{k}^{p,q}$ denote the spectral sequence associated to the filtration \eqref{equ:filtration}.  Then the following hold:
\begin{enumerate}
    \item The Morse-theoretic spectral sequence $\cE_{k}^{p,q}$ is isomorphic to the Leray--Serre spectral sequence $LS_{k}^{p,q}$ for $k\geq 2$. 
    \item The family Morse cohomology $HM^{\bullet}(L(N); \widehat{f}_{N})$ is isomorphic to the cellular cohomology of the total space, that is, $$HM^{\bullet}(L(N); \widehat{f}_{N})\cong H_{\mathrm{Cell}}^{\bullet}(L(N)),$$  and hence to the singular cohomology $H^{\bullet}(L(N))$. 
\end{enumerate}
\end{prop}

\begin{proof}
Since each fiberwise Morse function $f_{c}$ satisfies the Palais-Smale condition (C),  the argument of \cite{Hutchings08} applies verbatim in our case. Here we give an alternative proof based on the CW-complex decomposition; a closely related approach can be found in \cite{Oan08}.  

The vector field defining the moduli space \eqref{equ:Family_Morse} is given by 
\begin{align*}
    \cV_{L(N)}\coloneqq  \cV_{N}^{\#}  -\nabla^{v} f_{\gamma(\tau)},
\end{align*}
where $\cV_{N}^{\#}$ denotes the horizontal lift of the negative pseudo-gradient vector field  $\cV_{N}$. We now show that this vector field can also be realized as the pseudo-gradient of a suitably constructed Morse function $F_{N}$. 

Note that the argument of Lemma \ref{lem:pseudo_gradient_Lemma} does not apply directly, as each fiber is noncompact and we would need a uniform bound on the fiberwise Morse function $|f_{c}|$ and its gradient $\norm{\nabla^{v}f_{c}}$ to ensure that  it will not break  the pseudo-gradient properties and no new critical points are generated. However, because each fiber $L$ admits  a cylindrical structure, we can borrow some ideas from Floer theory and resolve this issue in a soft manner.  

As before, choose a subspace $L^{\mathrm{in}}(N)\subset L(N)$, proper over the base, that it contains all the critical points $$\prod_{c_{i}\in \crit(f_{N})}\crit(f_{c_{i}}).$$ Outside a collared neighborhood of  $L^{\mathrm{in}}(N)$, choose a smooth function $F_{\mathrm{out}}$  such that $dF_{\mathrm{out}}(\cV_{L_{N}})<0$.  This is possible because  each fiberwise  function $f_{b}$ is linear near infinity; indeed, any positive multiple of the radial coordinate $r$ suffices. (Note that  $r$ is globally defined outside $L^{\mathrm{in}}(N)$ since it is $G$-invariant). 

In the same notation as in Lemma \ref{lem:pseudo_gradient_Lemma}, we define a  function on $L(N)$ by 
\begin{align}
\label{equ:Morse_function_wrt_pseudo_gradient}
    F_{N} \coloneqq  \begin{dcases*}
        \pi_{N}^*f_{N} + \sum_{c_{i}\in \crit(f_{N})}\epsilon_{i}\phi_{i}f_{c_{i}} &   in $ L^{\mathrm{in}}(N)$  \\
        F_{\mathrm{out}}   & outside a neighborhood of $L^{\mathrm{in}}(N)$
    \end{dcases*}
\end{align}
where the two definitions are interpolated by a smooth cutoff function on a collared neighborhood of $\partial L^{\mathrm{in}}(N)$, so that $F_{N}$ is a smooth function. 

Since each $\phi_{i}$ is constant near the critical point $c_{i}$, for sufficiently small  $\epsilon_{i}$, $F_{N}$ is a Morse function and its critical points coincide with the generators  of the complex $CM^{\bullet}(L(N); \widehat{f}_{N})$.  To verify the pseudo-gradient property, observe that inside $L^{\mathrm{in}}(N)$ we have  
\begin{align}
    \label{equ:pseudo_gradient_vector_field}
    dF_{N}( \cV_{L(N)})=  df_{N}( \cV_{N}) - \sum_{c_{i}\in \crit(f_{N})}\epsilon_{i} d\phi_{i}( \cV_{N})f_{c_{i}} +  \sum_{c_{i}\in \crit(f_{N})} \epsilon_{i}\phi_{i}df_{c_{i}}(- \nabla^{v} f_{\gamma(\tau)}). 
\end{align} 
Since $\cV_{N}$ is a pseudo-gradient adapted to $f_{N}$, and the family $f_{b}$ is locally constant  around $c_{i}$, the first and third term are both non-positive. Inside $L^{\mathrm{in}}(N)$, the second term is also bounded. The rest of the proof is identical to that in Lemma \ref{lem:pseudo_gradient_Lemma}.

As discussed above, there is a homotopy equivalence  \begin{align*}
    K(N)^{a} \coloneqq \bigsqcup_{F_{N}(p)\leq a} W^{u}(p) \cong L(N)^{a} \cong L(N),
\end{align*}
when $a$ is chosen larger than all critical values of $F_{N}$.

The resulting CW-complex has the property that the projection of each $k$-cell corresponds to a $k'$-cell in the base, since  by construction, the projection of the unstable manifold of $\cV_{L(N)}$ is precisely the unstable manifold of $\cV_{N}$. Hence, the associated cellular complex \begin{align*}
    \mathrm{Cell}^k(K(N)^{a})\coloneqq  \bigoplus_{i}\bK \langle e_{i}^k \rangle
\end{align*}
  admits a natural filtration 
\begin{align*}
    F^{p} \mathrm{Cell}^k(K(N)^{a}) = \bigoplus_{\dim \pi(e_{i}^k)\geq p}\bK \langle e_{i}^k\rangle,
\end{align*}
and the corresponding spectral sequence is a cellular model for the Leray--Serre spectral sequence (see,  for example, \cite{MaC01}). We remark that the compactness of the base 
$BG(N)$  ensures that the unstable manifolds of $(f_{N},\cV_{N})$ assemble into a finite CW–complex structure.

By the functoriality of the Leray--Serre spectral sequence and the homotopy equivalences established above, it therefore suffices to work with the Leray--Serre spectral sequence associated to the finite CW-complex $K(N)^{a}$.

We define a map from the Morse complex $CM^{\bullet}$ to the cellular complex $ \mathrm{Cell}^k(K(N)^{a})$ by sending each  generator $\mathbf{a}\coloneqq  (c,\alpha)$ to the corresponding unstable manifold, which can be regarded as a basis of the cellular cohomology. That this map is a chain morphism follows from the classical principle ``Morse cohomology  is the cellular cohomology''. We briefly recall the construction  here (see, for example,  \cite{Oan08, Qin10,AD14}).

Let $\cB^k$ denote the $k$-skeleton of the CW-complex. Then the  cellular cochain complex $\mathrm{Cell}^k(L(N)) $ is given by $H^k(\cB^k,\cB^{k-1};\bK) $, with the differential $\partial_{\mathrm{Cell}}$ defined as the  connecting homomorphism in the long exact sequence 
\begin{align*}
    H^k(\cB^k,\cB^{k-1};\bK) \rightarrow H^k(\cB^k;\bK)\rightarrow H^{k+1}(\cB^{k+1},\cB^{k};\bK). 
\end{align*}
Explicitly, one has 
\begin{align}
    \label{equ:cell_complex}
    \partial_{\mathrm{Cell}}(e_{i}^k)= \sum_{j}[e_{i}^k:e_{j}^{k+1}]e_{j}^{k+1}
\end{align}
where $[e_{i}^k:e_{j}^{k+1}]$ denotes the incidence number. 

These incidence numbers satisfy $[W^u (\mathbf{b}), W^u (\mathbf{a})]= \#\cM(\mathbf{a},\mathbf{b};\widehat{f}_{N})$. Furthermore, by the construction of $(f_{N},\cV_{N})$ (Definition \ref{thm:family_of_Morse_functions}) and the assumption that each $f_{c_{i}}$ is a Morse function that is linear near infinity, the right-hand side of \eqref{equ:cell_complex} is  a finite sum. This correspondence respects the filtration tautologically and induces an isomorphism of the cochain complex.

Therefore, the Morse spectral sequence coincides with the Leray--Serre spectral  sequence from the $E_{2}$-page (all models of Leray--Serre spectral sequence are isomorphic from $E_{2}$-page). 
\end{proof}

\subsubsection{From Floer to Morse}
Similarly we consider  the spectral sequence $E_{\bullet}^{p,q}$ associated with the filtration 
\begin{align*}
    F^p CF^{\bullet}(L,L;X(N))\coloneqq  \{(c, x)\mid |c|\geq p \}.
\end{align*}

To compare the spectral sequence $E_{\bullet}^{p,q}$ with the Leray--Serre spectral sequence, by  Proposition \ref{prop:comparison_between_Morse_and_LS} it suffices  to compare $E_{\bullet}^{p,q}$ with the Morse-theoretic spectral sequence $\cE_{\bullet}^{p,q}$  associated with the filtration \eqref{equ:filtration} for the family of Morse function $(L(N),\widehat{f}_{N})$. 
\begin{thm}[Comparison with Leray--Serre]
\label{thm:comparison_between_spectral_sequences} 
Let $\pi_{N}: L(N)\rightarrow 0_{BG(N)}$ be the fibration constructed before. Then: 
\begin{enumerate}
    \item There is an isomorphism from the $E_{2}$-page onward between the spectral sequence $E_{\bullet}^{p,q}$ and the Morse-theoretic Leray--Serre spectral sequence $\cE_{\bullet}^{p,q}$. 
    \item The  Floer cohomology $HF^{\bullet}(L,L;X(N))$ is isomorphic to the family Morse cohomology $HM^{\bullet}(L(N);\widehat{f}_{N})$.   
\end{enumerate}
\end{thm} 
\begin{proof}
The proof follows from an adaptation of the standard pearl trajectory argument (see, for example, \cite{Alb08, BC09}) to the setting of a fibration, see Figure \ref{fig:pearl_trajectory}. For generators of the Lagrangian Floer complex associated to the Morse function $f_{N}$, we write 
\begin{align*}
   \mathbf{x} \coloneqq  (c_{i}, x) & \in CF^{\bullet}(L_{0},L_{1};X(N)),  \\
   \mathbf{a} \coloneqq  (c_{j}, \alpha) & \in CM^{\bullet}(L(N); \widehat{f}_{N}).
\end{align*}
We define the moduli space $$\cP\left(\mathbf{a}, \mathbf{x}; H(N), \widehat{f}_{N}\right)$$ to consist of triples $(\gamma, \varphi,u)$, where $(\gamma, \varphi)$ is a family of Morse trajectory and $u$ is a pseudo-holomorphic strip satisfying the following conditions: 
\begin{itemize}
    \item The family Morse trajectory is a map 
    \begin{align*}
        \varphi: (-\infty,0] \longrightarrow L(N)
    \end{align*}
    together with a projection $\gamma:(-\infty,0] \rightarrow BG(N)$, and $(\gamma,\varphi)$ solves the fiberwise Morse equation \eqref{equ:Family_Morse} associated to the family $\widehat{f}_{N}$. 
    \item The pseudo-holomorphic component $u$   
    \begin{align*}
        u:\bR\times[0,1]\rightarrow L(N)
    \end{align*}
    is a strip mapping into $L(N)$, which also projects to $\gamma$,   that satisfies the inhomogeneous Cauchy-Riemann equation 
    \begin{align*}
        \partial_{\tau}u^{\mathrm{vert}} + J_{(\tau,t)}^{\mathrm{vert}}\left(\partial_{t}u^{\mathrm{vert}} - \psi(\tau)X_{H_{(\tau)}}\right) = 0,
    \end{align*}
    where $\psi:\bR\to [0,1]$ is a smooth cutoff function such that $\psi(\tau)=0$ for $\tau\ll 0$ and $\psi(\tau)=1$ for $\tau\gg 0$. 
\end{itemize}
Thus, $\psi(\tau)$ interpolates between the Cauchy--Riemann equation $\overline{\partial}_{J}u=0$ and the inhomogeneous Floer equation \eqref{equ:vertical_Floer_equation}.  They are required to satisfy following boundary and asymptotic conditions: 
\begin{align}
    \begin{dcases}
        u(\tau,i)  \in L(N),&  \tau \in \bR, i= 0,1;  \\
        \lim_{\tau \rightarrow +\infty}u(\tau, t) = \mathbf{x};\\
        \lim_{\tau \rightarrow -\infty}u(\tau, t) = \varphi(0) ; \\
        \lim_{\tau \rightarrow -\infty} (\gamma, \varphi)(\tau) = \mathbf{a}. 
    \end{dcases}
\end{align}

\begin{figure}[ht]
	\includegraphics[scale=0.6]{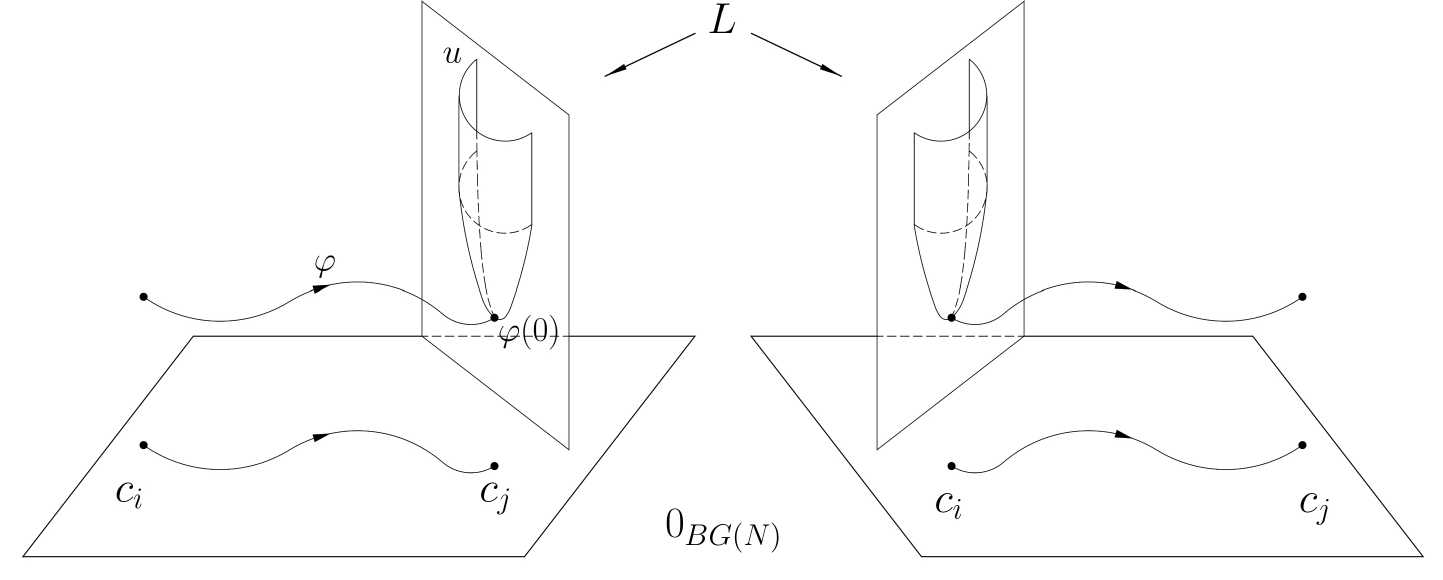}
	\caption{Configurations of pearl trajectories that define the maps $\mathrm{PSS}_{N}$ and $\mathrm{PSS}_{N}^{-1}$.}
	\label{fig:pearl_trajectory}
\end{figure}
Up to generic choices, $\cP\left(\mathbf{a},\mathbf{x};H(N),\widehat{f}_{N} \right)$ is a smooth manifold.
By our choice of Hamiltonians, we have $E(u)<\infty$, and hence Gromov-Floer compactness applies.  Consequently, we can compactify the one-dimensional stratum of $\cP\bigl(\mathbf{a},\mathbf{x};H(N),\widehat{f}_{N}\bigr)$ by adding either a  Morse trajectory or a Floer strip. The boundary strata admit the explicit description
\begin{align}
\label{equ:boundary_of_pearl_trajectory}
\begin{split}
    \partial_{[1]} \overline{\cP}\left(\mathbf{a},\mathbf{x};H(N),\widehat{f}_{N} \right) = &  \bigcup_{\mathbf{b}\in CM^{\bullet}} \cM_{\mathrm{Morse}}(\mathbf{a}, \mathbf{b}; \widehat{f}_{N}) \times  \cP\left(\mathbf{b},\mathbf{x};H(N),\widehat{f}_{N}\right) \\ 
& \cup \bigcup_{\mathbf{y}\in CF^{\bullet}}    \cP\left(\mathbf{a},\mathbf{y};H(N),\widehat{f}_{N} \right)\times  \cM_{\mathrm{Morse}}(\mathbf{y}, \mathbf{b}; \widehat{f}_{N}).
\end{split}
\end{align}
Now define $\mathrm{PSS}_{N}:  CF^{\bullet}(L_{0},L_{1};X(N)) \rightarrow CM^{\bullet}(L(N); \widehat{f}_{N})$ by 
\begin{align*}
    \mathrm{PSS}_{N} (\mathbf{x})\coloneqq  \sum_{\substack{\mathbf{a} = (c_{i}, \alpha) \\ |\mathbf{a}|=  |\mathbf{x}|}  }\# \cP\left(\mathbf{a},\mathbf{x};H(N),\widehat{f}_{N} \right) \mathbf{a} .
\end{align*}
It follows from \eqref{equ:boundary_of_pearl_trajectory} that $\mathrm{PSS}_{N}$ is a chain map. To see that $\mathrm{PSS}_{N}$ respects the filtration, note that there is a natural projection $\cP\left(\mathbf{a},\mathbf{x};H(N),\widehat{f}_{N} \right) \rightarrow \cM_{\mathrm{Morse}}(c_{i},c_{j};f_{N})$ which implies that whenever 
$\cP\left(\mathbf{a}, \mathbf{x}; H(N), \widehat{f}_{N}\right)\neq\emptyset$, 
we have $|c_{i}|\geq |c_{j}|$, since $(f_{N},\cV_{N})$ is Morse-Smale.  On the $E_{2}$-page, it induces a   morphism between the spectral sequence
\begin{align}
    \label{equ:E_2_page_Comparison}
   E_{2}^{p,q} \xrightarrow{[\mathrm{PSS}_{N}]} \cE_{\bullet}^{p,q}.
\end{align}
The inverse morphism $\mathrm{PSS}_{N}^{-1}$ is defined similarly: 
the pearl trajectory now starts from a Hamiltonian chord and connects to a Morse critical point via a negative gradient flow line. 
In this case, the half–disk solves the Floer equation
    \begin{align*}
        \partial_{\tau}u^{\mathrm{vert}} + J_{(\tau,t)}^{\mathrm{vert}}\left(\partial_{t}u^{\mathrm{vert}} - \psi(-\tau)X_{H_{(\tau)}}\right) = 0,
    \end{align*}

To conclude the proof, we apply a standard gluing and homotopy argument (see \cite{Alb08}) to show the compositions $\mathrm{PSS}_{N}^{-1}\circ \mathrm{PSS}_{N}$ and $\mathrm{PSS}_{N}\circ \mathrm{PSS}_{N}^{-1}$ are homotopic to the identity respectively.

One convenient way to reduce the fibration setting to the classical case is to trivialize the pearl trajectory over the Morse trajectory of $f_{N}$.  
Recall that the connection induced by the symplectic form $\omega_{X(N)}$ defines a parallel transport that preserves the Lagrangian submanifolds (see Section \ref{rem:Borel_Lagrangians}).  
\begin{comment}
    The only subtlety here is that the Morse component of the pearl trajectory is governed not by a genuine gradient flow but by a pseudo-gradient flow, as discussed in the proof of Proposition \ref{prop:comparison_between_Morse_and_LS} (see the Morse function \eqref{equ:Morse_function_wrt_pseudo_gradient} and the pseudo-gradient vector field \eqref{equ:pseudo_gradient_vector_field}).  
Nevertheless, the argument of \cite{Alb08} carries over verbatim here, since the pseudo-gradient vector field is in the same direction with the true gradient vector field. 
\end{comment}
Since $L$ is exact, we conclude that \eqref{equ:E_2_page_Comparison} is an isomorphism, with inverse given by $\mathrm{PSS}_{N}^{-1}$. 
\end{proof}

\begin{proof}[\hypertarget{proof:G-equivariant_PSS_isomorphism}{Proof of Theorem \ref{thm:G-equivariant_PSS_isomorphism}}]
    Similar to the argument of Lemma \ref{lem:approximation_commutative}, the family Morse cohomology $HM^{\bullet}(L(N);\widehat{f}_{N})$ forms  an inverse system under the projection maps $P_{N}$ induced by the inclusions. 

    In view of the proof of  Theorem \ref{thm:comparison_between_spectral_sequences},  an argument parallel to Lemma \ref{lem:approximation_commutative} gives rise to the following commutative diagram:
    \begin{equation}
        \label{equ:PSS_commutative}
         \begin{tikzcd}
             HF^{\bullet}(L, L; X(N+1)) \arrow[r, "P_{N+1}"] \arrow[d, "PSS_{N+1}"'] &  HF^{\bullet}(L, L; X(N)) \arrow[d, "PSS_{N}"] \\
            HM^{\bullet}(L(N+1);\widehat{f}_{N+1}) \arrow[r, "P_{N+1}"] & HM^{\bullet}(L(N);\widehat{f}_{N}) 
         \end{tikzcd}.
     \end{equation} 
Taking the inverse limit over $N$, we obtain $$HF^{\bullet}_{G}(L,L) \cong  \varprojlim_{N}HM^{\bullet}(L(N);\widehat{f}_{N}).$$ Thus it remains to verify the isomorphisms $$ HM^{\bullet}(L(N);\widehat{f}_{N}) \cong H^{\bullet}(L(N))$$ are  functorial  with respect to the approximation. 

Recall that in the proof of Proposition \ref{prop:comparison_between_Morse_and_LS}, this isomorphism was constructed by passing through the cellular cohomology of the CW-complex structure associated to $\widehat{f}_{N}$. The inclusion $L(N)\hookrightarrow L(N+1)$ induces a morphism $H^{\bullet}_{\mathrm{Cell}}(L(N+1))\rightarrow H^{\bullet}_{\mathrm{Cell}}(L(N))$, and the following diagram commutes 
    \begin{equation}
        \label{equ:Morse_commutative}
         \begin{tikzcd}
            HM^{\bullet}(L(N+1);\widehat{f}_{N+1}) \arrow[r, "P_{N+1}"] \arrow[d] &  HM^{\bullet}(L(N);\widehat{f}_{N}) \arrow[d] \\
            H^{\bullet}_{\mathrm{Cell}}(L(N+1)) \arrow[r] & H^{\bullet}_{\mathrm{Cell}}(L(N))
         \end{tikzcd}.
     \end{equation} 
This commutativity is tautological,  since the cellular differential \eqref{equ:cell_complex} is exactly the Morse differential by construction. Therefore, we conclude that 
\begin{align*}
    HF^{\bullet}_{G}(L,L)  \cong  \varprojlim_{N}H^{\bullet}_{\mathrm{Cell}}(L(N)) \cong H_{\mathrm{Cell}}^{\bullet}(L_{G};\bK) \cong  H_{G}^{\bullet}(L;\bK). 
\end{align*}
The last two isomorphisms are standard. 
\end{proof}
\section{Equivariant Fukaya categories}
\label{sec:equivariant_Floer_categories} 
\subsection{A \texorpdfstring{$G$}{LG}-equivariant directed category}
We shall now extend the discussion from the previous section to an $A_{\infty}$-category. The construction follows an approach analogous to the \emph{fiberwise wrapped Fukaya category} developed by Abouzaid and Auroux \cite{AA24}. However, due to the structure provided by the Borel--Liouville  construction, our situation is technically more accessible than in \emph{loc. cit.}

To proceed, we reformulate part of the setup from the previous section in a more compact form for later reference. 
\begin{lemd}[cf. \cite{AA24}]
\label{lemd:fiberwise_wrapped_data}
    Let $X(N)\xrightarrow{\pi_{N}} T^*BG(N)$ be the Borel--Liouville  construction of a  manifold. Fixed a closed subspace $X^{\mathrm{in}}(N)\subset X(N)$ which is proper over the base. Then one can choose a triple $(H,h,J_{N})$ consisting of: 

    \begin{enumerate}
        \item   an almost complex structure $J_{N}$  on $X(N)$, such that its restriction to each fiber is   compatible with the symplectic form, and  $\pi_{N}$ is $J_{N}$-holomorphic; 
        \item  a $J_{N}$-plurisubharmonic function $h:X(N)\rightarrow [0,\infty)$;
        \item a non-negative Hamiltonian $H: X(N) \rightarrow [0,\infty)$ called the  \textit{wrapping Hamiltonian}.
    \end{enumerate}
    These data are required to satisfy: 
    \begin{enumerate}[resume]
        \item The restrictions of $h$ and $H$ to each fiber of $\pi_{N}$ are proper. The  intersection of $X^{\mathrm{in}}$ with any fiber is a  sublevel set of $h$.   
        \item The Hamiltonian vector field $X_{H}$ is vertical, i.e. $d\pi_{N}(X_{H}) = 0$.  The parallel transport preserves the level sets of $h$ outside $X^{\mathrm{in}}(N)$, i.e.  $$dh(X_{H}) = 0\quad \text{ outside }  X^{\mathrm{in}}(N). $$ 
        \item If $\xi^{\#}$ is the horizontal lift of a vector on $T^*BG(N)$, then $dH(\xi^{\#}) = 0 $, and 
        \begin{align*}
            dh(\xi^{\#})= 0, \quad d^c h(\xi^{\#}) = 0\quad  \text{outside }  X^{\mathrm{in}}(N)
        \end{align*}
        where $d^c h \coloneqq  -dh\circ J_{N}$. 
        \item The $1$-form $d^c h$ is invariant under both flows near infinity: $\cL_{X_{H}}d^c h =\cL_{\xi^{\#}}d^c h  = 0$ away from $X^{\mathrm{in}}(N)$. 
        \item{(Non-negative wrapping)} $d^c h(X_{H})\geq 0$ outside $X^{\mathrm{in}}(N)$. 
    \end{enumerate}
    Finally, the choices $(H,h, J_{N})$ are required to be compatible with the approximation as $N$ varies. 
\end{lemd}
\begin{rem}
When $X$ is a Liouville sector, in addition to the requirements of Lemma \ref{lemd:fiberwise_wrapped_data}, we also impose a boundary compatibility condition. Namely,  near  $\pi_{\partial X}$ the product decomposition induces a projection  $$\left(\mathrm{Nbh}(\partial X), J_{N}\right)\rightarrow (\bC_{\mathrm{|Re|\geq 0}},j),$$ which we require to be  $(J,j)$-holomorphic. In the Borel construction, by a standard partition of unity argument, this extends to a projection
\begin{align*}
    \pi_{\partial X(N)}: \mathrm{Nbh}(\partial X(N))\rightarrow \bC_{\mathrm{|Re|\geq 0}}.
\end{align*}
As shown  in \cite{GPS1}, there is abundant freedom in the  choice of almost complex structures so that the vertical holomorphicity condition 
\begin{align*}
     d\pi_{\partial X(N)}|_{TX(N)^{\mathrm{vert}}}\circ J_{N} = j\circ d\pi_{\partial X(N)}|_{TX(N)^{\mathrm{vert}}}
\end{align*}
is satisfied. This ensures that the standard  maximum principal applies, thereby preventing the pseudo-holomorphic strips from escaping cross the boundary. 
\end{rem}

To construct a wrapping Hamiltonian $H$, we require it to be compatible with the $G$-action, in the sense that $\cL_{X_{\xi}}X_{H} =0$ for $\xi \in \mathfrak{g}$. In the case where $X$ is a Liouville manifold, such a Hamiltonian can be constructed in a straightforward manner: one begins with a proper linear Hamiltonian function and and averages it over the group $G$. The following lemma shows that a similar  construction extends to Liouville sectors, and moreover produces a  cofinal wrapping sequence of Lagrangians.
\begin{lemma}[$G$-invariant wrapping]
    \label{lem:G_equivariant_wrapping}
    Let $X$ be a Liouville sector with compact contact boundary $\partial_{\infty} X$. Then there exists a $G$-invariant, non-negative Hamiltonian function $$H:X\rightarrow \bR_{\geq 0 }$$
    such that:
    \begin{enumerate}
    \item $H$ is a $G$-invariant linear Hamiltonian that vanishes  to second order precisely on $\partial X$;
    \item the Hamiltonian vector field $X_H|_{X^{\circ}}$ is complete;
    \item the action of $G$ preserves the Hamiltonian vector field $X_{H}$;
    \item the Hamiltonian flow lies in the levels sets of $\mu_{X}$. 
    \end{enumerate}
    \end{lemma}
    
    \begin{proof}
    We begin by choosing a proper nonnegative function that is linear on the cylindrical end and vanishes  to second order along $\partial X$. Average this function over $G$, we obtain a nonnegative Hamiltonian function $H$.

    For property (1), note that because the $G$-action is a local diffeomorphism, it preserves the boundary $\partial X$. Since $G$ preserves the Liouville structure, it also preserves the radial coordinate; hence $H$ is a $G$-invariant linear Hamiltonian that vanishes precisely on $\partial X$.

    By construction, $H|_{X^\circ}$ is strictly positive and proper. Furthermore, since $dH(X_H)=0$, the Hamiltonian flow preserves the level sets of $H$. In particular, the vector field $X_H$ cannot drive points away from compact sets in finite time. Escape lemma then implies that $X_H$ is a complete vector field on $X^\circ$, establishing (2). 

     To verify (3) and (4), let $\xi \in \mathfrak{g}$ and denote by $X_{\xi}$ the infinitesimal vector field. Taking the Lie derivative of defining relation $\omega(X_{H},\cdot) = -dH$, we obtain $\cL_{X_{\xi}}X_{H} =0$. Let $\phi_{H}^t $ denote the Hamiltonian flow of $X_{H}$. Recall  that the moment map of a Liouville $G$-action is given by $\langle \mu_{X}(x),\xi  \rangle = \lambda(X_{\xi})$ (See Proposition \ref{prop:Group_action_1_form}).
     We compute for   $x\in X$ and $\xi \in \mathfrak{g}$:
     \begin{align*}
          \langle \mu_{X}(\phi^t_{H}(x)),\xi\rangle & = \lambda\left(X_{\xi}(\phi^{t}_{H}(x))\right)\\
         & = ((\phi^{t}_{H})^*\lambda)(X_{\xi})\\
         & = \lambda (X_{\xi})\\
         & = \langle \mu_{X}(x),\xi\rangle,
     \end{align*}
     where the second equality follows from the commutation  $[X_{H},X_{\xi}] = 0$ established in (3). 
     Thus we have shown that $\mu_{X}$ is invariant under the Hamiltonian flow of $X_{H}$. Clearly, the same argument applies verbatim in the Liouville manifold case by taking $\partial X= \emptyset$. 
     \end{proof} 

\begin{proof}[Proof of Lemma \ref{lemd:fiberwise_wrapped_data}]
The family of almost complex structures has  already been constructaed in Section \ref{sec:almost_complex_structure}. Choose  a $J_{N}$-plurisubharmonic function $h: X(N)\rightarrow \bR$ which agrees with the radial coordinate $r$ outside $X^{\mathrm{in}}(N)$.  Let $H$ be a G-invariant non-negative function chosen as in Lemma \ref{lem:G_equivariant_wrapping}; it desecnds to a Hamiltonian function on $X(N)$ through the quotient map 
\begin{align}
    \label{equ:quotient_map_used_in_proof}
    \left(X\times T^*EG(N), d\lambda_{X}+d\lambda_{\mathrm{can}}\right) \longrightarrow( X(N), d\lambda_{X(N)}). 
\end{align}
The verification of the conditions in Lemma~\ref{lemd:fiberwise_wrapped_data} proceeds as follows.  

Condition (4) holds  since both $H$ and $h$ are linear near infinity. The projection $\pi_{N}$ is induced by  the projection to the second factor in \eqref{equ:quotient_map_used_in_proof}, so that $d\pi_{N}(X_{H}) = 0$. By construction $h$ agrees with $r$ outside $X^{\mathrm{in}}(N)$, and $r$ is invariant under the parallel transport away from $X^{\mathrm{in}}(N)$, hence Condition (5) follows.  To see (6), recall that the connection we use is induced by the symplectic form $\omega_{X(N)}$ so that $$dH(\xi^{\#})  = -\omega_{X(N)}(X_{H}, \xi^{\#})=0.$$ The same results apllies to $dh$  when restricted  away from $X^{\mathrm{in}}(N)$. Finally,  both conditions (7) and (8) follow immediately from the assumption that $J_{N}$ is fiberwise of contact type near infinity. 

\end{proof}

Let us fix a countable collection of exact cylindrical $G$-invariant Lagrangians 
\begin{align*}
    \mathbf{L} \coloneqq  \{L_{0}, L_{1}, L_{2}, \ldots \},    
\end{align*}
which we refer to as \textit{admissible Lagrangians}. 
Applying the  Liouville  Borel construction  to these objects yields a  corresponding countable collection of Borel Lagrangian spaces
\begin{align*}
    \mathbf{L}_{N} \coloneqq  \{L_{0}(N), L_{1}(N), L_{2}(N), \ldots\}.    
\end{align*}
In analogy with  the construction of equivariant Floer cohomology, our goal is to associate to each $N \in \mathbb{N}$ an $A_{\infty}$-category, and to show that the approximation functors arising from varying the level $N$ are compatible with the $A_{\infty}$-structures.  

\subsubsection{Trees and Disks}
To this end, we study  a hybrid object consisting of Morse trees and  pseudo-holomorphic disks that solve fiberwise Cauchy-Riemann equations with inhomogeneous perturbation terms. As before, each such disk lies over the  corresponding tree.

We begin by recalling the following definition. 
\begin{defn}
A \textit{directed $d$-leafed planar tree} $T$ is a driected tree embedded in $\bR^2$ consisting of the following data: 
\begin{enumerate}
    \item a set of vertices $V(T)$ and a set of edges $E(T)$; 
    \item a collection of $d$ semi-infinite outgoing edges, referred to as leaves; 
    \item a distinguished vertex $v_{0}\in V(T)$, called the \textit{root}, which is connected to a single semi-infinite incoming edge; 
    \item a decomposition $E(T) = E^{\mathrm{ext}}(T)\sqcup E^{\mathrm{int}}(T)$, where $E^{\mathrm{ext}}(T)$ consists of semi-infinite (external) edges, and $E^{\mathrm{int}}(T)$ consists of the (internal) internal edges; 
 \end{enumerate}
For a vertex $v\in V(T)$, we denote by $|v|$ its valency.  The tree $T$ is called 
\textit{stable} if $|v|\geq 3$ for all $v$, and \textit{semistable} if $|v|\geq 2$. 

Let $\mathbf{L}_{T}\subset \mathbf{L}$ be a finite subset of admissible Lagrangians. We say  that a tree $T$ is \textit{labeled by} $\mathbf{L}_{T}$ if each  connected components of $\bR^2 \setminus T$ assigned a label from $\mathbf{L}_{T}$.

We fix an orientation so that every edge is directed from root to the leaves. This  determines functions $$\mathrm{lef},\mathrm{rig}:E(T)\rightarrow \mathbf{L}_{T} $$ such that for each edge $e\in E(T)$, the label $\mathrm{lef}(e)$ (resp. $\mathrm{rig}(e)$) denotes the Lagrangian lying to the left (resp. right)  of $e\in E(T)$ with respect to its orientation.  

A Riemannian metric $g_{T}$ on $T$ assigns a length to  each edge of $T$ such that any external edge is of infinite length. 
\end{defn}

For  each pair of Lagrangian $L_{i},L_{j} \in \mathbf{L}$, we assign an admissible family of Morse-Smale pairs $$(f_{(i,j), N}, \cV_{(i,j), N})$$ where $f_{(i,j), N}: BG(N) \rightarrow \bR$ is a Morse function and  $\cV_{(i,j), N} $ is the corresponding pseudo-gradient vector field of $f_{(i,j),N}$ in the sense of Theorem-Definition \ref{thm:family_of_Morse_functions}. 

By Lemma \ref{lem:abundant_choices}, such data can be chosen inductively. The corresponding $G$-equivariant Floer complex  $CF_{G}^{\bullet}(L_{i_{0}},L_{i_{1}})$ is then defined using the associated admissible Morse-Smale pairs $(f_{(i_{0},i_{1}), N}, \cV_{(i_{0},i_{1}), N})$ following the construction as in Section \ref{sec:equivariant_Floer_cohomology}. 

In some cases, such as when defining  a product structure on  $$\mathbf{x}_{i}\coloneqq  (c_{i}, x_{i}) \in CF^{\bullet}(L_{0}, L_{0}; X(N)),$$ the moduli space involves counting intersections of  stable manifolds of critical points $c_{i}$ and $c_{j}$ of $f_{(0,0)}$.  These intersections are empty unless $c_{i}= c_{j}$. To achieve the transversality  in such situations, it is therefore necessary to introduce appropriate perturbation data. We follow the construction as in \cite[Definition 2.6]{Abo11}. 
\begin{defn}
A \textit{gradient flow perturbation datum} on a metric tree $(T,g_{T})$ consists of an assignment, for each edge $e\in E(T)$, of a family of vector fields: 
\begin{align*}
X_{e}(N): e\rightarrow T^{\infty}(BG(N))
\end{align*}
which vanishes outside a  bounded subset of $e$. 
\end{defn}  
These data are chosen coherently in two directions: first, they are compatible with perturbations of the gradient flow equation on higher-dimensional moduli spaces of trees; second, they are compatible with the system of approximations as $N$ varies. 
\begin{defn}[Horizontal Morse trees]
\label{defn:horizontal_Morse_tree}
A (perturbed) Morse tree corresponding to $(T, g_{T})$ with labels $L_{T}$ is a continuous map 
\begin{align*}
    \psi: T\rightarrow BG(N)
\end{align*}
whose restriction to every edge $e\in E(T)$ solves perturbed gradient flow equation 
\begin{align*}
    \frac{d}{d\tau}  \psi|_{e}= \cV_{(\mathrm{lef}(e),\mathrm{rig}(e)), N}+X_{e}(N).
\end{align*}
\end{defn}

The vertical component of the construction proceeds as in the standard setting, except that it now solves the pseudo-holomorphic equation fiberwise. We recall some notions following \cite{Sei08,abouzaid2010open}. 

Let $\cS^{d+1} \rightarrow \cR^{d+1}$ denote universal family of disks with $d + 1$ ordered boundary marked points. The base $\cR^{d+1, T}$ can be compactified to a manifold with corners by adjoining all stable broken disks, 
\begin{align*}
    \overline{\cR}^{d+1} = \coprod_{T} \cR^{d+1, T}. 
\end{align*}
where each stratum $\cR^{d+1, T}$ denote the moduli space of broken disk modeled on $T$.

Let $$\Sigma = \bD\setminus \{z_{0},\cdots, z_{d}\}$$ denote a disk with  $d+1$ boundary marked points removed. For each puncture $z_{i}$, we choose a strip-like end $\epsilon_{i}:\bR_{\pm }\times [0,1]\rightarrow \Sigma$  in a neighborhood of $z_{i}$ for  all surfaces in $\overline{\cR}^{d+1, T}$.

Such  a family of strip-like ends is required to be chosen in a \textit{universal way} in the sense of \cite{Sei08}, meaning that  it is compatible with  gluing along  collar neighborhoods of boundary strata:  
\begin{align*}
    \overline{\cR}^{i+1, T_{1}}\times  \overline{\cR}^{j+1, T_{2}}\rightarrow  \overline{\cR}^{i+j+1, T}. 
\end{align*} 
We denote by $\partial_{i}\Sigma$ the boundary component between $z_{i}$ and $z_{i+1}$, with the convention $z_{d+1} = z_{0}$. 
 
Similarly, we choose a   \emph{universal perturbation datum}, namely a $1$-form $Y\otimes \alpha $ on $\partial\Sigma$ with values in vector fields. In our case, we  consider moving boundary condition given by the flow of the  $G$-invariant wrapping Hamiltonian $H$. 

Fix a sub-closed  $1$-form $\alpha \in \Omega^1(\Sigma)$ satisfying 
\begin{align}
    \label{equ:sub_closed_1_form}
    \epsilon_{i}^*\alpha =w_{i}dt \text{  for some }  w_{i}\in \bR_{\geq 0}, \quad d\alpha = 0 \text{ near } \partial \Sigma, \quad  \alpha|_{\partial \Sigma } = 0. 
\end{align}

\begin{defn}[Fiberwise pseudoholomorphic treed disks]
    \label{defn:Fiberwise_pseudoholomorphic_treed_disks}
    For a  tree labeled by $\mathbf{L}_{T}$, the moduli space $\cM(\mathbf{x}_{0}; \mathbf{x}_{1}\cdots, \mathbf{x}_{d};X(N))$ consists of  pairs of maps $(\psi , u)$ satisfying the following conditions: 
    \begin{enumerate}
        \item The map $\psi: T\rightarrow BG(N)$ is the Morse tree defined in Definition \ref{defn:horizontal_Morse_tree}. 
        \item The map $u:\Sigma \rightarrow X(N)$ solves the inhomogeneous fiberwise pseudoholomorphic holomorphic equation 
        \begin{align}
            \label{equ:inhomogeneous_floer_equation_used_in_surface}
            (du^{\mathrm{vert}} -  X_{H}\otimes \alpha )^{0,1}=0.
        \end{align} 
        \item The projection  $\pi_{N}\circ u = \psi$ sends the thin parts of the disks to the edges of the tree, and thick parts to vertices. 
        \item (Boundary and asymptotic condition)  Theses maps satisfy the boundary and asymptotic conditions: 
        \begin{align*}
            u(\partial_{i}\Sigma) \subset L_{i}(N) \quad \lim_{\tau\rightarrow \pm\infty}(\psi,u)(\epsilon_{i}(\tau,\cdot))= \mathbf{x}_{i},
        \end{align*}
        for generators $\mathbf{x}_{i}\in CF^{\bullet}(L_{i-1}, L_{i};X(N))$. 
    \end{enumerate}
\end{defn}

By construction, the wrapping Hamiltonian $H$ typically does not produce transverse intersections, since the intersection points occur in $G$-orbits. To remedy this issue, one supplements \eqref{equ:inhomogeneous_floer_equation_used_in_surface} with an additional compactly supported vertical inhomogeneous perturbation term. This local perturbation breaks the $G$-symmetry near the relevant intersection loci and achieve transversality.

After a generic choice of the perturbation data, the moduli space is  a regular manifold of the expected dimension 
\begin{align*}
   d- 2+|\mathbf{x}_{0}|- \sum_{1\leq k\leq d} |\mathbf{x}_{k}|.
\end{align*}
This follows from the surjectivity of the linearized Cauchy-Riemann operator,  which decomposes into vertical and horizontal components,  each known to be surjective under generic perturbations.  The surjectivity for the horizontal Morse tree component is established in  \cite{Abo11}.

\begin{prop}[Maximum principle]
    \label{pro:maximum_principle}
    Let $u:\Sigma \to X(N)$ be a solution to the inhomogeneous Floer equation \eqref{equ:inhomogeneous_floer_equation_used_in_surface}. 
    Then $u$ satisfies the following maximum principles:
 The function $h\circ u$ attains no local maxima outside of $X^{\mathrm{in}}(N)$, where $h$ is the $J_{N}$-plurisubharmonic function defined in Lemma~\ref{lemd:fiberwise_wrapped_data}.
    \end{prop}
    
    \begin{proof}[Sketch of proof]
        The statement follows from the standard maximum principle for $J_{N}$-holomorphic maps with inhomogeneous perturbations, applied to almost complex structures of contact type near infinity. One may either apply the general global version of the maximum principle \cite[Proposition~3.11]{AA24}, or give a direct argument exploiting the fact that we are working with a Liouville–Borel construction as follows. 
        
        Since the radial coordinate on the cylindrical end $r$ is invariant under $G$-action, it is globally defined function away from $X^{\mathrm{in}}(N)$. Lemma \ref{lem:no_escape_lemma} ensures that no pseudo-holomorphic strip escapes to infinity.
    \end{proof}
    \begin{rem}[The product and $H_{G}^{\bullet}(*)$-module structure] The product structure of $HW^{\bullet}$ is defined by considering the tree with three edges. In particular, this induces a natural  $H_{G}^{\bullet}(*)$-module structure on $HW^{\bullet}$. 
    Recall that the terminal map    $L\rightarrow *$ induces a morphism $$H^{\bullet}_{G}(*)\rightarrow H^{\bullet}_{G}(L),$$ endowing  $ H^{\bullet}_{G}(L)$ with the  structure of a module over $H^{\bullet}_{G}(*)$. 
By combining this map with the $G$-equviariant PSS isomorphism, the $\mu^2$ product operation, and  then passing to the direct limit, we obtain the composition 
\begin{align*}
    H^{\bullet}_{G}(*)\rightarrow  H_{G}^{\bullet}(L) \xrightarrow{\mathrm{PSS}_{G}} HF_{G}^{\bullet}(L,L)
    \xrightarrow{\mu^2(\cdot,\cdot)}  HF_{G}^{\bullet}(L,K)\rightarrow HW_{G}^{\bullet}(L,K),
\end{align*}
which equips $HW^{\bullet}_{G}$ with the desired  $H^{\bullet}_{G}(*)$-module structure. 
\end{rem}

\subsection{\texorpdfstring{$G$}{LG}-equivariant wrapped Fukaya category}
For each  object $L(N)\in \mathbf{L}_{N}$, we write 
\begin{align*}
    L(N)(t) \coloneqq  \phi_{H}^{t}(L(N))
\end{align*} 
where $\phi_{H}^{t}$ denotes the flow of the wrapping Hamiltonian $H$. By Condition (4) of Lemma \ref{lem:G_equivariant_wrapping}, the $H$-wrapping preserves the  moment map level, so $L(t)$ remains in the same level set of the moment map, and the same is true for all of its associated Borel spaces. Condition (6) of Lemma \ref{lemd:fiberwise_wrapped_data} implies that the wrapping isotopy commutes with the isotopy generated by the horizontal perturbation data.

    Choose a generic $\epsilon>0$ such that for every pair of Lagrangians $L_{0}, L_{1} \in \mathbf{L}$, the intersection  $$L_{0}(N)(\epsilon k_{0})\cap L_{1}(N)(\epsilon k_{1}) $$ is compact for  all distinct integers $k_{0} \neq k_{1}$ and for all $N \in \mathbb{N}$. This set of such admissible $\epsilon$ is a countable intersection of Baire sets and hence is dense as well.  
    
  At each level $N$,   we define a directed $A_{\infty}$-category $\cO_{N}$ whose objects are $L^{k}(N)\coloneqq  L(N)(-\epsilon k)$ for all $k\in \bZ$ and $L\in \mathbf{L}$. The morphism spaces are given by 
    \begin{align*}
        \cO_{N}(L_{0}^{k_{0}},L_{1}^{k_{1}})\coloneqq \begin{dcases*}
            CF^{\bullet}\left(L_{0}(N)(-\epsilon k_{0}), L_{1}(N)(-\epsilon k_{1})\right) & if $k_{0}<k_{1}$,\\ 
            \mathbb{K}\cdot \mathrm{id} & if $k_{0}= k_{1}$ and $L_{0} = L_{1}$,\\
            0 & otherwise.
        \end{dcases*}
    \end{align*}
    The $A_{\infty}$-structures $\mu^d_{N}$ are defined by counting the solutions to the inhomogeneous Floer equation \eqref{equ:inhomogeneous_floer_equation_used_in_surface} with compactly supported vertical perturbation terms. The perturbation data are chosen coherently across all moduli space (see \cite{Sei08} for details). Compactness of the resulting moduli space  follows from the maximum principal as in Proposition \ref{pro:maximum_principle}.

    \subsubsection{Equivariant Fukaya--Seidel category}

Let $(I,\leq)$ be a partially ordered set, viewed as a category in which there exists a unique morphism  $i \to j$ whenever $i \leq j$, and $i \to i$ is the identity morphism.  
Suppose we are given a strictly commuting diagram of $A_\infty$-categories 
\begin{align*}
    \{\cC_i, f_{ij} : \cC_i \to \cC_j\}_{i \leq j \in I},
\end{align*}
in the sense that for any composable sequence $[i \to j \to k]$ in $I$, one has
\begin{align*}
    f_{jk} \circ f_{ij} = f_{ik}.
\end{align*}
Here each $f_{ij}$ is assumed to be a strict $A_\infty$-functor.  

\begin{defn}
\label{defn:inverse_limit_of_A_infinity_categories}
The \emph{inverse limit} \(A_\infty\)-category $$ \cC \coloneqq  \varprojlim_{i \in I} \cC_i $$
is defined as follows:
\begin{enumerate}
    \item An object $C$ of $\cC$ is a  family  
    \begin{align*}
        \{C_i\}_{i \in I} \in \prod_{i \in I} \mathrm{Ob}(\cC_i) 
    \end{align*}
    satisfying $f_{ij}(C_i) = C_j$ for all $i \leq j$.
    \item For objects $C^{k_0}, C^{k_1} \in \cC$, the morphism space is given by the inverse limit
   \begin{align*}
    \cC(C^{k_0}, C^{k_1}) \coloneqq   \varprojlim_{i \in I}  \cC_i(C^{k_0}_i, C^{k_1}_i).
   \end{align*}
    \item The $A_\infty$-operations are defined componentwise:
    \begin{align*}
        \mu^d_{\cC}(c^{k_1}, \ldots, c^{k_d})\coloneqq  \varprojlim_{i \in I} \mu^d_{\cC_i}(c^{k_1}_i, \ldots, c^{k_d}_i)
    \end{align*}
    for morphisms $c^{k_j} \in \cC(C^{k_{j}}, C^{k_{j+1}})$. Since each $\mu_{\cC_{i}}^{d}$ satisfies the $A_{\infty}$-relations, so does $\mu^d_{\cC}$. 
\end{enumerate}
\end{defn}
We now show that the projection map induces a strict $A_{\infty}$-functor.
\begin{prop}
\label{pro:strict_A_infinity_functor}
Let $P_{N}$ denote the projection map introduced in  Section \ref{sec:differential}. Then $P_{N}$ induces a strict $A_{\infty}$-functor, that is  
\begin{align}
    \label{equ:strict_A_infinity_functor}
    P_{N+1}\left(\mu^{d}_{\cO_{N+1}}(\mathbf{x}_{i_{1}},\cdots, \mathbf{x}_{i_{k}})\right)= \mu_{\cO_{N}}^{d}\left(P_{N+1}(\mathbf{x}_{i_{1}}, \cdots, P_{N+1}( \mathbf{x}_{i_{d}})\right),
\end{align} 
for all  $\mathbf{x}_{i_{k}}\in CF^{\bullet}(L_{i_{k-1}},L_{i_{k}}; X(N+1))$. 
\end{prop}
\begin{proof}
    The argument is similar to that of Lemma \ref{lem:approximation_commutative}, with the role of Morse trajectories replaced by Morse trees.  Since $\psi = \pi_{N+1}\circ u$ solves the Morse tree equations, 
we obtain, whenever, $\cM(\mathbf{x}_{i_{0}}; \mathbf{x}_{i_{1}}\cdots, \mathbf{x}_{i_{d}};X(N+1))\neq \emptyset$, a fibration of moduli spaces
\begin{align*}
    \cM_{[T]}(\mathbf{x}_{i_{0}}; \mathbf{x}_{i_{1}}\cdots, \mathbf{x}_{i_{d}};X(N+1))\hookrightarrow \cM(\mathbf{x}_{i_{0}}; \mathbf{x}_{i_{1}}\cdots, \mathbf{x}_{i_{d}};X(N+1))\\ \xrightarrow{\pi_{N+1}} \cM_{\mathrm{Morse}}(c_{i_{0}}; c_{i_{1}}\cdots, c_{i_{d}};X(N+1)),  
\end{align*} 
where $\cM_{\mathrm{Morse}}$ denotes the moduli space of Morse trees defining the $A_{\infty}$-algebra  structure of $0_{BG_{N+1}}$, and $\cM_{[T]}$ denote the subspace of $\cM $ projecting to  a fixed tree $T$. 

 By the construction of admissible families of Morse functions (Definition \ref{thm:family_of_Morse_functions}), the counts of such configurations remain invariant under inclusion. As a result, the equality \eqref{equ:strict_A_infinity_functor} follows. 
\end{proof}

As a corollary,  let $I \coloneqq  (\mathbb{N}, \leq)$ be the poset of natural numbers, then  the collection $\{\cO_i\}_{i \in I^{\op}}$ of Fukaya--Seidel categories forms an inverse system, where the morphisms are given by the projection functors 
$P_{i+1}: \cO_{i+1} \to \cO_i$.
\begin{defn}
    We define the \emph{equivariant directed category} as the inverse limit
\begin{align*}
    \cO_{G} \coloneqq  \varprojlim_{i \in I^{\op}} \cO_i.
\end{align*}
 An object of $\cO_{G}$ associated to a Lagrangian $L_{k}\in \mathbf{L}$  is given by a compatible sequence of Borel Lagrangians. For notational convenience,  we continue to denote such sequences simply by $L_{k}$. The morphism space are given by 
\begin{align*}
    \cO_{G}(L_{0}^{k_{0}},L_{1}^{k_{1}})\coloneqq \begin{dcases*}
        CF_{G}^{\bullet}\left(L_{0}(-\epsilon k_{0}), L_{1}(-\epsilon k_{1})\right) & if $k_{0}<k_{1}$,\\ 
        \mathbb{K}\cdot \mathrm{id} & if $k_{0}= k_{1}$ and $L_{0} = L_{1}$,\\
        0 & otherwise.
    \end{dcases*}
\end{align*}
The $A_{\infty}$-operations are defined componentwise, as in Definition \ref{defn:inverse_limit_of_A_infinity_categories}. 
\end{defn}

A standard construction \cite{GPS1,AA24} produces a distinguished collection of elements 
\begin{align*}
    \{e_{L^k}\} \subset HF^{0}(L^{k},L^{k+1}),
\end{align*}
called the \emph{quasi-units}, defined as the images of the identity class in $H^{\bullet}(L^k)$ under the composition
\begin{align*}
    H^{\bullet}(L^k) \xrightarrow{\mathrm{PSS}} HF^{\bullet}(L^k,L^{k})\longrightarrow HF^{\bullet}(L^k,L^{k+1}).  
\end{align*}
Analogously, we define the \emph{$G$-equivariant quasi-units}. Let $\cZ_{G}$ denote the set of representatives of elements in $HF_{G}^{0}(L^k,L^{k+1})$ obtained as the image of the identity class under the composition
\begin{align*}
    H_{G}^{\bullet}(L^k) \xrightarrow{\mathrm{PSS}_{G}} HF_{G}^{\bullet}(L^k,L^{k})
    \longrightarrow HF_{G}^{\bullet}(L^k,L^{k+1}).
\end{align*}
Here, the continuation maps and $G$-equivariant PSS isomorphisms are established in  Sections \ref{sec:continuation_maps} and  \ref{sec:G_equivariant_PSS_isomorphism} respectively.

\begin{defn} 
Let $(X,\lambda)$ be a Liouville manifold acted by a compact connected Lie group $G$ which preserves the Liouville structure.  The $G$-equivariant wrapped Fukaya category is defined as the localization of the equivariant directed category $\cO_{G}$ by the quasi-units
\begin{align*}
    \cW_{G}(X)\coloneqq  \cZ_{G}^{-1}\cO_{G}
\end{align*}
in the sense of \cite{LO06}.
\end{defn}
Following \cite{AA24}, we now present an explicit model of the $G$-equivariant wrapped Fukaya category, which will be convenient for subsequent computations. We begin by constructing explicit cochain-level representatives of the $G$-equivariant quasi-units. 

At each level $N\in \bN$, the quasi-units are defined by counting maps $u :\bD\setminus \{1\} \rightarrow X(N)$ solving the Floer equation \eqref{equ:inhomogeneous_floer_equation_used_in_surface}. Here, $\bD\setminus \{1\}$ is a punctured Riemann disk, and the moving boundary condition is  given by $L^t = L(-\epsilon t)$ for $t\in [k,k+1]$, constant near the end. We denote by 
\begin{align*}
    e_{L^k}^N \in CF^{0}\left(L(N)(-\epsilon k), L(N)(-\epsilon (k+1))\right) =  \cO_{N}(L^k, L^{k+1}),
\end{align*}
the cochain-level quasi-unit.
By choosing the perturbation data coherently with respect to the system of approximations, a similar argument as before shows that
\begin{align}
    \label{equ:quasi_units_identification}
    P_{N+1}(e_{L^k}^{N+1}) = e_{L^k}^N.
\end{align}
We therefore define
\begin{align*}
    e^{G}_{L^k} \coloneqq  \varprojlim_{N} e_{L^k}^N \in CF^{0}_{G}(L^k, L^{k+1}),
\end{align*}
and refer to it as the \emph{$G$-equivariant quasi-unit}.  
Although the construction of quasi-units depends on several auxiliary choices, standard homotopy arguments ensure that the resulting elements are well-defined up to homotopy equivalence.

On the other hand, we have continuation maps as defined in Section \ref{sec:continuation_maps}, 
\begin{align*}
    F^N_{L_{0}^k, L_{1}^j}: \cO_{N}(L_{0}^{k}, L_{1}^j)\rightarrow \cO_{N}(L_{0}^{k+1}, L_{1}^{j+1})
\end{align*}
obtained by counting maps $u:\bR\times [0,1]\rightarrow X(N)$ that solve the Floer equation with corresponding moving boundary conditions.  It is standard that composition with quasi-units is homotopic to the corresponding continuation map. We now adapt this construction to the setting of Borel construction. 

More precisely, we stack the diagram of quasi-unit compositions over one another via the projections induced by the Borel construction (see Diagram \ref{diagram:commutativity_of_wrapping_and_approximations}) and we verify that all such compositions commute up to homotopy.  
Note that we adopt a different $A_{\infty}$-order convention from that used in \cite{AA24}.

\begin{center}
    \begin{tikzcd}[row sep={65,between origins}, column sep={90,between origins}]
        & \cO_{N+1}(L_{0}^{k+1}, L_{1}^j) \arrow[rr ] \arrow[dd] \arrow[dl] 
          & &  |[alias=youshangjiao]| \cO_{N+1}(L_{0}^{k+1}, L_{1}^{j+1})  \arrow[dd] \arrow[dl] \\
          \cO_{N+1}(L_{0}^{k}, L_{1}^j) \arrow[to=youshangjiao] \arrow[rr, crossing over] \arrow[dd, "P_{N+1}"'] 
          & & \cO_{N+1}(L_{0}^{k}, L_{1}^{j+1}) \arrow[dd] \\
        &  \cO_{N}(L_{0}^{k+1}, L_{1}^j)  \arrow[rr, "\mu^2_{N}\bigl(\cdot \text{, }e^N_{L_{1}^j}\bigr)" near start] \arrow[dl, "\mu^2_{N}\bigl(e^N_{L_{0}^k}\text{, }\cdot\bigr)"' near start] 
          & &  |[alias=youshangjiao2]| \cO_{N}(L_{0}^{k+1}, L_{1}^{j+1}) \arrow[dl, "\mu^2_{N}\bigl(e^N_{L_{0}^k}\text{, }\cdot\bigr)"] \\
          \cO_{N}(L_{0}^{k}, L_{1}^j)  \arrow[to=youshangjiao2, "F^N_{L_{0}^k, L_{1}^j}"] \arrow[rr, "\mu^2_{N}\bigl(\cdot \text{, }e^N_{L_{1}^j}\bigr) "'] 
          & & \cO_{N}(L_{0}^{k}, L_{1}^{j+1}) \arrow[from=uu, crossing over]
      \end{tikzcd}

      \captionsetup{type=figure}
      \captionof{figure}{Quasi-units are natural with respect to continuation maps and approximations.}
      \label{diagram:commutativity_of_wrapping_and_approximations}
\end{center}
\begin{prop}
\label{prop:commutativity_of_wrapping_and_approximations}
    The horizontal squares and triangles in Diagram \ref{diagram:commutativity_of_wrapping_and_approximations} commute up to homotopy, while the vertical squares commute strictly.
\end{prop}
\begin{proof}
    The homotopy commutativity of the horizontal squares and triangles follows from an  argument identitical to that in  \cite[Lemma 3.19]{AA24}.  
    The strict commutativity of the vertical squares is an immediate consequence of  Proposition \ref{pro:strict_A_infinity_functor} together with the identification \eqref{equ:quasi_units_identification}.  
\end{proof}
By taking the inverse limit, we obtain the corresponding commutative diagram in $\cO_{G}$. 
\begin{cor}
\label{cor:G_equivariant_homotopy_commutative}
The following diagram commutes up to homotopy:
\begin{equation}
    \label{equ:G_equivariant_homotopy_commutative}
     \begin{tikzcd}[row sep={70,between origins}, column sep={120,between origins}]
        \cO_{G}(L_{0}^{k+1}, L_{1}^{j}) \arrow[r, "\mu^2_{G}\bigl(\cdot \text{, }e^G_{L_{1}^j}\bigr)"] \arrow[d, "\mu^2_{G}\bigl(e^G_{L_{0}^k}\text{, }\cdot\bigr)"'] & |[alias=youshangjiao]| \cO_{G}(L_{0}^{k+1}, L_{1}^{j+1}) \arrow[d, "\mu^2_{G}\bigl(e^G_{L_{0}^k}\text{, }\cdot\bigr)"] \\
        \cO_{G}(L_{0}^{k}, L_{1}^{j})  \arrow[to=youshangjiao, "F^G_{L_{0}^k, L_{1}^j}"]\arrow[r, "\mu^2_{G}\bigl(\cdot \text{, }e^G_{L_{1}^j}\bigr)"] &  \cO_{G}(L_{0}^{k}, L_{1}^{j+1})
     \end{tikzcd}.
 \end{equation} 
\end{cor}

To proceed, for $L^k \in \cO_{G}$, we denote by   $\cY_{L^k}$ the corresponding  Yoneda module 
\begin{align*}
    X \rightarrow \cO_{G}(X,L^k). 
\end{align*}
By Corollary \ref{cor:G_equivariant_homotopy_commutative}, these Yoneda modules form a direct system, where the connecting morphisms $\cY_{L^k}\rightarrow \cY_{L^{k+1}}$ are induced by composition with  $G$-equivariant quasi-units $e_{L^k}^{G}$.

Hence, we define 
\begin{align*}
    \cY_{L^{\infty}} \coloneqq  \uhclim{k\rightarrow \infty} \cY_{L^k},
\end{align*}
where we model the homotopy colimit by the usual telescope construction:  
\begin{align*}
    \uhclim{k\rightarrow \infty} \cY_{L^k} \cong \Cone\left(\bigoplus_{k=0}^{\infty} \cY_{L^k}\rightarrow  \bigoplus_{k=0}^{\infty} \cY_{L^k}\right).
\end{align*}
Here, the arrow denotes the direct sum of the maps $\mathrm{id}- e_{L^k}^{G}$, where $e_{L^k}^{G}$ are the $G$-equivariant quasi-units. 

Corollary~\ref{cor:colimit_compute_localization} essentially means that the category generated by Yoneda modules $\cY_{L^{\infty}}$ are $\cZ_{G}$-local (see \cite[Section 3]{GPS1} and \cite[Section 3.5]{AA24}), and therefore quasi-equivalent to the localization of $\cO_{G}$. In other words, we have 
\begin{align*}
        \cW_{G}(L_{0},L_{1}) \cong \hom_{\cO_{G}}(\cY_{L_{0}^{\infty}}, \cY_{L_{1}^{\infty}}).
\end{align*}
For our purposes, we record only the minimal properties that will be used below, and refer the reader to the references above for a complete discussion.
\begin{lemma}
\label{lem:quasi_isomorphism_between_wrapping}
There is a quasi-isomorphism 
\begin{align*}
    \hom_{\cO_{G}}(\cY_{L_{0}^{\infty}}, \cY_{L_{1}^{\infty}}) \cong \uhlim{k\rightarrow \infty} \uhclim{j\rightarrow \infty}\cO_{G}(L_{0}^k,L_{1}^j),
\end{align*}
where the  morphisms in the direct and inverse  system are given by composition with the $G$-equivariant quasi-units. 
\end{lemma}
\begin{proof}
    Since the hom-functor preserves the limits,  we obtain a natural duality between the  the homotopy equalizer and coequalizer of the identity map and the connecting morphism. 
These constructions provide respective  models for  homotopy (co)limit. See \cite[Lemma 3.22]{AA24} for further details. 
\end{proof}

\begin{cor}
\label{cor:colimit_compute_localization}
For each pair of objects $L_{0}^{k}$ and $L_{1}$ of objects in $\cO$, there is a  natural isomorphism 
\begin{align*}
    \varinjlim_{j\rightarrow \infty } HF^{\bullet}_{G}(L^k_{0},L_{1}^j) = H^{\bullet} \hom_{\cO}(\cY_{L_{0}^{\infty}}, \cY_{L_{1}^{\infty}}).
\end{align*}
\end{cor}
\begin{proof}
The map 
\begin{align*} 
    \uhclim{j\rightarrow \infty}\cO_{G}(L_{0}^{k+1},L_{1}^j) \rightarrow  \uhclim{j\rightarrow \infty}\cO_{G}(L_{0}^{k},L_{1}^j)
\end{align*}
given by composing the qusi-unit is a quasi-isomorphism. This follows from the compatibility of the continuation maps with $G$-equivariant  quasi-unit (Corollary \ref{cor:G_equivariant_homotopy_commutative}); see \cite[Lemma 3.23, Corollary 3.24]{AA24} for details. The inverse system obtained by varying $k$ satisfies the Mittag-Leffler condition.  In particular, the following composition 
\begin{align}
    \uhclim{j\rightarrow \infty}\cO_{G}(L_{0}^{k+1},L_{1}^j) \leftarrow  \uhlim{k\rightarrow \infty} \uhclim{j\rightarrow \infty}\cO_{G}(L_{0}^k,L_{1}^j) \cong  \hom_{\cO_{G}}(\cY_{L_{0}^{\infty}}, \cY_{L_{1}^{\infty}}) 
\end{align}
induces an isomorphism on cohomology. 
\end{proof}

\section{Equivariant wrapped Floer theory for Lagrangian correspondences}
\begin{comment}
[***Temporarily for internal use] Quick Summary of the whole section:
\begin{enumerate}
    \item (Section \ref{exactlagcorr}) \begin{itemize}
        \item Exact Lagrangian correspondences and their compositions
        \item Highlighting Definition \ref{cleancomp}: clean compositions which appear in the example of Liouville quotients.
        \end{itemize}
    \item (Section \ref{cwlagcorr}) 
    \begin{itemize}
        \item Recalling (resp. defining) wrapped Floer theory for Lagrangian (resp. correspondences) in Definition \ref{wfc} (resp. Definition \ref{wrappedcorrbimod}).
        \item Fukaya-categorical statements in Theorem \ref{quiltedbifunc}.
        \item (In Progress) Cyclic property for embedded compositions.
        \end{itemize}
    \item (Section \ref{equivwft}) 
        \begin{itemize}
        \item Defining equivariant wrapped Floer theory for Lagrangian (resp. correspondences) in Definition \ref{equifloercpx} (resp. Definition \ref{equicorrbimod}). 
        \item Equivariant Fukaya-categorical statements in Theorem \ref{equivquiltedbifunc}.
        \item (In Progress) Equivariant cyclic property for clean compositions.
         \end{itemize}
    
    \item (Section \ref{liouquot}) 
    \begin{itemize}
    \item Equivariant geometric Composition theorem (Theorem \ref{eqgeocompthm}).
    
    \item Induced $A_\infty$ functors between Fukaya categories (Corollary \ref{fukcatfunc}).  
    \end{itemize}
\end{enumerate}    
\end{comment}    

Throughout this section, we restrict  to the case where the  Liouville action of  $G$ on a Liouville manifold $X$ is free. For simplicity of exposition, we will work with ungraded  (wrapped) Floer groups.

Let $\cW_{G}(X)_{0}\subset \cW_{G}(X)$ denote the full subcategory generated by $G$-invariant exact Lagrangians contained in $\mu^{-1}(0)$, we construct a $(\cW_{G}(X)_{0}, \cW(X\sq G))$-bilinear module
\begin{align}
    \label{equ:bilinear_functor}
    \cB_{G}: \cW_{G}(X)^{\mathrm{op}}_{0}\times \cW(X\sq G) \rightarrow \mathbf{Ch},
\end{align}
which on the level of cohomology takes $L_{i}\times \underline{L}_{j}$ to  isomorphic cohomology groups $HW^{\bullet}_{G}(L_{i}, L_{j}) = HW^{\bullet} ( \underline{L}_{i},  \underline{L}_{j}) $.  We then  show for $L \in \ob \cW_{G}(X)_{0}$,  the functor  
\begin{align*}
    \cB_{G}(L, -) : \cW_{G}(X)_{0}\rightarrow \mathrm{Mod}_{\cW(X\sq G)}
\end{align*}
takes values in representable right modules.  Up to quasi-equivalence, this defines an $A_{\infty}$-functor
\begin{align*}
    \Psi_{\downarrow}:  \cW_{G}(X)_{0} \rightarrow \cW(X\sq G),
\end{align*}
which sends a $G$-invariant exact cylindrical Lagrangian  $L$ to its reduction  $\underline{L}$. 

An similar argument also applies to $\cB_{G}(-, \underline{L})$, yielding a corresponding $A_{\infty}$-functor 
\begin{align*}
    \Psi_{\uparrow}:  \cW(X\sq G) \rightarrow  \cW_{G}(X)_{0}, 
\end{align*}
 which sends a reduced Lagrangian $\underline{L}$ to its preimage $L\subset \mu^{-1}(0)$. Hence, $\cW(X\sq G)$ is quasi-equivalent to this full subcategory of $\cW_{G}(X)$.

 \begin{convention}
 To simplify notation, from now on we will drop the subscript $0$ and, whenever we refer to the equivariant Fukaya category, implicitly restrict to the full subcategory consisting of Lagrangians contained in the zero level set of the moment map. 
 \end{convention}

 The construction was first introduced in \cite{Ma15} and, in the exact setting, further developed by \cite{gao2018functors}. See also \cite{KLZ23,GPS3} for related constructions.  We adapt the classical construction to the Borel construction. 
\subsection{Moment correspondence}
\begin{defn}
    Let  $(X, \lambda_X), (Y, \lambda_Y)$ be Liouville manifolds. An \textit{exact Lagrangian correspondence}  from $X$ to $Y$ is an exact Lagrangian submanifold 
    $$L_{X,Y} \subset (X^- \times Y, -\mathrm{pr}_{X}^*\lambda_X +  \mathrm{pr}_{Y}^*\lambda_Y),$$
    where $X^{-}$ denotes $X$ equipped with the negated Liouville form $-\lambda_{X}$. 
\end{defn}
In the setting of Liouville reduction, the moment correspondence provides a canonical choice for Lagrangian correspondence. 
\begin{defn}    
Let $X$ be a Liouville manifold equipped with a Liouville action of a Lie group $G$. The \emph{moment correspondence} is defined as the Lagrangian submanifold 
\begin{align*}
    L^{\pi}\coloneqq  \left\{ (x, [x]) \in X^{-}\times X\sq G \mid x\in \mu^{-1} (0)\right\}\subset X^{-}\times X\sq G.
\end{align*}
\end{defn}
The fact that $L^{\pi}$ is Lagrangian follows directly from the defining property of symplectic reduction.  
Moreover, since the $G$-action preserves the Liouville structure, we obtain the following characterization.
\begin{lemma}
    \label{lem:moment_correspondence}
    The moment correspondence $L^{\pi}$ is an exact, cylindrical Lagrangian submanifold. 
\end{lemma}
\begin{proof}
Since  $\lambda_{X\sq G}$ is obtained by descending the $G$-invariant Liouville form $\lambda_{X}|_{\mu_{X}^{-1}(0)}$, we have 
\begin{align*}
    \lambda_{X^{-}\times (X\sq G)}|_{L^{\pi}} = 0.
\end{align*}
Hence $L^{\pi}$ is an exact Lagrangian submanifold with constant primitive. In particular, this implies that the Liouville vector field is tangent to $L^{\pi}$ everywhere, since $Zf = 0$ for any function $f$ vanishing on $L^{\pi}$. 
\end{proof}

\begin{defn}
\label{defn:G_composable_intersections}
A tuple of exact cylindrical $G$-invariant Lagrangians $$(L_{0},L_{1},\cdots, L_{d})\subset \mu^{-1}(0)$$  is said to be  \textit{$G$-composable} if it satisfies the following equivalent conditions:
\begin{enumerate}
    \item For any $i,j$, the intersection $L_{i}\cap L_{j}$ is clean and consists of a (possibly empty) union of $G$-orbits of single points; that is,
\begin{align}
            L_{i}\cap L_{j} = \bigsqcup_{\alpha} G\cdot z_{\alpha}.
\end{align}
    \item The corresponding reduced Lagrangians \( \underline{L}_{i} \subset X \sslash G \) intersect pairwise transversely. 
    \item The product $L_{i}\times \underline{L}_{j} \subset X^{-}\times (X\sq G)$ intersects the moment correspondence $L^{\pi}$ cleanly along the $G$-orbits of the form $(G\cdot z_{\alpha}) \times z_{\alpha}$, where $z_{\alpha}$ denotes  both the intersection point of  $\underline{L}_{i}\cap \underline{L}_{j}$  and a chosen representative of its lift in $\mu^{-1}(0)$. 
\end{enumerate}    
\end{defn}
For a general tuple of Lagrangians, we may replace each Lagrangian by an isomorphic object  so that the resulting tuple becomes $G$-composable. In Condition (2), by requirement of pairwise transversality, we also exclude triple or higher intersections. For a pair of Lagrangians $(L_{i},\underline{L}_{j}) \subset X^{-}\times (X\sq G)$, we similarly say it is $G$-composable if the pair  $(L_{i},L_{j}) $ is $G$-composable.

Formally, we will consider Lagrangian correspondence $$L_{G}^{\pi} \subset X^{-}_{G}\times X\sq G,$$  obtained as the direct limit of   a sequence of finite-dimensional approximations, which fits into the diagram  
\begin{equation} 
    \label{equ:correspondence_borel_construction}
\begin{tikzcd} 
  L^{\pi} \arrow[r, hook] \arrow[d]
    & L^{\pi}(N) \arrow[r] \arrow[d]
    & 0_{BG(N)}  \arrow[d]\\
    X^{-}\times X\sq G \arrow[r, hook] 
    & X^{-}(N)\times X\sq G \arrow[r, "\pi_{N}"] 
     & T^*BG(N)  
    \end{tikzcd}.
\end{equation}
 Correspondingly, for each pair of $G$-composable Lagrangians $L_{i},L_{j}\subset \mu^{-1}(0)$, we consider Lagrangian Borel spaces $(L_{i}(N)\times \underline{L}_{j})$ in $X^{-}(N)\times X\sq G$ lying over the zero section of $T^*BG(N)$.  
 
Before proceeding with the detailed construction, let us briefly outline the main ideas underlying the bilinear functor  \eqref{equ:bilinear_functor} and its representable property.

It is a classical fact that for a pair of Lagrangians $L$ and $K$ intersecting cleanly, there exists a spectral sequence converging to the Floer cohomology $HF^{\bullet}(L,K)$; see, for example, \cite{Poz99,Sei99, Sch16}.  The  $E_1$-page of this spectral sequence takes the form 
\begin{align}
    \label{equ:clean_spectral_sequence}
    E_{1}^{p,q}= \begin{dcases*}
        HF^{p+q,\mathrm{loc}}(L, K; C_{p}) & $1\leq p\leq r$ \\
        0 & otherwise
    \end{dcases*}.
\end{align}
Here $C_{1},\cdots, C_{r}$ are the connected components of $L\cap K$, ordered according to the  values of the action  on the constant paths.  The \textit{local Floer cohomology group} $HF^{p+q,\mathrm{loc}}$ is defined by counting pseudo-holomorphic strips whose images are  contained in a neighborhood of connected component $C_{l}$. The filtration underlying this spectral sequence is defined by 
\begin{align}
    F^{p}CF^{\bullet}(L,K) \coloneqq \{ x \in CF^{\bullet}(L,K) \mid  x\not \in C_{1}\cup \cdots \cup C_{r}\}. 
\end{align}
Since the Floer differential  increases  the action, it preserves this filtration.

Each local Floer cohomology group $ HF^{\mathrm{loc}}(L, K; C_{l})$ can be identified with the singular cohomology $ H^{\bullet}( C_{l})$, up to a grading shift determined by Robbin-Salamon index $\mu(C_{l})$. In the special case where the intersection $L\cap K$ is transverse, the local Floer group reduces to   $HF^{\mathrm{loc}}(L, K; C_{l}) = \bK$.

In our case, the clean intersection between $L^{\pi}$ and $ L_{i}\times \underline{L}_{j}$ are given by $G$-orbits, and hence are invariant under the $G$-action. This invariance allows us to identify the connected component $C_{l}$ of $L^{\pi} \cap \left(L_{i}\times \underline{L}_{j}\right)$  across the  fibers of  the Borel construction \eqref{equ:correspondence_borel_construction}; for simplicity,  we continue to denote them  by $C_{l}$. 

The above identifications also extend to the primitives and hence to the values of the action functional. For future reference, we formalize this as follows.
\begin{rem}[Continuation of Section~\ref{rem:Borel_Lagrangians}]
    \label{rem:primitive_identification}
    By Lemma \ref{lem:moment_correspondence}, the primitive of $L^{\pi}$ may  be chosen to be zero. Thus the value of the action functional (defined using $-\lambda_{X}+\lambda_{X\sq G}$) on $L^{\pi} \cap \left(L_{i}\times \underline{L}_{j}\right)$ is simply $f_{L_{j}} -f_{L_{i}}$ (note that $L_{i}\subset X^{-}$ contributes the opposite primitive). 

In other words, under the natural identifications among: 
\begin{enumerate}
    \item the  connected components of clean intersection $L_{i}\cap L_{j} $, 
    \item the transverse intersection points $\underline{L}_{i}\cap \underline{L}_{j}$, 
    \item the connected components of $L^{\pi} \cap \left(L_{i}\times \underline{L}_{j}\right)$, 
    \item the corresponding components in the Borel--Liouville construction,
\end{enumerate}
the respective action functionals take the same value $f_{L_{j}} -f_{L_{i}}$. 

If the components are ordered so that
\begin{align*}
    f_{L_{j}}-f_{L_{i}}|_{C_{1}}< \cdots < f_{L_{j}}-f_{L_{i}}|_{C_{r}},
\end{align*}
then this ordering is preserved under all of the identifications above. 

The above discussion is  carried out under the assumption  $H=0$.  In general case, we may restrict  $H$ to  a sufficiently small neighborhood of  $0$ in  the space of admissible Hamiltonians, ensuring that its contribution to the action functional is small enough that will not alter the ordering. A more convenient approach, avoiding the Hamiltonian, is to work entirely in the Morse--Bott model,  see Remark \ref{rem:Morse_Bott_models} below. 
\end{rem}

We may define the local Floer cohomology group in Borel construction by considering the Floer strips \eqref{defn:moduli_space} whose Hamiltonian chords and images lie in a small neighborhood of $C_{l}$.

Taking the inverse limit with respect to the approximations yields a spectral sequence whose $E_{1}$-page is given by 
\begin{align}
    \label{equ:G_spectral_sequence}
    E_{1}^{p,q}= \begin{dcases*}
        HF^{p+q,\mathrm{loc}}_{G}(L^{\pi},  L_{i}\times \underline{L}_{j} ; C_{p}) & $1\leq p\leq r$ \\
        0 & otherwise. 
    \end{dcases*}
\end{align}
As  shown in Section \ref{sec:G_equivariant_PSS_isomorphism}, there is an  isomorphism $HF^{\mathrm{loc}}_{G}(L^{\pi},  L_{i}\times \underline{L}_{j}; C_{l}) \cong H_{G}^{\bullet}(C_{l})$ up to a grading shift (cf. Remark \ref{rem:Morse_Bott_models}).

Since the $G$-action is free, we have  
\begin{align}
\label{equ:reduced_Lagrangians}
H_{G}^{\bullet}(C_{l}) \cong H^{\bullet}(C_{l}/G)\cong HF^{\bullet, \mathrm{loc}}(\underline{L}_{i}, \underline{L}_{j}), 
\end{align}
since $C_{l}/G$ is precisely the transverse intersection between the reduced Lagrangians  $\underline{L}_{i}$ and $\underline{L}_{j}$. On the other hand, by the identification between $L^{\pi}\cap \left(L_{i}\times \underline{L}_{j}\right)$ and $L_{i}\cap L_{j}$ described in Definition \ref{defn:G_composable_intersections}, we obtain 
\begin{equation}
\label{eq: CohoID}
HF^{\bullet, \mathrm{loc}}_{G}(L^{\pi},  L_{i}\times \underline{L}_{j}; C_{l}) \cong H_{G}^{\bullet}(C_{l}) \cong  HF_{G}^{\bullet, \mathrm{loc}}(L_{i}, L_{j};C_{l}).    
\end{equation}

As a result, we obtain a naive identification between the $E_{1}$-pages of two spectral sequence converging to  $HF^{\bullet}(\underline{L}_{i}, \underline{L}_{j})$ and $HF^{\bullet}_{G}(L_{i},L_{j})$ respectively. The remaining  task is to strengthen this identification to  an isomorphism between spectral sequences, and moreover, to promote it to a representable $A_{\infty}$-functor at the categorical level.

\begin{rem}[Morse-Bott model]
\label{rem:Morse_Bott_models}
    Since the relevant intersections occur cleanly in our case, it is  convenient to adapt the Definition \ref{defn:Fiberwise_pseudoholomorphic_treed_disks} to the  ``Morse-Bott" model; see, for example, \cite{BC09,Sch16}. 
    
    For a cleanly intersecting pair of Lagrangians $(L,K)$ , we choose a Morse function on each connected component of $L\cap K$  and work with hybrid configurations consisting of a pseudo-holomorphic disk together with  Morse trajectories connecting  boundary marked points of the disk to critical points of the chosen Morse functions. Equivalently, one may  require that each boundary marked point lies on the stable or unstable manifolds of a Morse critical point. 

    This provides a realization of the spectral sequence  \eqref{equ:clean_spectral_sequence} obtained by filtering the Floer complex by energy. The $E_{1}$-page is computed by the local Morse cohomology on the clean intersection components (see, for example, \cite{Poz99,Sch16}).

    Returning to the Borel construction, we choose a Morse function for each connected component  $$C_{l}\hookrightarrow L^{\pi}\cap (L_{i}\times \underline{L}_{j})\hookrightarrow L^{\pi}_{G}\cap \left(L_{i,G}\times \underline{L}_{j}\right) \rightarrow 0_{BG}$$ over  critical points of Morse function on $0_{BG}$.  The hybrid objects we count are   pseudo-holomorphic strips parametrized by Morse trajectories (Definition \ref{defn:moduli_space} with no Hamiltonians) whose  boundary marked point  are connected to critical points in these components by (families of) Morse trajectories \eqref{equ:Family_Morse}. This  is analogous to the PSS-type configurations considered earlier in Section \ref{sec:G_equivariant_PSS_isomorphism}.

    Then the filtration by energy produces the spectral sequence \eqref{equ:G_spectral_sequence},  and  the $E_{1}$-page is computed by (families of) Morse trajectories in these components. The  argument in Section \ref{sec:G_equivariant_PSS_isomorphism} shows that its $E_{1}$-page is isomorphic to $G$-equivariant cohomology of each connected component $C_{l}$ (up to a shift of grading).  
\end{rem}
    
\subsection{\texorpdfstring{$A_{\infty}$}{LG}-bimodules}
To define the bimodule 
\begin{align}
    \cB_{G}: \cW_{G}(X)^{\mathrm{op}}\times \cW(X\sq G) \rightarrow \mathbf{Ch},
\end{align}
we proceed in a manner analogous to Definition \ref{defn:Fiberwise_pseudoholomorphic_treed_disks}. We consider objects consisting of  a Morse tree together with a  fiberwise pseudo-holomorphic strip mapping  to the Borel space \eqref{equ:correspondence_borel_construction}.

 For the Morse tree part, we now require  the existence of one  distinguished incoming semi-infinite edge and one outgoing edge.

 \begin{rem}
    From this point, we will no longer refer explicitly to the finite-dimensional approximation construction.
    All pseudo-holomorphic strips or disks mapping to  $X_{G}$, together with their boundary and asymptotic conditions, are understood  as in Definition~\ref{defn:moduli_space} and Definition~\ref{defn:Fiberwise_pseudoholomorphic_treed_disks}. 
    
    For example, a pseudo-holomorphic strip $u:Z\rightarrow L_{G}$ should always be interpreted  as a pair $(\gamma, u)$ exactly as in  Definition~\ref{defn:moduli_space}; the same convention applies to disks, and all associated asymptotic and boundary conditions. 
 \end{rem}
 
The  quilted strip is  illustrated in Figure \ref{fig:Quilted_strip}. It can  be viewed either as a pair of pseudo-holomorphic strips mapping to  $X\sq G$ and $X_{G}$  respectively, or equivalently as a single  pseudo-holomorphic map to $X_{G}^{-}\times X\sq G$ via the standard ``folding strip'' argument. See for example \cite{Ma15}.

 \begin{figure}[ht]
    \def\svgwidth{0.3\textwidth}
	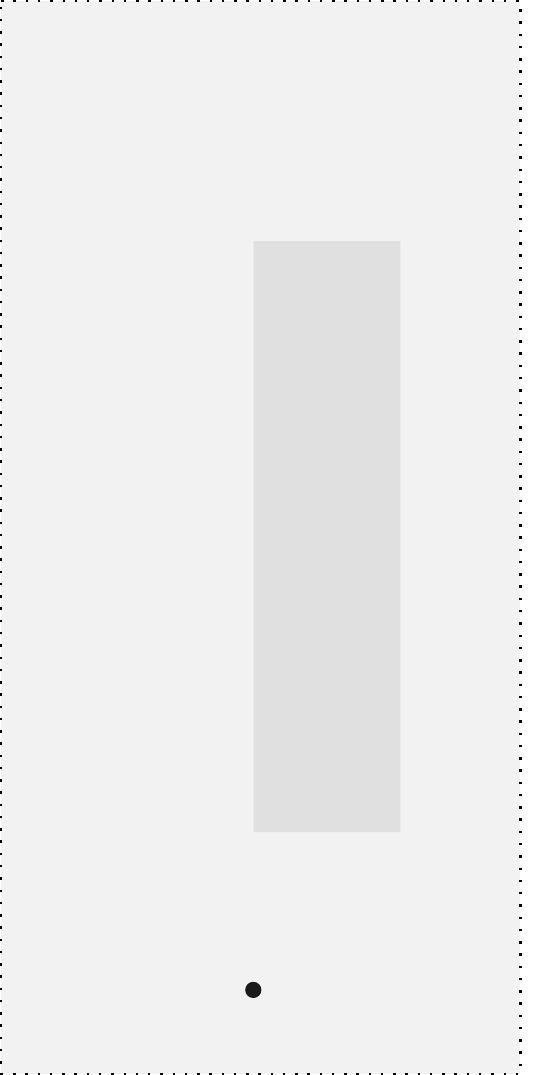
	\caption{Quilted Strips Defining the Bimodule.}
\label{fig:Quilted_strip}
\end{figure}

The quilted maps $u = (u_1, u_2): \Sigma = \Sigma_1 \cup \Sigma_2 \rightarrow X\sq G \times X_{G} $ with Lagrangian boundary conditions $(\underline{L}'
    _{0}, \cdots \underline{L}'
    _{k'}) \subseteq X\sq G
    $, $(L
    _{0, G}, \cdots, L
    _{k, G}) \subseteq X_G
    $ and seam condition $L^\pi_G$,  are described as follows.
\begin{itemize}
    \item The domain  $\Sigma \coloneqq [-1,1] \times \mathbb{R}$
    is a bordered quilted strip with two patches $\Sigma_1 \coloneqq [-1,0] \times \mathbb{R}$ and $\Sigma_2 \coloneqq [0,1] \times \mathbb{R}$  with a seam $\sigma \coloneqq \{0\} \times \mathbb{R}$. The left and right boundary components are denoted by $\sigma^{(1)} = \{-1\}\times \bR$ and $\sigma^{(2)} = \{1\}\times \bR$. 
    \item Let $z^{(1)} = (z^{(1)}_1, \cdots z^{(1)}_{k'})$  (resp. $z^{(2)} = (z^{(2)}_1, \cdots z^{(2)}_{k})$) be boundary marked points on $\sigma^{(1)}$ (resp. $\sigma^{(2)}$) such that for any $j_1 < j_2$, we have $\operatorname{Im} z^{(i)}_{j_1} > \operatorname{Im} z^{(i)}_{j_2}$. 
    \item For $j = 0, \dots k$ (resp. $j' = 0, \dots k'$), we  denote  the boundary segments between marked points by  $$\sigma^{(1)}_{j'} \coloneqq \{-1\} \times [\operatorname{Im}z^{(1)}_{j'+1}, \operatorname{Im}z^{(1)}_{j'}], \quad \sigma^{(2)}_{j} \coloneqq  \{1\} \times [\operatorname{Im}z^{(2)}_{j+1}, \operatorname{Im}z^{(2)}_{j}],$$
    with the convention  $\operatorname{Im}z^{(i)}_{0}\coloneqq +\infty$ and $\operatorname{Im}z^{(1)}_{k'+1} \coloneqq -\infty \eqqcolon \operatorname{Im}z^{(2)}_{k+1}$. Note that 
    $$\sigma^{(1)}_{j'-1} \cap \sigma^{(1)}_{j'} = \{z^{(1)}_{j'}\}, \quad  \sigma^{(2)}_{j-1} \cap \sigma^{(2)}_{j} = \{z^{(2)}_{j}\}.$$
    \item A quilted map  consists of  
    \begin{align*}
        u_1: \Sigma_1 \rightarrow (X\sq G, J_{X\sq G}),\quad  u_2: \Sigma_2 \rightarrow (X_{G}, J_{X_{G}}), 
    \end{align*}
    each pseudo-holomorphic with respect to the corresponding almost complex structures, subject to: 
    \begin{itemize}
    \item{[Boundary conditions]} For  each boundary segment,  $$u_1|_{\sigma^{(1)}_{j'}}: \sigma^{(1)}_{j'} \rightarrow \underline{L}^{'}
    _{j'} \subset X\sq G,  \quad  u_2|_{\sigma^{(2)}_{j}}: \sigma^{(2)}_{j} \rightarrow L_{j, G} \subset X_G. $$
    In particular, at the marked points: $u_1(z^{(1)}_{j'}) \in \underline{L}^{'}_{j'-1}\cap \underline{L}^{'}_{j'}$ and  $u_2(z^{(2)}_{j}) \in L_{j-1, G}\cap L_{j, G}$. 
    \item{[Seam condition]} Along the seam: $(u_1, u_2)|_{\sigma}: \sigma \rightarrow L^\pi_{G}$. 
    \item {[Asymptotic conditions]} For any $t_1\in [-1,0]$ (resp. $t_2\in [0,1]$), the following limits 
    $$\lim_{\tau\rightarrow \pm \infty}u_1(t_1, \tau) \eqqcolon q_{\pm\infty} \in X \sq G, \quad \lim_{\tau\rightarrow \pm \infty}u_2(t_2, \tau)\eqqcolon p_{\pm\infty} \in X_G$$ exist and are independent of $t_i$. 
  The seam conditions forces 
    \begin{align*}
        (q_{+\infty},  p_{+\infty}) & \in L^\pi_G\cap (\underline{L}^{'}_0 \times L_{0, G})\\
        (q_{-\infty},  p_{-\infty})& \in L^\pi_G\cap (\underline{L}^{'}_{k'} \times L_{k, G}). 
    \end{align*}
    These limiting points correspond either to generators of the equivariant Floer complexes or connecting to  critical points of Morse functions, depending on whether we work in the Floer or Morse--Bott model (cf. Remark~\ref{rem:Morse_Bott_models}).
    \item{[Stability conditions]} 
    The automorphism group $\Aut(u)$ is finite.
    \end{itemize}
\end{itemize}
Suppose all tuples are $G$-composable in the sense of Definition \ref{defn:G_composable_intersections}. After choosing the  perturbation data generically, transversality is achieved. Assume further that the sequences satisfy $L_{0}> \cdots > L_{k-1}> L_{k}$ and $\underline{L}_{k'}'>\cdots \underline{L}_{1}'>\underline{L}_{0}'$ so that the  maximum principal applies. We apply the Gromov--Floer compactification and  count the zero-dimensional moduli spaces, this  defines the bimodule structure. An  argument analogous to that  in Lemma \ref{lem:approximation_commutative} shows that these bimodule strucutres commute with the projection $P_{N}$, and we can take the inverse limit to obtain the following bimodule strucutre. 
\begin{align*}
  \mu^{k|1|k'}:CF_{G}^\bullet(L_{k}, L_{k-1})\otimes \cdots \otimes CF_{G}^\bullet(L_{1}, L_{0}) \otimes CF_{G}^\bullet(L_0, L^{\pi}, \underline{L}'_{0})\otimes CF^{\bullet}(\underline{L}_{0}', \underline{L}_{1}' )\otimes \cdots \\
   \otimes CF^\bullet (\underline{L}_{k'-1}', \underline{L}_{k'}') \rightarrow CF_{G}^\bullet(L_k, L^{\pi},\underline{L}_{k'}'). 
\end{align*}
As before, we can choose the Floer data coherently across all moduli spaces, which produces an  $(\cO_{G},\cO_{X\sq G})$-bimodule
\begin{equation}
    \label{eq: defn Q}
    \cQ \coloneqq \left\{CF_{G}^\bullet(-, L^{\pi},-),  \left\{ \mu^{k|1|k'} \right\}_{k, k' \in \Z_{\geq0}} \right\}.
\end{equation} 
Localizing with respect to quasi-units $\cZ_{G}$ and $\cZ_{X\sq G}$ defines a $(\cW_{G}(X), \cW(X\sq G))$-bimodule, which we denote by 
\begin{equation}
    \label{eq: defn B_G}
    \cB_G \coloneqq {}_{\cZ_{G}^{-1}}{\cQ}_{\cZ^{-1}_{X\sq G}}.
\end{equation} 
We remark that because  the wrapping Hamiltonian $H$ is $G$-invariant, its flow induces wrapping both in  $X$ and  in $X\sq G$. Thus the wrapping sequences in   $\cO_{G}$ and  $\cO_{X\sq G}$ agree under the natural identification of  $L_{i}^k\Leftrightarrow \underline{L}_{i}^k$.
\begin{defn}(\cite{Fukaya2023})
\label{defn: cyclic elt}
Let $L$ be a $G$-invariant Lagrangian. We say an element $\mathbf {1}_L \in CF^0_G(L, L^{\pi}, \underline L)$ is:
\begin{enumerate}
    \item \emph{left cyclic} if it is $\mu^{0|1|0}$-closed and 
    \begin{equation*}
        \mu^{1|1|0}(-, \mathbf{1}_L): CF_{G}^{\bullet}(K, L) \rightarrow CF^{\bullet}_{G}(K, L^{\pi},\underline{L})
    \end{equation*}
    is a quasi-isomorphism for all $K$.
    \item \emph{right cyclic} if it is $\mu^{0|1|0}$-closed and 
    \begin{equation*}
        \mu^{0|1|1}(\mathbf{1}_L, -)\colon CF^{\bullet}(\underline{L}, \underline{K}) \rightarrow CF^{\bullet}_{G}(L, L^{\pi},\underline{K})
    \end{equation*}
    is a quasi-isomorphism for all $\underline{K}$.
    \item \emph{bicyclic} if it is both left cyclic and right cyclic.
\end{enumerate}
\end{defn}
\begin{rem}
    \label{rem: modulemap}
    From the $A_\infty$-bimodule equations satisfied by the operations $\{\mu^{k|1|k'}\}$, it follows that the collection of maps 
\begin{align*}
    \{\mu^{k|1|0}(-, \mathbf{a})\}    
\end{align*}
    defines a morphism of left $A_\infty$-modules 
\begin{align*}
        CF^{\bullet}_G(-, L) \rightarrow CF^{\bullet}_G(-, L^{\pi}, \underline{L})
\end{align*}
    whenever $\mathbf{a}$ is $\mu^{0|1|0}$-closed.  
    In particular, having a \emph{left cyclic element} $\mathbf{1}_L$ is equivalent to requiring that
\begin{align*}
        \{\mu^{k|1|0}(-, \mathbf{1}_L)\} : CF^{\bullet}_G(-, L) \longrightarrow CF^{\bullet}_G(-, L^{\pi}, \underline{L})
\end{align*}
    is a quasi-isomorphism of left $\mathcal{O}_G$-modules.  
    Analogous characterizations for right cyclic and bicyclic elements are formulated in the same way.
\end{rem}

By the connectedness assumption on $L$, $L^{\pi}$ intersects $L\times \underline{L}$ cleanly along a single connected component $C_0$, namely $L$ itself. Hence the $E_1$-page of the spectral sequence \eqref{equ:G_spectral_sequence} has only one nonzero row $E_1^{0,q}$. Therefore, it degenerates at the $E_{1}$-page, yielding an isomorphism 
\begin{align}
\label{equ:identification_between_equviariant_cohomology_and_self_intersection}
H_G^\bullet(L) \cong HF_{G}^{\bullet, \mathrm{loc}}(L^{\pi}, L\times \underline{L}; C_0) \stackrel{\eqref{equ:G_spectral_sequence}}{\cong} HF^{\bullet}_{G}(L, L^{\pi},\underline{L}).
\end{align}
The first map is defined by a PSS-type map supported near $C_0$, in complete analogy with the equivariant PSS map from Section~\ref{sec:G_equivariant_PSS_isomorphism} (cf.~Remark~\ref{rem:Morse_Bott_models}).  
We denote the cochain-level map  of  \eqref{equ:identification_between_equviariant_cohomology_and_self_intersection} by
\begin{align*}
    \Psi_L \colon C_G^\bullet(L) \rightarrow 
    CF_G^\bullet(L^\pi,\, L \times \underline{L};\, C_0).
\end{align*}
\begin{defn}
    Define
    \begin{equation*}
        \mathbf 1_L \coloneqq \Psi_L(c_{\min}, x_{\min}) \in CF_G^\bullet(L^{\pi}, (L\times \underline{L}))
    \end{equation*}
    where $c_{\min}$ is the sum of minima of the Morse functions chosen on $0_{BG}$ and $x_{\min}$ is the sum of minima of the Morse function on $L$, viewed as the  connected component of $L^{\pi} \cap (L \times \underline{L})$.
\end{defn}
In other words, $\mathbf 1_L$ is the cochain level representative of the unit transferred via \eqref{equ:identification_between_equviariant_cohomology_and_self_intersection}. Also notice that by Condition (5) of Definition \ref{thm:family_of_Morse_functions} and the assumption that the Morse function on $L$ is linear near infinity, there are only finitely many critical points contributing to these sums. Therefore, $\mathbf1_L$ is well-defined.

\begin{lemma}
\label{lem: bicyclic}
 The element  $\mathbf{1}_L$ is bicyclic. 
\end{lemma}
\begin{proof}
    By construction, $\mathbf{1}_L$ is closed. To establish right cyclicity, it suffices to show that 
    \begin{align}
        \label{equ:from_downstairs_to_bimodules}
        \mu^{0|1|1}(\mathbf{1}_{L}, -): CF^{\bullet}(\underline{L}, \underline{K}) \rightarrow CF^{\bullet}_{G}(L, L^{\pi},\underline{K}).
    \end{align}
    induces isomorphisms between cohomologies.
    
    Because the action functionals on both sides agree under the identifications of Remark~\ref{rem:primitive_identification}, the map \eqref{equ:from_downstairs_to_bimodules} preserves the energy filtrations. Therefore, it induces a morphism of  the corresponding spectral sequences. On the $E_{1}$-page, it is given by 
    $$
    [\mu^{0|1|1, \mathrm{loc}}(\mathbf{1}_L, -)_{C_l}] : HF^{\bullet,\mathrm{loc}}(\underline{L}, \underline{K}; C_l/G) \to HF^{\bullet, \mathrm{loc}}_G(L^\pi, L\times \underline{K}; C_l),
    $$
where $C_{l}$ denotes a connected component of the clean intersection, and   $\mu^{0|1|1, \mathrm{loc}}(\mathbf{1}_L, -)_{C_l}$ is defined by the same moduli spaces as in  \eqref{equ:from_downstairs_to_bimodules}, except  localized to a sufficiently small neighborhood of $C_l$. 

Recall that, by \eqref{equ:reduced_Lagrangians}, both sides are identified with $H_{G}^{\bullet}(C_{l}) \cong H^{\bullet}(C_{l}/G)$ as vector spaces. Therefore, it suffices to show that $[\mu^{0|1|1, \mathrm{loc}}(\mathbf{1}_L, -)_{C_l}]$ is \emph{injective}.

    To this end, we define a map $F: CF^{\bullet}_{G}(L, L^{\pi},\underline{K})  \rightarrow CF^{\bullet}(\underline{L}, \underline{K})$ by reversing the roles of the input and output of $\mu^{0|1|1}(-,-)$ and we show that $F$ provides a left inverse to  $ [\mu^{0|1|1, \mathrm{loc}}(\mathbf{1}_L, -)_{C_l}]$; see Figure \ref{fig:left_inverse_of_chain_map}.

    \begin{figure}[ht]
    \def\svgwidth{0.6\textwidth}
	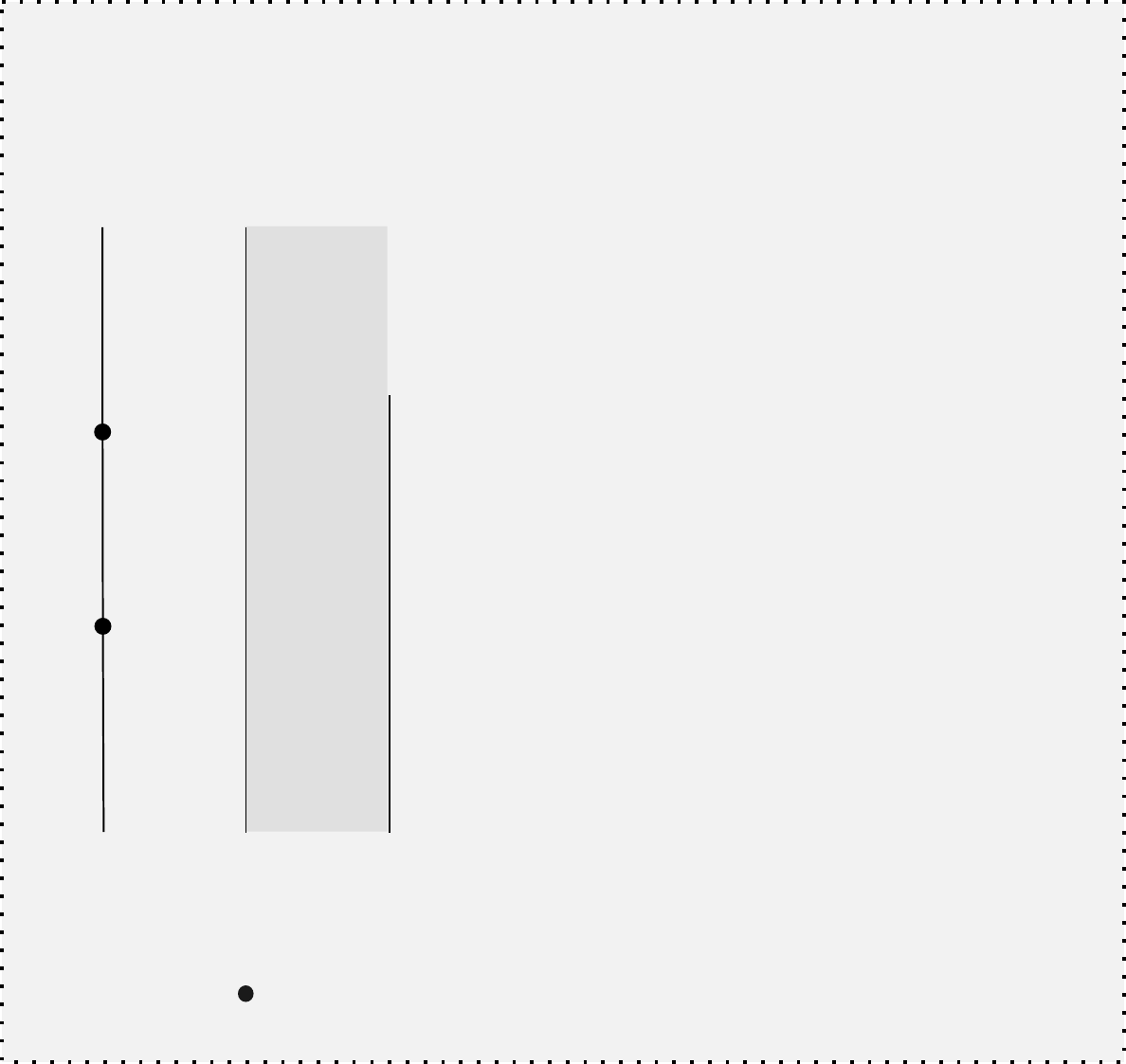
	\caption{Composition of $[\mu^{0|1|1}(\mathbf{1}_{L}, -)]$ and its left inverse.}
\label{fig:left_inverse_of_chain_map}
\end{figure}
    
    As depicted by Figure \ref{fig:left_inverse_of_chain_map}, the composition  $$F\circ\mu^{0|1|1}(\mathbf{1}_{L}, -)$$ is homotopic to $$\mu_{\cO_{X\sq G}}^2(F(\mathbf 1_L), -).$$ Passing to the spectral sequence, we obtain 
    \begin{equation}
    \label{eq: Fmu2eq}
    [F^{\mathrm{loc}}]\circ[\mu^{0|1|1}(\mathbf{1}_{L}, -)] \sim [\mu_{\cO_{X\sq G}}^{2, \mathrm{loc}}]\left([F(\mathbf{1}_L)], -\right),
    \end{equation}
    where $F^{\mathrm{loc}}$ and $\mu_{\cO_{X\sq G}}^{2,\mathrm{loc}}$ are defined by restricting $F$ and $\mu_{\cO_{X\sq G}}^{2}$ to those contributions whose outputs lie in the component $C_l$.
    
We claim that 
    \begin{equation}
    \label{eq: computeF(1)}
        [F(\mathbf 1_L)] = e_{\underline{L}}\in HF^{\bullet}(\underline{L}, \underline{L}).
    \end{equation}
    Indeed,  when both input and output are the element $\mathbf{1}_{L}$, the only quilted configurations that can contribute to the strip component in Figure~\ref{fig:left_inverse_of_chain_map} are constant quilts.  
   Therefore, the relevant moduli space is identified with constant Morse trajectories at the minimum $(c_{\min},x_{\min})$, which implies that $F(\mathbf{1}_L)$ represents the unit class.

    Combining \eqref{eq: Fmu2eq} with  \eqref{eq: computeF(1)}, we conclude that
    $$[F^{\mathrm{loc}}]\circ[\mu^{0|1|1}(\mathbf{1}_{L}, -)] = \mathrm{Id}_{HF^{\bullet, \mathrm{loc}}(\underline{L}, \underline{K},  C_l/G)}.$$  
    Thus  $[\mu^{0|1|1, \mathrm{loc}}(\mathbf{1}_{L}, -)]$ is injective.  This completes the proof of the right cyclic property of $\mathbf 1_L$. 
    
    A similar argument applied to
    \begin{align}
        \label{equ:from_G_to_bimodule}
        \mu^{1|1|0}(-,\mathbf{1}_{L}): CF_{G}^{\bullet}(K, L) \rightarrow CF^{\bullet}_{G}(K, L^{\pi},\underline{L})
    \end{align}
    shows that it is also a quasi-isomorphism. The proof is now complete. 
\end{proof}

\begin{cor}
\label{cor: Qfunctor}
    The $A_\infty$-functor defined by the bimodule $\cQ$,
\begin{align*}
        \cQ \colon \cO_G \longrightarrow \mathrm{Mod}_{\cO_{X\sslash G}}, 
        \qquad 
        L \longmapsto \cQ(L,-),
\end{align*}
lands in  representable right $\cO_{X\sslash G}$-modules, each of which is represented by the reduced Lagrangian $\underline{L}$.  
    
    Similarly, the functor
\begin{align*}
        \cQ \colon \cO_{X\sslash G} \longrightarrow \mathrm{Mod}^{\cO_G}, 
        \qquad 
        \underline{K} \longmapsto \cQ(-,\underline{K}),
\end{align*}
    lands in representable left $\cO_G$-modules, represented by the G-invariant Lagrangian $K$.
    \end{cor}
\begin{proof}
By the bicyclicity of the element $\mathbf 1_L$, we obtain quasi-isomorphisms
\begin{align}
    \label{equ:from_G_to_downstairs}
    \cQ(-, \underline{K}) & = CF^\bullet_{G}(-, L^\pi,\underline{K})\xleftarrow[\sim]{\left\{ \mu^{k|1|0}(-, \mathbf{1}_{K}) \right\}}CF_{G}^{\bullet}(-, K),  \\    
    \label{equ:from_downstairs_to_G} 
    CF^{\bullet}(\underline{L}, -) \xrightarrow[\sim]{\left\{ \mu^{0|1|k}(\mathbf{1}_{L},-) \right\}}\cQ(L, -) & = CF^{\bullet}_{G}(L, L^{\pi}, -). 
 \end{align}
between left $\cO_G$-modules and right $\cO_{X\sslash G}$-modules, respectively; see Remark~\ref{rem: modulemap}.  
Here ``quasi-isomorphism'' means a pointwise quasi-isomorphism, i.e.\ for each test object the maps  
$\mu^{1|1|0}(-,\mathbf{1}_K)$ and $\mu^{0|1|1}(\mathbf{1}_L,-)$ induce quasi-isomorphisms of cochain complexes.

Thus the  cohomology functor  $H^{\bullet}\cQ(-, \underline{K})$
is pointwise isomorphic to the left $ H^{\bullet}\cO_{G}$-module obtained via the Yoneda embedding 
\begin{align*}
    H^{\bullet}\cO_{G} \hookrightarrow  H^{\bullet}\mathrm{Mod}^{\cO_G}, 
    \quad 
    K \longmapsto H\cO^{\bullet}_{G}(-, K).    
\end{align*}
Since we work over a field (so all $\cO_G$-modules are cofibrant), this cohomology representability is sufficient to conclude that the $A_{\infty}$-module $\cQ(-, \underline{K})$ is representable, This follows from an $A_\infty$-version of Whitehead’s theorem; see \cite[Appendix~A]{GPS3}.  An identical argument applies to the other functor, showing that it is likewise representable.
\end{proof}

In particular, combining \eqref{equ:from_G_to_downstairs} with \eqref{equ:from_downstairs_to_G},  we obtain a quasi-isomorphism of cochain complexes
\begin{align}
    \label{equ:cochain_level_isomorphism}
    CF^{\bullet}(\underline{L},\underline{K}) \;\xrightarrow{\;\simeq\;}\; CF_{G}^{\bullet}(L, K).
\end{align}

A similar argument to that of Proposition \ref{prop:commutativity_of_wrapping_and_approximations} and Corollary \ref{cor:G_equivariant_homotopy_commutative} gives rise to  the following corollary.
\begin{cor}
    \label{prop:left_local_property}
    The following diagram 
    \begin{equation}
        \label{equ:left_local_property}
         \begin{tikzcd}[row sep={70,between origins}, column sep={120,between origins}]
            \cQ(L^{k+1}, \underline{K}^{j}) \arrow[r] \arrow[d] & |[alias=youshangjiao]| \cQ(L^{k+1}, \underline{K}^{j+1}) \arrow[d] \\
            \cQ(L^{k}, \underline{K}^{j})  \arrow[to=youshangjiao, "F^G_{L_{0}^k, L_{1}^j}"]\arrow[r] &  \cQ(L^{k}, \underline{K}^{j+1})
         \end{tikzcd}
     \end{equation} 
     commutes up to homotopy. 
     Here, the horizontal and vertical arrows are given by multiplication by the  images of the tensor products of units and quasi-units under the K\"unneth isomorphism  $$HF^{\bullet}_{G}(L'\times \underline{K}',L\times \underline{K}) \cong HF^{\bullet}_{G}(L',L)\otimes HF^{\bullet}(\underline{K}',\underline{K}).$$  For example, the upper horizontal arrow is given  by multiplication with  $\mathrm{id}_{L^{k+1}}\otimes e_{\underline{K}^{j}}$ where $\mathrm{id}_{L^{k+1}}$ denote the unit in $HF_{G}^{0}(L^{k+1},L^{k+1})$, and  $e_{\underline{K}^{j}}$ denotes the quasi-unit in $HF^{0}(\underline{K}^{j},\underline{K}^{j+1})$. 

     In particular, this implies the left locality of the direct limit, i.e.
     \begin{equation}
        \label{equ:left_locality_property}
        \varinjlim_{j} H\cQ^{\bullet}(L^{k+1}, \underline{K}^{j}) \cong         \varinjlim_{j} H\cQ^{\bullet}(L^{k}, \underline{K}^{j}).
     \end{equation}
    \end{cor}
    
Once the isomorphism \eqref{equ:cochain_level_isomorphism} and its functoriality have been established, passing to the localized category (wrapped Fukaya cateogry) is immediate and entirely formal. We follow the localization framework of \cite{GPS3}, adapting their construction to the present equivariant setting.

\begin{thm}
\label{prop: GoingDown}
The $A_\infty$-functor associated to the bimodule $\cB_G$ defined in \eqref{eq: defn B_G},
\begin{align*}
    \cB_G: \cW_{G}(X) \rightarrow \mathrm{Mod}_{\cW(X\sq G)}, \quad L \mapsto \cB_G(L, -)\coloneqq{}_{\cZ_{G}^{-1}}{\cQ}_{\cZ^{-1}_{X\sq G}}(L,-)
\end{align*}
lands in representable right modules. Thus, it induces an $A_{\infty}$-functor 
\begin{align*}
    \Phi_{\downarrow}: \cW_{G}(X) \rightarrow \cW(X\sq G),
\end{align*}
well-defined up to quasi-equivalence. 
\end{thm}
\begin{proof}
A module analogue of the argument in Corollary  \ref{cor:colimit_compute_localization} (see \cite[Corollary 3.38]{GPS1}) implies that the cohomology of the right localization ${\cQ}_{\cZ^{-1}_{X\sq G}}$  is computed by taking colimit over the wrapping sequence in $\cO_{X\sq G}$.   Combined with  \eqref{equ:from_downstairs_to_G}, we obtain

\begin{align*}
    H^{\bullet}{\cQ}_{\cZ^{-1}_{X\sq G}}(L,-) \simeq HW^{\bullet}(\underline{L},-) \in \mathrm{Mod}_{H^\bullet\cW(X\sq G)}
\end{align*}
 In particular, ${\cQ}_{\cZ^{-1}_{X\sq G}}$ is  $\cZ_{G}$-local on the left, and therefore the  map ${\cQ}_{\cZ^{-1}_{X\sq G}}(L,-) \rightarrow \cB_{G}(L,-)$ is a quasi-isomorphism by \cite[Lemma 3.13]{GPS1}.

Finally, since $CF_{\cZ^{-1}_{X\sq G}}^{\bullet}(\underline{L}, -) $ is quasi-isomorphic to  ${\cQ}_{\cZ^{-1}_{X\sq G}}(L,-)$ and is, by definition, the Yoneda module represented by $\underline{L}$, we conclude that $\cB_{G}(L,-)$ is represented by $\underline{L}$ by \cite[Appendix.A3]{GPS3}. This completes the proof. 
\end{proof}

A similar argument applied to $\underline{L}\rightarrow \cB_{G}(-, \underline{L}) $ produces the functor in the opposite direction, namely an $A_{\infty}$-functor  
\begin{align*}
    \Psi_{\uparrow}:  \cW(X\sq G) \rightarrow  \cW_{G}(X), 
\end{align*}
 which sends a reduced Lagrangian $\underline{L}$ to its preimage $L$. 

\section{Symplectic quotient with singularities}
In the previous section, we assumed that the Hamiltonian group action on the zero level set of the moment map is free. Unfortunately, this assumption fails in many interesting examples, leading to a singular exact symplectic quotient.  
In this section, we focus on the simplest type of such singularities, arising from the quotient of $\C^{2}$ by the Hamiltonian $\bS^{1}$-action with weights $(1,-1)$. 

\subsection{SNC divisor complement} 
\label{subsec: SNCdivcmpl}
We briefly recall a construction from \cite{ML12} of a Liouville structure on the complement of a divisor. Let 
 $$ 
 K = Z(s) \subset Y
 $$ 
 be a simple normal crossing divisor in a smooth projective variety $Y$, determined by a very ample line bundle $\cL$. Here $s$ denotes a defining section and we write the  smooth irreducible components of $K$ as $K_i$.  We denote the complement  of the divisor by $X \coloneqq Y\setminus K$. Note that $X$ is a complex affine variety, and any complex affine variety can be presented in this way.  

Fix a Hermitian metric $\lVert \cdot \rVert$ on $\cL$, and define 
\begin{equation}
    \label{defn: omegalambda}
    \omega \coloneqq -dd^c\log \lVert s \rVert^2, \quad \lambda \coloneqq  -d^c\log \lVert s \rVert^2.
\end{equation}
Choose a regular open neighborhood $U \supset K $ that is diffeomorphic to (a smoothing of corners of) the union $\bigcup_{i} N^{\leq \epsilon}_{K_i/Y}$ of disk bundles. The positivity of $K$ ensures that $\lambda$ defines a Liouville form on $Y\setminus U$ with respect to the symplectic form $\omega$. By abuse of notation, we denote this Liouville domain by $X$, and denote its completion by $\hat X$.

This particular Liouville form $\lambda$ has simple Reeb dynamics at $\partial X$. There is a projection
    \begin{equation*}
        p: \partial X \to K
    \end{equation*}
whose restriction to $K_i^\circ \coloneqq K_i - \bigcup_{j\neq i}K_j$ is a principal $S^1$-bundle by \cite[Section 5]{ML12}. Its angular form is given by
$$
\lambda \vert_{p^{-1}(K_i^\circ)} = -\Theta, \quad \Theta \coloneqq \text{Hermitian connection 1-form of $\Vert \cdot \Vert$}.
$$
This means that the induced $S^1$-action is generated by the Reeb flow of the minimal period $2\pi$. This is a consequence of the simplicity of the divisor $K$.

\begin{rem}
    The $S^1$-action generated by the Reeb flow is distinct from the given Hamiltonian $\bS^1$-action, hence we use a  different notation. Note that the above construction extends to the equivariant setting, provided that the divisor $K$ is invariant under the group action under consideration.
\end{rem}
\subsection{Conic fibration}
Let $\underline{X}$ be a smooth quasi-projective variety over $\C$ equipped with a positive line bundle $\mathcal L \to \underline X$. Let $\underline D$ be a smooth ample divisor defined by a section $s\in \Gamma(\underline X, \mathcal L^{\otimes 2}).$ The \emph{conic fibration} associated with this data is the following variety with a prescribed $\bS^1$-action, which canonically extends to K\"ahler $\C^*$-action:
\begin{equation*}
       \bS^1 \circlearrowright X \coloneqq  Z(zw=s) \subset \mathrm{Tot}(\mathcal L^{\oplus 2}), \quad c \cdot (\mathbf x, z, w) = (\mathbf x, c z, c^{-1}w),
\end{equation*}
where $(z,w)$ denotes the fiber coordinates. As its name suggests, $X$ comes with a conic fibration $X \xrightarrow{\pi} \underline X$ induced from $\mathcal L^{\oplus 2} \rightarrow \underline X$. The pullback divisor $D \coloneqq \pi^{-1}(\underline D)$ is a normal crossing divisor with two irreducible components, 
$$
D = D_1 + D_2,\quad D_1 \coloneqq \mathrm{Tot}_{\underline{D}}(\mathcal L \oplus 0), \quad D_2\coloneqq \mathrm{Tot}_{\underline{D}}(0\oplus \mathcal L).
$$
In the language of geometric invariant theory, $\C^*$-orbits of $X$ are of the following types: 
\begin{itemize}
    \item the stable locus $X^s \coloneqq X \setminus D$;
    \item the strictly semistable locus $X^{s.s} \coloneqq D$; and 
    \item the polystable locus $X^{p.s} \coloneqq D_1 \cap D_2 = \mathrm{Tot}_{\underline{D}}(0\oplus 0)$. 
\end{itemize} 
We introduce some notations that will be used throughout the rest of the section. We denote the polystable locus by $F$, emphasizing that it coincides with the fixed-point locus of the $\C^*$-action. Let $N$ be a regular neighborhood of $F$, symplectomorphic to the disk bundle $DN_{F/X}$. Locally near the $F$, $X$ is modeled on a product with the standard Lefschetz fibration $zw: \C^2 \to \C$. If $V$ is a small trivializing neighborhood of $\mathcal L\vert_{\underline{D}}$ , then we have 
\begin{equation}
\label{eq: LocModelConicFib}
\begin{tikzcd}
    U\coloneqq\pi^{-1}(V) \ar[r, "\sim"]&V\times \mathbb D^2_{z,w} \ar[d, "zw"] \ar[r, hook] & N \simeq DN_{F/X} \ar[d]\\
    &V \times \mathbb D \ar[r, hook] & DN_{\underline{D}/\underline{X}}
\end{tikzcd}.
\end{equation}
We denote the regular locus of $X$ and its quotient, respectively, by
$$ 
X^{\circ} \coloneqq X^s, \quad \underline{X}^\circ \coloneqq \underline{X}\setminus \underline{D}.
$$ 
From now on, we impose the following assumptions on  $\underline X$:
\begin{itemize}
    \item there exists a smooth projective compactification $\underline{Y}\supset\underline{X}$ with a SNC boundary divisor $\underline{K}$, such that $\mathcal{L}\rightarrow \underline{X}$ extends to a positive line bundle $\mathcal{L}\rightarrow\underline{Y}$ (still denoted $\mathcal{L}$ by abuse of notation);
    \item  the space $\underline{X}$ is equipped with a Liouville structure described in \eqref{defn: omegalambda};
    \item  $s$ extends to a section of $\mathcal{L}^{\otimes 2}\rightarrow\underline{Y}$, still denoted as $s$, such that  $Z(s) \subset \underline{Y}$ is a smooth ample divisor extending $\underline{D}\subset \underline{X}$, still denoted as $\underline{D}$.
\end{itemize}
Under these assumptions, we can realize $X$ as a divisor complement as follows. Define 
$$
Y\coloneqq Z(zw=su^2) \subset \mathbb P_{\underline{Y}}(\mathcal L^{\oplus2} \oplus \mathcal O),
$$
equipped with a natural $\bS^1$-action extending that on $X$, and a $\mathbb P^1$-fibration $Y \xrightarrow{\pi} \underline{Y}$ extending the conic fibration $X \xrightarrow{\pi} \underline X$ (still denoted as $\pi$). It is straightforward to check that $X = Y \setminus(K + H_\infty)$ where $K \coloneqq \pi^{-1}(\underline{K}) \subset Y$ and $H_\infty$ is the hyperplane divisor at infinity. Note that both divisors are $\bS^1$-invariant. 

Next, we describe an $\bS^1$-equivariant Liouville structure on 
$$
X^{\circ} = Y\setminus (K+D+H_\infty) = X\setminus D.
$$ 
The construction is motivated by Kirwan's desingularization process \cite{Kir85}. We first perform an equivariant blow-up along the polystable locus,
\begin{equation*}
    \begin{tikzcd}
    \bS^1 \circlearrowright\mathrm{Bl}_{F}X \ar[d, "\pi"] \ar[dr, "\tilde \mu"] &  \\
    X \ar[r, "\mu"] & \R,
    \end{tikzcd}
\end{equation*}
equipped with the lifted moment map $\tilde \mu$. In the original desingularization process, one removes the strict transform $\widetilde X^{s.s}$ of the strictly semistable locus $X^{s.s}$. In our setting, we further excise a neighborhood of the exceptional divisor. More precisely, we define 
\begin{equation*}
    X^{\circ} \cong \mathrm{Bl}_{F}X \setminus (D+E).
\end{equation*}
This is symplectomorphic to the original regular locus $X^\circ = X^s$ by \cite{Kir85}. Since $K+D+E +H_\infty$ is very ample and simple normal crossing, we equip $X^\circ$ with the Liouville form defined in \eqref{defn: omegalambda}, and denote the resulting Liouville manifold by 
\begin{equation*}
(X^\circ, \lambda^\circ).    
\end{equation*}
Because the divisor is $\bS^1$-equivariant and the $\bS^1$-action is free on its complement, $\lambda^\circ$ descends to a Liouville form $(\underline X^{\circ}, \underline \lambda^{\circ})$ by Theorem \ref{thm:Liouville_reduction}. 

Since $X$ is already realized as a complement of divisors, it is worth explaining the role of the blow-up. Locally, the blow-up diagram above is modeled on
\begin{equation}
   V \times 
   \left( \begin{tikzcd}
    ((\C^*)^{2\circ}, \lambda^{\circ}) \ar[r, hook] \ar[d, hook, "i"] & (\mathrm{Bl}_{0}\C^2, \widetilde \omega) \ar[d, "p"] \ar[dr, "\tilde \mu"] &  \\
    ((\C^*)^2, \lambda_{\mathrm{std}}) \ar[r, hook ] & (\C^2, \omega_{\C^2} )\ar[r, "\mu"] & \R
\end{tikzcd} \right).
\end{equation}
We focus on the symplectic and Liouville structures on the right-hand side of the diagram. The form $\widetilde \omega$ agrees with the pullback of $\omega_{\C^2}$ away from the blow-up locus. On the other hand, the Liouville form is modified along the  Liouville isotopy $i: (\C^*)^{2\circ} \hookrightarrow (\C^*)^2$.
 
The advantages of the Liouville form $\lambda^{\circ}$ over $\lambda_{\mathrm{std}}$ are as follows. First, $\lambda^{\circ}$ simplifies the Reeb dynamics near the singular locus of $\mu$. Observe that $\tilde \mu^{-1}(0)$ intersects $D+E$ only along the smooth locus $E^\circ$ of $E$. Explicitly, a neighborhood of $\partial X^{\circ}$ near $\tilde \mu^{-1}(0)$ is modeled on 
\begin{equation}
\label{eq: boundaryloc}
    V \times (S^3\setminus p^{-1}(W)) \xrightarrow{(\mathrm{id}, p)} V \times \mathbb (\mathbb P^1\setminus W), \quad  W \text{ a small neighborhood of } \{0,\infty\}.
\end{equation}
The $\bS^1$-action is given by $\lambda\cdot (\mathbf x,z, w) \mapsto (\mathbf x, \lambda z, \lambda^{-1} w)$, whereas the $S^1$-action generated by the Reeb flow is  $\eta\cdot (\mathbf x,z, w) \mapsto (\mathbf x, \eta z, \eta w)$.

Second, $\lambda^\circ$ descends under Liouville reduction to a Liouville form $(\underline X, \underline \lambda^{\circ})$, and admits a similar interpretation. This is closely related to the fact that the resolution
\begin{equation*}
    \tilde \mu^{-1}(0)/\bS^1 \to \mu^{-1}(0)/\bS^1
\end{equation*} 
is a homeomorphism and the left hand side carries a well-defined symplectic form. The inverse image of the singular locus is precisely the divisor $\underline{D} \subset \tilde \mu^{-1}(0)/\bS^1$. The $S^1$-action on $p^{-1}(E^\circ)$ generated by the Reeb flow is transverse to the given $\bS^1$-action. Thus, the Reeb flow of $\underline{\lambda}^\circ$ is periodic of minimal period $2\pi$. This implies  $\underline \lambda^{\circ}$ coincides with the one induced by the presentation of $\underline{X}^{\circ}$ as a divisor complement: 
\begin{equation*}
    \underline X^{\circ} \cong \tilde \mu^{-1}(0)/\bS^1 - \underline{D}.
\end{equation*}

\subsubsection*{Guiding example}
We explain what happens in the simplest example $zw: \C^2 \to \C$, for which we have $\hat X^\circ = (\C^*)^{2\circ}, \hat{\underline{X}}^\circ = \C^*.$
Our goal is to compare the Floer cohomologies associated to the Lagrangians
$$ L=\mathbb D^2_{\Re} \subset \C^2, \quad L^\circ = L \cap X^\circ\subset X^\circ, \quad \text{and} \quad \underline{L}^\circ = L^\circ /\bS^1 = \R_{>0} \subset \C^*.
$$
Here, $\mathbb D_\Re$ denotes the Lefschetz thimbles over $\mathbb R_{\geq 0}$. 

If we view the base $\C$ of the map $zw$ as a quotient $\underline{X}$ of $\bS^1$-action, then the origin has a conical singularity on which the symplectic form blows up. To resolve the singularity, consider $\underline{X}^\circ$, an annulus around the origin in $\C$. Let $\theta$ be the coordinate of one of the contact boundary $S^1$ of $\underline{X}^\circ$ near the origin and $r\in(0,\infty)$ be the radial coordinate of its symplectization. The desingularization process amounts to taking a quotient $S^1\times (0,1]$ by $S^1$-action generated by Reeb flow at $r=1$, recovers a standard symplectic fom on $\C$. The map $zw$ becomes a Lefschetz fibration. 

Fix a small interval $(\delta_0, \delta_1) \subset (0,1)$ to view $S^1 \times (\delta_0, \delta_1)$ as a collar neighborhood of the contact boundary. For each $\nu\in \mathbb N$, choose Hamiltonians on $\C$ satisfying
    \begin{equation}
    \label{eq: HamModel}
        \underline{K}_\nu = 
        \left\{\begin{array}{cc}
        (\nu+\epsilon) r + C_\nu& (\delta_1<r) \\
        \epsilon r& (r<\delta_0)
        \end{array}
        \right ., \quad \quad  \begin{array}{l}
             K_\nu(z,w) \coloneqq \underline{K}_\nu\circ zw,\\
             K^\circ_\nu\coloneqq K_\nu |_{X^\circ}.
        \end{array}
    \end{equation}
in such a way that $\underline{K}_\nu$ depends only on $r$ and strictly convex along $\delta_0 < r<\delta_1$. We extend $\underline{K}_\nu$ to the whole $\underline{X}^\circ$ so that $\underline L^\circ \cap \phi^1_{\underline{K}_\nu^\circ}(\underline{L}^\circ)$ intersect exactly once outside the collar neighborhood; see Figure \ref{fig:wrapped}. Hamiltonian chords of $\underline{K}_\nu$ are either Reeb orbits in the collar neighborhood, or Morse critical points in the interior. By extending the linear part of $\underline{K}_\nu$ to $r\to\infty$, this family will computes the wrapped Floer cohomology of $L^\circ$. On the other hand, they also define Hamiltonians on the standard $\C$ (on which the coordinate $r$ does not extend beyond $1$). We can compute Floer cohomology of $L$ using $K_\nu$. 

The intersection $L\cap \phi^1_{K_\nu}(L)$ consists of two types of components: $\lfloor \nu \rfloor$ cleanly intersecting circles over $\underline{L}^\circ \cap\phi_{\underline{K}^\circ_\nu}^1(\underline{L}^\circ)$, and the origin. Using the identification in Remark \ref{rem:primitive_identification}, we can explicitly describe the Floer complexes as 
\begin{equation}
\label{eq:computeFloerCochain}
    \begin{split}
    CF^\bullet_{\bS^1}(L^\circ, L^\circ; K_\nu^\circ) &= \C[\lambda, t]\langle 1, x, \ldots x^{\lfloor \nu \rfloor}\rangle, \\
    CF^\bullet_{\bS^1}(L, L; K_\nu) &= \C[\lambda, t]\langle 1, x, \ldots x^{\lfloor \nu \rfloor}\rangle \oplus \C[\lambda] \langle p^{\nu+1}\rangle, \\
    \deg x= 2, \quad & \deg \lambda =2, \quad \deg t=1,\quad \deg p^{\lfloor \nu \rfloor+1} =2\nu+2.
    \end{split}
\end{equation}
Here, $t$ denotes a representative of the point class of $\bS^1$, $\lambda$ is the degree $2$ generator of $H_{\mathbb{S}^1}^{\bullet}(*)$, $p^{\lfloor \nu \rfloor+1}$ denotes the fixed point at the origin, and the element $x^k$ record the number of wrappings around the origin. See Figure \ref{fig:wrapped}. Their degrees are computed using the standard complex volume form $dz$ on $\C$. The differentials are given by 
\begin{equation}
\label{eq:computeFloerDiff}
\begin{split}
\mu^{1}_{L^\circ}\left(\sum_{k=0}^\nu c_{k,1} x^k t + \sum_{k=0}^\nu c_{k,0}x^k\right) &= \sum_{k=0}^\nu c_{k,1} x^k \lambda, \\  
\mu^1_L\left(\sum_{k=0}^\nu c_{k,1} x^k t + \sum_{k=0}^{\nu+1} c_{k,0}x^k\right) &= \sum_{k=0}^\nu c_{k,1} x^k \lambda \pm \sum_{k=0}^{\nu} c_{k,1} x^{k+1},
\end{split}
\end{equation}
 where we set $x^{\nu+1} \coloneqq p^{\nu+1}$ for notation convenience (see Figure \ref{fig:wrapped}). The first differential is standard. Indeed, there are no pseudo-holomorphic sections of $zw: X^\circ \to \C^*$ with the given Lagrangian boundary condition. Thus, Floer strips are confined to the fibers of this map, which implies $\mu^1_{L^\circ}(x^k) = 0$. The relation $\mu^1_{L^\circ}(x^kt)=x^k\lambda $ is an instance of the computation in \cite[Lemma 4.4]{KLZ23}. 
 
 In contrast, the second differential contains an additional term $\mu^1_{L}(x^kt)=x^k\lambda \pm x^{k+1}$. When $k={\lfloor \nu \rfloor}$, the strip contributing to $x^{k+1}$ can be  read directly from the base of the fibration: see Figure \ref{fig:wrapped}. Combining these relations, we compute the Floer cohomologies as follows: 
\begin{equation}
\label{eq:computeFloerCoh}
    \begin{split}
        HF_{\bS^1}^\bullet (L^\circ, L^\circ;K^\circ_\nu ) &\simeq \C\langle [1], [x], \ldots, [x^{\lfloor \nu \rfloor}]\rangle, \\
        HF_{\bS^1}^\bullet (L, L ;K_\nu ) &\simeq \C[\lambda], \quad [\lambda^k] \sim [x^k], \quad k\in 0, \ldots, {\lfloor \nu \rfloor}.
    \end{split}
\end{equation}
The first computation confirms Corollary \ref{cor: Qfunctor}, since it is isomorphic to $HF^\bullet(\underline{L}^\circ, \underline{L}^\circ ; \underline{K}_\nu)$. The second computation verifies Theorem \ref{thm:G-equivariant_PSS_isomorphism}; it is isomorphic to $H_{\bS^1}^\bullet(L) = \C[\lambda]$. Letting $\nu \to \infty$, we conclude that the natural inclusion (which is not a subcomplex)
$$CF^\bullet_{\bS^1}(L^\circ, L^\circ; K_\nu^\circ) \subset CF^\bullet_{\bS^1}(L, L; K_\nu), \quad x, \lambda, t \mapsto x, \lambda, t, $$
induces an isomorphism on cohomology, as observed in \cite{lekili2023equivariant}. The generators $x$ and $\lambda$ are intertwined, and we obtain:
\begin{equation}
\label{eq:isoFloerCoh}
     \C[x] \simeq \lim_{\nu\to\infty}HF^\bullet(\underline{L}^\circ, \underline{L}^\circ; \underline{K}_\nu) \simeq \lim_{\nu\to\infty}HF^\bullet_{\bS^1}(L^\circ, L^\circ; K_\nu) \simeq HF^\bullet_{\bS^1}(L, L) \simeq \C[\lambda].
\end{equation}
\begin{figure}[ht]
	\includegraphics[scale=0.6]{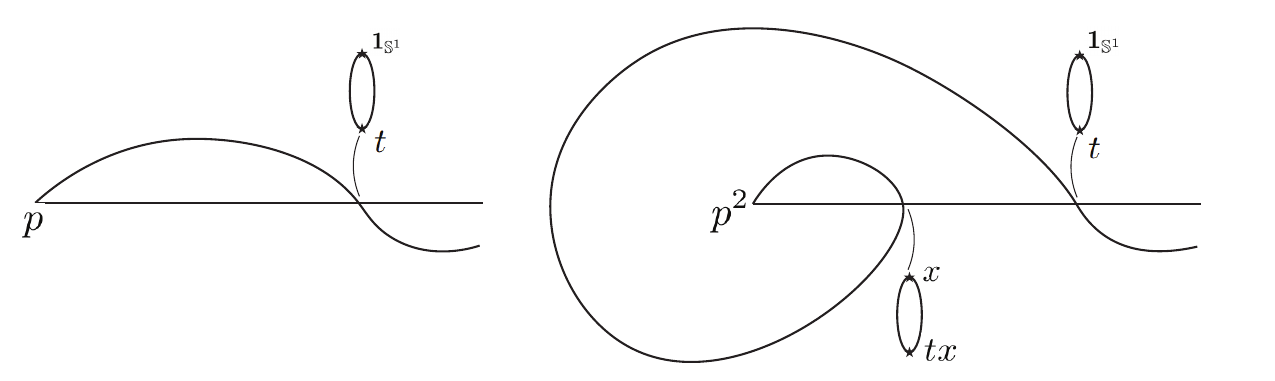}
	\caption{$CF^\bullet_{\bS^1}(L, L; K_\nu)$ for ${\lfloor \nu \rfloor}=0$ and ${\lfloor \nu \rfloor}=1$.}
	\label{fig:wrapped}
\end{figure}
%\begin{rem}
%    According to \cite{lekili2023equivariant}, the isomorphism \eqref{eq:isoFloerCoh} should be a consequence of the more general conjecture, which claims the triviality of the equivariant relative Fukaya category $\cW_{\bS^1}(Y, D)$, with one notable difference. A current The Lagrangian disc $L$ meets the divisor $D$. 
%\end{rem}
This computation generalizes to Lagrangians in a conic fibration. 
\begin{defn}
    Let $X$ be a conic fibration, and let $L\subset \mu^{-1}(0)$ be a connected exact $G$-invariant Lagrangian. We say that $L$ is of \emph{disk type near $F$} if, in the local model \eqref{eq: LocModelConicFib} near $F$, it can be described as
    \begin{equation}
    \left(
    \begin{tikzcd}
    \mathbb D_{\Re} \ar[r] & L\cap N \ar[d] \\
    & L'
    \end{tikzcd} 
    \right) \;\ \subset \;\ 
    \left(
    \begin{tikzcd}
    \mathbb D^4 \ar[r] & N\simeq DN_{F/X} \ar[d] \\
    & F
    \end{tikzcd}
    \right),
    \end{equation}
    i.e., a Lagrangian disk bundle over a Lagrangian $L'\subset F$ compatible with $DN_{F/X} \to F$.
    \end{defn}

\begin{lemma}
     Let $L \subset \mu^{-1}(0)$ be an exact, $\bS^1$-invariant Lagrangian $L \subset X$ which is of disk type near $F$. Then the Lagrangians
     \begin{equation*}
         L^{\circ}\coloneqq L\cap  X^{\circ}, \quad \underline L^{\circ}\coloneqq (L\cap X^\circ)/\bS^1 \subset \underline X^{\circ},
     \end{equation*}
     are cylindrical Lagrangians of $X^{\circ}$ and $\underline X^{\circ}$ respectively. 
\end{lemma}
\begin{proof}
With respect to \eqref{eq: boundaryloc}, we can describe $\partial L^{\circ}$ as follows:
\begin{equation*}
    \left(
    \begin{tikzcd}
    S^1_z \ar[r] & \partial L^\circ \ar[d] \\
    & L'
    \end{tikzcd} \right)
    \;\ \subset \;\ 
    \left(
    \begin{tikzcd}
    \mathbb S^3 \ar[r] & SN_{F/X} \ar[d] \\
    & F
    \end{tikzcd}
    \right).
    \end{equation*}
Here, $S^1$ is a Legenderian circle lying inside $S^3 \subset \mathbb D^2_{z,w}$. Thus, $L^\circ$ is cylindrical near infinity in $X^\circ$. A similar argument shows that $\underline L^\circ$ is a cylindrical Lagrangian of $\underline X^\circ$.
\end{proof}

\begin{lemma}
\label{lem: step1}
Let $L_i \subset \mu^{-1}(0)$ be exact, $\bS^1$-invariant, immersed Lagrangians which are of disk type near $F$. Then there is an isomorphism 
    \begin{equation*}
    HW_{\bS^1}^\bullet(L^\circ_1,L^\circ_2) \cong HW^\bullet (\underline{L}^\circ_1,\underline{L}^\circ_2).
\end{equation*}
\end{lemma}
\begin{proof}
    The $\bS^1$-action on each $L^\circ_i \subset X^\circ$ is free. The result follows from  Theorem \ref{prop: GoingDown}. 
\end{proof}

As a next step, we relate $HF_{\bS^1}^\bullet(L,L)$ to $HW_{\bS^1}^\bullet(L^\circ,L^\circ)$. We extend the Hamiltonians defined in \eqref{eq: HamModel} in a similar manner. Let $N$ denote a neighborhood of $F$ which is symplectomorphic to a disc bundle $DN_{F/X}$, which can be viewed as a union of  local models \eqref{eq: LocModelConicFib} along $\underline{D}$. We choose $\bS^1$-invariant Hamiltonians $H_\nu$ on $X$ satisfying:
\begin{enumerate}
    \item $H_{\nu}$ is $C^2$-small on the complement $X\setminus N$, and 
    \item on a model neighborhood $U \subset N$ \eqref{eq: LocModelConicFib}, $H_{\nu}$ is of product form $$H_\nu \vert_U = K_\nu \times H'\vert_V,$$ where $K_\nu$ is the Hamiltonian on $\C^2$ defined in \eqref{eq: HamModel} and $H'$ is a $C^2$-small Hamiltonian on $F$ chosen so that intersections $L'_1\cap \phi_{H'}^1(L_2')$ are transverse. 
\end{enumerate}
We arrange the generators of $CF^\bullet_{\bS^1}(L_1, L_2; H_\nu)$ according to whether it occurs in $X\setminus N$ or not. The former one consists of $CF^\bullet_{(X\setminus N)\sslash\bS^1}(\underline{L}^\circ_1, \underline{L}^\circ_2; \underline{H}_\nu)$-many cleanly intersecting circles on which $\bS^1$ acts freely, where the subscript means the intersections appearing in $(X\setminus N)\sslash\bS^1$. On the other hand, the latter ones are further decomposed into two parts as in the local model \eqref{eq:computeFloerCochain}: for each transverse intersection $L_1'\cap \phi_{H'}^1(L_2')$, we obtain $\nu$ cleanly intersecting free $\bS^1$-orbits with a single $\bS^1$-fixed point. This decomposes the Floer complex into three summands as follows, 
\begin{align}
\label{equ:FloerDecomp}
    CF^\bullet_{\bS^1}(L_1, L_2; H_\nu) &= C_{I}[\lambda, t] \oplus C_{II}[\lambda, t]\langle x^0, \ldots x^{\lfloor \nu \rfloor} \rangle \oplus C_{III}[\lambda]\langle p^{{\lfloor \nu \rfloor}+1}\rangle.
\end{align}
where the variable $t$ denotes the point class of degree one on $S^1$, the variable $x$ labels the wrapping number, $p^{\nu+1}$ labels the intersections inside $F$, and
\begin{equation*}
     C_I = CF^\bullet_{(X\setminus N)\sslash\bS^1}(\underline{L}^\circ_1, \underline{L}^\circ_2; \underline{H}_\nu), \quad C_{II}=C_{III} = CF^\bullet(L_1', L_2'; H').
\end{equation*}
Consider a complex 
\begin{equation*}
    CF^\bullet_{I, II}(L_1, L_2; H_\nu) \coloneqq C_I[\lambda, t]\oplus C_{II}[\lambda, t]\langle x^1, \cdots, x^{\lfloor \nu \rfloor}\rangle \subset CF^\bullet_{\bS^1}(L_1, L_2; H_\nu)
\end{equation*}
whose differential, denoted by $\mu^1_{I, II}$ is defined by the counting of pseudo-holomorphic strips between those generators which \emph{does not intersect with $D_{\bS^1} \subset X_{\bS^1}$}. Likewise, we define a continuation map 
\begin{equation*}
    c_{I, II}^{\nu, \nu'}: CF^\bullet_{I, II}(L_1, L_2; H_{\nu}) \to CF^\bullet_{I, II}(L_1, L_2; H_{\nu'}), \quad (\nu\leq\nu')
\end{equation*}
by counting solutions of a Floer continuation equation which does not intersect with $D_{\bS^1}$.
\begin{lemma}
    \label{lem: complexI,II}
    The pair $\left(CF^\bullet_{I, II}(L_1, L_2; H_\nu), \mu^1_{I, II}\right)$ is a complex and $c_{I, II}^{\nu, \nu'}$ is a cochain map. Moreover, 
    $$
    \lim_{\nu \to \infty} HF^\bullet_{I, II}(L_1, L_2; H_\nu) \cong HW^\bullet_{\bS^1}(L_1^\circ, L_2^\circ).
    $$
    Here, the limit has been taken over continuation maps $c^{\nu, \nu'}_{I, II}$.
\end{lemma}
\begin{proof}
Observe that the generators of $CF^\bullet_{I, II}$ are in one-to-one correspondence with the generators of $CF^\bullet_{\bS^1}(L_1^\circ, L_2^\circ; H_\nu^\circ)$. 

To prove the rest of the claim, we use a cobordism version of no-escape lemma \cite[Proposition 5.12]{Ton19} which adapts to our Borel spaces as well. It tells us that Floer strips and continuation solutions are confined to $X^\circ_{\bS^1}$ whenever they do not intersect with $D_{\bS^1}$. In particular, a degeneration of such strips is always happening inside $X^\circ_{\bS^1}$. This ensures that $(\mu^1_{I, II})^2=0$ and $c_{I, II}^{\nu, \nu'}$ are cochain maps. Therefore $(CF^\bullet_{I, II}(L_1, L_2; H_\nu), \mu^1_{I, II})$ is identified with $CF^\bullet_{\bS^1}(L_1^\circ, L_2^\circ; H_\nu^\circ)$, and so does $c^{\nu, \nu'}_{I, II}$ with the continuation map from $CF^\bullet_{\bS^1}(L_1^\circ, L_2^\circ; H_\nu^\circ)$ to $CF^\bullet_{\bS^1}(L_1^\circ, L_2^\circ; H_{\nu'}^\circ)$. 
\end{proof}
\begin{rem}
\label{rem:NoSubcomp}
Notice that $CF^\bullet_{I, II}(L_1, L_2; H_\nu)$ is not a subcomplex of $CF^\bullet_{\bS^1}(L_1, L_2; H_\nu)$.     
\end{rem}
We analyze Morse-Bott type spectral sequence \eqref{equ:clean_spectral_sequence} for $L_1=L_2=L$. There are three types of chords and they match up with the components of \eqref{equ:FloerDecomp}. 
\begin{enumerate}
    \item The first family is $L^\circ_{\bS^1}$ (or its slight pushoff by the Liouville flow) which corresponds $C_{I}[\lambda, t]$.
    \item For each $1\leq k \leq \nu$, we have $\partial L^\circ_{\bS^1}$-family of chord coming from the orbits wrapping $k$ times around $F$. They correspond to $C_{II}[\lambda, t] \cdot x^k$ for each $k$.
    \item The fixed locus $L' = L\cap F$ gives the last family $L' \times B\bS^1$, which corresponds to $C_{III}[\lambda]\cdot p^{{\lfloor \nu \rfloor}+1}$. 
\end{enumerate}
We obtain two spectral sequences \eqref{equ:clean_spectral_sequence} computing $HF^\bullet_{\bS^1}(L, L; H_\nu)$ and $HF^\bullet_{\bS^1}(L^\circ, L^\circ; H_\nu^\circ)$ respectively, which we denote by $E(\nu)$ and $E^\circ(\nu)$. Notice that $X$ and $X^\circ$ are both exact and the primitive $f_L$ of $\lambda$ is locally constant near $\partial X^\circ$. Therefore the action of chords shares the same formal expression
    $$\int_{0}^{1} -x^*\lambda+ K(x(t)) dt.$$
But notice that we use different Liouville forms, namely $\lambda^\circ$ for $X^\circ$ and $\lambda$ for $X$.

The first page of $E^\circ(\nu)$ is as follows. 
\begin{equation*}
\label{equ: computeEcirc}
    E^\circ(\nu)^{p,q}_1 = \left \{\begin{array}{cc}
        H^{q}_{\bS^1}(L^\circ) \cong H^q(\underline{L}^\circ) & p=0\\
        H^{p+q+2p}_{\bS^1}(\partial L^\circ) \cong H^{p+q+2p}(\partial \underline{L}^\circ) & -\lfloor \nu \rfloor \leq p\leq -1\\
        0 & \text{otherwise}. 
    \end{array}
    \right.
    \;\; \Longrightarrow HF^\bullet_{\bS^1}(L^\circ, L^\circ; H_\nu^\circ).
\end{equation*}
The contribution $2p$ is the Conley-Zehnder index of the connected component corresponding to $C_{II}[\lambda, t]\cdot x^p$. We only use the generators of the first and second type and do not use the third type. The action of the first type is close to zero. For the generators of the second type, we use standard analysis \cite{abouzaid2010open}. The chords at $\{r = r_0\}$ is given by
\begin{equation}
\label{equ:actionXcirc}
    -r_0 \cdot \frac{\partial K}{\partial r} (r_0) + K(r_0), \quad \text{whose derivative is} \quad -r_0 \cdot \frac{\partial^2 K}{\partial r^2} (r_0) <0,
\end{equation}
since $K$ is strictly convex along the chosen Liouville collar. It is based on the fact that the Hamiltonian chords near the boundary are the Reeb orbits of the contact form. Therefore the action decreases as the power $x^k$ increases.

On the contrary, the first page of $E(\nu)$ is 
\begin{equation*}
\label{equ: computeE}
    E(\nu)^{p,q}_1 = \left \{\begin{array}{cc}
        H^{q}_{\bS^1}(L^\circ) \cong H^q(\underline{L}^\circ) & p=0\\
        H^{p+q-2p}_{\bS^1}(\partial L^\circ) \cong H^{p+q-2p}(\partial \underline{L}^\circ) & 1 \leq p\leq \lfloor \nu \rfloor \\
        H^{p+q-2(p+1)}_{\bS^1}(L') \cong H^{p+q-2(p+1)}(L')[\lambda]& p=\lfloor \nu \rfloor +1\\
        0 & \text{otherwise}. 
    \end{array}
    \right.
    \;\;\Longrightarrow HF^\bullet_{\bS^1}(L, L; H_{\nu}).
\end{equation*}
In particular, the action is increasing as the power $x^k$ grows. The reason is the following;  the action of the first type is close to zero as before. But the standard analysis \eqref{equ:actionXcirc} is no longer valid for $X$. Instead, observe that these orbits are the boundary of the diagonal disk $\Delta_{\mathbb D^2} \subset \mathbb D^4 \subset \mathbb C^2$ of the fiber of the local neighborhood \eqref{eq: LocModelConicFib} of $F$. Therefore, we can replace the integral $-\int_{S^1}x^*\lambda$ to the symplectic area of the disk, counted with multiplicity. The result is 
\begin{equation}
\label{equ:actionX}
    (1-r_0) \cdot \frac{\partial K}{\partial r} (r_0) + K(r_0), \quad \text{whose derivative is} \quad (1-r_0) \cdot \frac{\partial^2 K}{\partial r^2} (r_0) >0
\end{equation}
because $K$ is strictly convex. As a corollary, we obtain the degeneration of $E^\circ(\nu)$. 
\begin{lemma}
    \label{lem:degenWrapped}
    The spectral sequence $E^\circ(\nu)$ degenerates at the first page. Therefore, 
    \begin{equation*}
        \label{equ:HFXcirc}
        \mathrm{gr}(HF^\bullet_{\bS^1}(L^\circ, L^\circ; H_\nu^\circ)) \cong H^\bullet_{\bS^1}(L^\circ) \oplus  \bigoplus_{p=1}^{\lfloor \nu \rfloor} H^{\bullet-2p}_{\bS^1}(\partial L^\circ)
    \end{equation*}
\end{lemma}
\begin{proof}
    If not, then there are non-trivial pseudo-holomorphic curves inside $X^\circ_{\bS^1}$ that increase the action. By Lemma \ref{lem: complexI,II}, it should also be visible inside $X_{\bS^1}$, but now they decrease the action. This is a contradiction. 
\end{proof}

We will prove a similar degeneration of $E(\nu)$. We want to point out that the specific choice of $H_\nu$ does not matter in the computation of $HF^\bullet_{\bS^1}(L, L)$. We start by looking at an easy case when $\nu$ is small. The following is a consequence of classical results. 

\begin{lemma}
    \label{lem:degenFloerI}
    Choose $\nu$ to be small enough so that
    \begin{enumerate}
        \item $H_\nu$ is $C^2$-small, and
        \item $E(\nu)^{p,q}_1 =0$ for $p\geq 2$, i.e., $E(\nu)_1$ has only two non-zero columns.
    \end{enumerate}
    Then the spectral sequence $E(\nu)$ degenerates at the first page. Therefore, 
    \begin{equation*}
        \label{equ:HFX}
        \mathrm{gr}(HF^\bullet_{\bS^1}(L, L; H_\nu)) \cong H^\bullet_{\bS^1}(L^\circ) \oplus H^{\bullet-2}_{\bS^1}(L').
    \end{equation*}
\end{lemma}
\begin{proof}
Via PSS isomorphism (Theorem \ref{thm:G-equivariant_PSS_isomorphism}) $HF^\bullet_{\bS^1}(L, L)$ is isomorphic to the equivariant cohomology $H^\bullet_{\bS^1}(L)$, which can be computed via Thom-Gysin exact sequence. 
\begin{equation*}
    \cdots \to H^{\bullet-2}_{\bS^1}(L') \xrightarrow{i_*} H^\bullet_{\bS^1}(L) \xrightarrow{\mathrm{res}}  H^\bullet_{\bS^1}(L^\circ) \xrightarrow{\delta} H^{\bullet-1} _{\bS^1}(L') \to \cdots.
\end{equation*}
Here, $i: L' \hookrightarrow L$ is the natural inclusion. We are using the fact that $L$ is of disk type near $F$ so that $N_{L'/L}$ is an equivariant complex line bundle. The differential $d_1$ of $E(\nu)_1$ is a Floer analog of the connecting map $\delta$. 

It is well-known that the exact sequence splits in this case. The reason is that the composition $i^*i_*=e_{\bS^1}(N_{L'/L})$ is injective because the equivariant class is non-zero divisor. In particular, the connecting map $\delta$ always vanishes. By degreewise dimension counting, it also implies the vanishing of $d_1$.
\end{proof}

We have encountered the main result of this section. Lemma \ref{lem: complexI,II} enables us to view $HF^\bullet_{\bS^1}(L, L)$ as a deformation of $HW^\bullet_{\bS^1}(L^\circ, L^\circ)$. The \eqref{equ:HFXcirc} and \eqref{equ:HFX} indicate that the deformation is trivial after exchanging the variables $x$ and $\lambda$. This fits well with the original formulation \cite{lekili2023equivariant}, except for one obstacle. A deformation scheme of the relative Fukaya category does not apply to our situation because our Lagrangian $L$ meets the divisor $D$. 

While keeping the same strategy in mind, we show something slightly different. We will prove that the cohomology $HF^\bullet_{\bS^1}(L, L; H_\nu)$ is getting closer and closer to $HW^\bullet(L^\circ, L^\circ)$ as a filtered vector space as $\nu \to \infty$. 

Consider a continuation map denoted by 
$$
c^{\nu, \nu'}: CF^\bullet_{\bS^1}(L, L; H_\nu) \to CF^\bullet_{\bS^1}(L, L; H_{\nu'}), \quad (\nu \leq \nu').
$$
We know that $c^{\nu, \nu'}$ induces an isomorphism on cohomologies. Since $c^{\nu, \nu'}$ preserves action filtration, it induces a map between spectral sequences. By abuse of notation we still denote them by $c^{\nu, \nu'}$. We approximate the low-energy component of this map. 
\begin{lemma}
    \label{lem:ContiSpectral}
    Suppose ${\lfloor \nu'\rfloor} = {\lfloor \nu \rfloor}+1$. Let us write
    \begin{equation*}
    (c^{\nu, \nu'})^{p,q} = \sum (c^{\nu, \nu'})_k^{p,q}: E(\nu)_1^{p,q} \to \bigoplus_{k=0}^{\lfloor \nu \rfloor+1 -p }E(\nu')_1^{p+k,q-k}.
\end{equation*}
Then we have the following. 
\begin{enumerate}
    \item For $0\leq p \leq \lfloor \nu \rfloor$,  $(c^{\nu, \nu'})_0^{p,q}$ are isomorphisms.
    \item for $p = \lfloor \nu \rfloor +1$, then the sum of 
    \begin{align*}
    (c^{\nu, \nu'})^{p,q}_0 \vert_{H^\bullet(L')} &: H^{q-\lfloor \nu \rfloor}(L') \to H^{q-\lfloor \nu \rfloor}_{\bS^1}(\partial L^\circ) \cong H^{q-\lfloor \nu \rfloor}(L'), \; \;\text{and}\\
    (c^{\nu, \nu'})^{p,q}_1 \vert_{H^\bullet(L')[\lambda]\cdot \lambda} &: H^{q-\lfloor \nu \rfloor}(L')[\lambda]\cdot \lambda \to H^{q-\lfloor \nu \rfloor-2}(L')[\lambda]
\end{align*}
    are isomorphisms of vector spaces. 
\end{enumerate}

\end{lemma}
\begin{proof}
    For (1), observe that $(c^{\nu, \nu'})_0^{p,q}$ counts Floer-theoretic continuation solutions confined in a small neighborhood of Morse-Bott component. They induce Morse-theoretic continuation map on each column, which implies that they are isomorphisms.

    For (2), let us write a representative of a class of the last column as follows. 
    \begin{equation*}
        \phi \lambda^k p^{\nu+1} \in C_{III}[\lambda] \cdot p^{\nu+1} \simeq C^\bullet(L')[\lambda], \quad d^\mathrm{Morse} \phi = 0.
    \end{equation*}
    Recall \eqref{eq: LocModelConicFib}, which states that locally $L\subset X$ is a Lagrangian disk fibration over $L'\subset F$ compatible with the $\mathbb D^4$-fibration $DN_{F/X} \to F$. Each fiber is a Lagrangian disk $\mathbb D_{\Re} \subset \mathbb D^4$ inside a standard $4$-disk.  In this situation, the guiding example \eqref{eq:computeFloerCoh} at the beginning computes the low-energy component.\footnote{Technically, the almost complex structures $J$ on $N_{F/X}$ and $J_F$ on $F$ need to be compatible with the bundle projection $N_{F/X}\to F$. This can be done by choosing a Hermitian connection on $N_{F/X}$ and requiring that $J$ preserves the horizontal and vertical subspaces.} In particular, we obtain
    \begin{equation}
    \label{equ:approximate_continuation}
        c^{\nu, \nu'}(\phi\lambda^kp^{{\lfloor \nu \rfloor}+1}) = \phi \lambda^k x^{{\lfloor \nu \rfloor}+1}+ \mathcal O(\lambda^{k+1}).
    \end{equation}
    In particular, $(c^{\nu, \nu'})_0$ restricted to $H^\bullet(L') \subset H^\bullet(L')[\lambda]$ surjects to $H^\bullet_{\bS^1}(\partial L^\circ) \cong H^\bullet(L')$. Therefore, it is an isomorphism. 

    If $k>0$, then we have $d_{\bS^1}^\mathrm{Morse}(\phi \lambda^{k-1} x^{\lfloor \nu \rfloor+1} t) = \phi  \lambda^{k} x^{\lfloor \nu \rfloor+1}$. Therefore, the element $\phi\lambda^k x^{\lfloor \nu \rfloor+1} = \phi\lambda^k x^{\lfloor \nu' \rfloor}$ becomes zero in $H^\bullet_{\bS^1}(\partial L^\circ)$. In this case, we use the computation of the differential \eqref{eq:computeFloerDiff} in the guiding example to conclude that 
    \begin{equation*}
    \label{equ:approximate_differential}
        \mu^1(\phi \lambda^{k-1} x^{\lfloor \nu' \rfloor}t) =  \phi \lambda^{k-1} p^{\lfloor \nu '\rfloor+1} \pm \phi \lambda^{k+1} x^{\lfloor \nu' \rfloor} + \mathcal O (\lambda^{k+2}).
    \end{equation*}
    The middle term is the equivariant Morse differential, and the first and the last term count extra strips from Floer theory. Combined with \eqref{equ:approximate_continuation}, we conclude that 
    \begin{equation*}
        (c^{\nu, \nu'})^{p,q}_1([\phi\lambda^kp^{{\lfloor \nu \rfloor}+1}]) = \pm[\phi \lambda^{k-1} p^{\lfloor \nu '\rfloor+1}] + \mathcal O(\lambda^k) \in H^\bullet(L')[\lambda] \quad \text{for $k>0$.}
    \end{equation*}
    This proves the last assertion of Lemma. 
\end{proof}
\begin{cor}
    \label{cor:degenFloerII}
The spectral sequence $E(\nu)$ degenerates at the first page. In particular, 
\begin{equation*}
    \mathrm{gr}(HF^\bullet_{\bS^1}(L, L; H_\nu)) = H^\bullet_{\bS^1}(L^\circ) \oplus \bigoplus_{k=1}^{\lfloor \nu \rfloor} H^\bullet_{\bS^1}(\partial L^\circ)\cdot x^k  \oplus H^\bullet_{\bS^1}(L')[\lambda]\cdot p^{\nu+1}
\end{equation*}
\end{cor}
\begin{proof}
Lemma \ref{lem:ContiSpectral} computes a leading term of the induced map of $c^{\nu, \nu'}$ on the first pages. Let $\nu$ sufficiently small as in Lemma \ref{lem:degenFloerI}, and $\nu < \nu'$. Then the leading term of 
\begin{equation*}
    c^{\nu, \nu'}: H^\bullet_{\bS^1}(L^\circ) \oplus H^\bullet_{\bS^1}(L')[\lambda] \to H^\bullet_{\bS^1}(L^\circ) \oplus \bigoplus_{k=1}^{\lfloor \nu \rfloor} H^\bullet_{\bS^1}(\partial L^\circ)\cdot x^k  \oplus H^\bullet_{\bS^1}(L')[\lambda]\cdot p^{\nu+1}
\end{equation*}
exchanges $x$ and $\lambda$ up to the power $k = \lfloor \nu \rfloor$, and then replaces $\lambda^{l+ k+1}$ to $\lambda^l p^{\nu'+1}$. This implies $c^{\nu, \nu'}$ induces an isomorphism on the first pages. Since $E(\nu)$ degenerates for small $\nu$ by Lemma \ref{lem:degenFloerI}, the assertion follows.  
\end{proof}

\begin{prop}
\label{prop: conicsingred}
    Let $L \subset \mu^{-1}(0)$ be an exact, $\bS^1$-invariant Lagrangian which are of disk type near $F$. Then we have an isomorphism of filtered vector spaces
    \begin{equation*}
        \lim_{\nu\to \infty}HF_{\bS^1}^\bullet(L,L; H_{\nu}) \cong HW_{\bS^1}^\bullet(L,L) \cong HW^\bullet(\underline L, \underline L).
    \end{equation*}
\end{prop}
\begin{proof}
As a consequeces of Corollary \ref{cor:degenFloerII}, combined with Lemma \ref{lem: complexI,II}, we obtain
\begin{align*}
     \lim_{\nu\to \infty}\mathrm{gr}(HF^\bullet_{\bS^1}(L, L; H_\nu)) &\cong  \lim_{\nu\to \infty} \mathrm{gr}(HF^\bullet_{I, II}(L, L; H_\nu)) \cong \mathrm{gr}(HW^\bullet_{\bS^1}(L^\circ, L^\circ; H_\nu^\circ)).
\end{align*}
\end{proof}

Considering a colimit over positive isotopies along $\partial X$, we obtain the following corollary. 
\begin{cor} 
    \label{cor: mainthmconicfib}
    Let $L_i \subset \mu^{-1}(0)$ be exact, $\bS^1$-invariant Lagrangians which are of disk type near $F$.  Then we have 
    \begin{equation*}
        HW_{\bS^1}^\bullet (L,L) \cong HW^\bullet (\underline{L}^\circ,\underline{L}^\circ).
    \end{equation*}
\end{cor}
The same proof of Proposition \ref{prop: conicsingred} applies when the singularities of a given $\bS^1$-action is nodal. Therefore, we have the following.
\begin{thm}
\label{thm: S1case}
    Let $\bS^1\circlearrowright(X,\lambda)$ be a Liouville domain with a Hamiltonian $\bS^1$-action satisfying the following assumptions: 
    \begin{enumerate}
        \item The isotropy group for each $p \in \mu^{-1}(0)$ is either trivial or the whole $\bS^1$. We denote the fixed locus by $F = (\mu^{-1}(0))^{\bS^1}$.
        \item Each connected component of $F$ is a codimension-four submanifold, admitting an equivariant symplectic neighborhood $U \subset X$ of the form  \eqref{eq: LocModelConicFib}, that is, modeled on a conic fibration. 
    \end{enumerate} 
    Then $X\setminus F$ can be equipped with an $\bS^1$-equivariant convex Liouville structure near the puncture $F$, which descends to the regular locus of its quotient
    $$\underline{X}^\circ \coloneqq  (\mu^{-1}\{0\} - F)/\bS^1.$$
    For any $\bS^1$-invariant Lagrangians $L \subset \mu^{-1}(0)$ which is of disk type near $F$, there is an isomorphism of filtered vector spaces
    $$
    \lim_{\nu \to \infty}HF^\bullet_{\bS^1}(L, L; H_\nu) \simeq HW^\bullet(\underline{L}^\circ, \underline{L}^\circ),
    $$
    where $\underline{L}^\circ \coloneqq  (L-F)/\bS^1$.
\end{thm}

\subsection{Iterated conic fibration}

In this subsection, we apply Corollary \ref{cor: mainthmconicfib} inductively to extend our results to an iterated conic fibration. Let $\underline X$ be a fixed smooth quasi-projective variety equipped with  sections $\{s_i\in \Gamma(Y, \mathcal L_i^{\otimes 2})\}_{i=1}^n$ of ample line bundles. We assume that $\{\underline{D}_i = Z(s_i=0)\subset Y\}_{i=1}^n $ form a simple normal crossing divisor with each $\underline{D}_i$ being smooth. In this setup we iterate conic fibration construction to obtain: 
\begin{align*}
    X_k \coloneqq  Z(z_iw_i = s_i) &\subset \mathrm{Tot}_Y(\oplus_{i=1}^k \mathcal L_i^{\oplus 2}),\\
    X\coloneqq X_n \xrightarrow{\pi_n}X_{n-1} \xrightarrow{\pi_{n-1}} &\cdots \xrightarrow{\pi_2}X_1 \xrightarrow{\pi_1} X_0\eqqcolon \underline X,
\end{align*}
where \begin{align}
        \pi_k(\mathbf x, z_1, w_i, \ldots, z_{k-1}, w_{k-1}, z_k, w_k)&\coloneqq(\mathbf x, z_1, w_i, \ldots, z_{k-1}, w_{k-1}). 
\end{align}
The total space $X$ is equipped with a Hamiltonian torus action
\begin{equation*}
    (\bS^1)^n \circlearrowright X, \quad (\lambda_i)_{i=1}^n \cdot (\mathbf x, w_i, z_i)_{i=1}^n =(\mathbf x, \lambda_i w_i, \lambda_i^{-1}z_i)_{i=1}^n.
\end{equation*}
Let us denote the circle action acting on the $k$-th coordinate by $\bS^1_k$. Then each $X_{k-1}$ is a singular reduction $X_{k} \sslash \bS^1_k$ with degeneracy locus $\pi_{k}^{-1}\circ \cdots \pi_1^{-1}(\underline{D}_k)$. Define
\begin{equation*}
    X_{k}^\circ \coloneqq  X_k - \bigcup_{i=k+1}^n D_{k, i}, \quad D_{k, i}\coloneqq  \pi_{k}^{-1}\circ \cdots \pi_1^{-1}(\underline{D}_i).
\end{equation*}
The reduction $X_{k}^\circ \sslash \bS^1_k$ is singular along $D_{k, k}$, and its non-singular locus is $X^\circ_k$. 

To apply previous results, we impose similar assumptions on $\underline{X}$ as before; we assume that there exists a smooth projective compactification $\underline{Y}\supset\underline{X}$ of $\underline{X}$ by a SNC divisor $\underline{K}$, and that each $\underline{D}_i$ extends to a smooth ample divisor of $\underline{Y}$. Since $X_{k}^\circ$ is obtained as a divisor complement, the same holds for  $X_{k-1}^\circ$. Both spaces are equipped with the Liouville structure described in \eqref{defn: omegalambda}. We choose Lagrangians compatibly with this structure: we say that an $(\bS^1)^n$-invariant Lagrangian is of disk type if it is of disk type near each $D_{j,n}$ for each $j$. We then inductively define
\begin{equation*}
    \underline L_n^\circ\coloneqq L\cap X^\circ_n, \quad \underline L_k^\circ \coloneqq L_{k+1}/(\bS^1)_{k+1} \cap X^\circ_k. 
\end{equation*}

By inductively applying Proposition \ref{prop: conicsingred}, we have the following result. 
\begin{thm} Let $L \subset \mu^{-1}(0)$ be an exact, $(\bS^1)^n$-invariant Lagrangian that is of disk type near the singularities. Then, for each $k=1,\cdots, n$, there is an isomorphism of filtered vector spaces
    \begin{equation*}
        HF^\bullet_{(\bS^1)^n}(L, L) \simeq HF_{(\bS^1)^{n-k}}^\bullet(\underline L^\circ_{k}, \underline L^\circ_{k}),
    \end{equation*}
    which is a cohomology version of \cite[Conjecture A]{lekili2023equivariant}.
\end{thm}

\bibliographystyle{alpha}
\bibliography{mybib}

\end{document}